\documentclass[12pt]{amsart}

\usepackage[utf8]{inputenc}
\usepackage{ae,aecompl,aeguill}	
\usepackage{mathtools}
\usepackage{amsmath}
\usepackage{amssymb}
\usepackage{mathrsfs}
\usepackage{amscd}
\usepackage{bbm}
\usepackage{enumerate}
\usepackage{multirow}
\usepackage{hyperref}
\usepackage{tikz}
\usetikzlibrary{arrows}
\usepackage{comment}
\usepackage{cite}

\newcommand{\be}{\begin{equation}}
\newcommand{\ee}{\end{equation}}
\newcommand{\iden}{\mathbbm{1}}

\newcommand{\ts}{\textstyle}
\newcommand{\ds}{\displaystyle}

\newcommand{\Monoid}{M}

\newcommand{\supp}{\mathrm{supp}}
\newcommand{\su}{\mathop{\boldsymbol{s}}}
\newcommand{\RR}{{\mathscr R}}

\newcommand{\LL}{{\mathscr L}}

\newcommand{\lex}{{\operatorname{lex}}}
\newcommand{\source}{\sigma}
\newcommand{\toomgens}{\partial}
\newcommand{\B}{\mathcal{B}}
\newcommand{\Z}{\mathbb{Z}}
\newcommand{\N}{\mathbb{N}}

\newcommand{\Poset}{\mathcal{P}}
\renewcommand{\supp}{\mathop{\mathrm{supp}}}
\newcommand{\LeftIdeal}{\mathop{\Upsilon}}

\newcommand{\inv}{^{-1}}

\newcommand{\wh}{\widehat}

\newcommand{\LO}[1]{\ell^1(#1)}

\newcommand{\suchthat}{\mid}

\newcommand{\fhb}{free tree monoid}
\newcommand{\FHB}{\operatorname{FT}}

\newtheorem{theorem}{Theorem}[section]
\newtheorem{proposition}[theorem]{Proposition}
\newtheorem{lemma}[theorem]{Lemma}
\newtheorem{corollary}[theorem]{Corollary}
\newtheorem{conjecture}[theorem]{Conjecture}
{\theoremstyle{definition}
\newtheorem{definition}[theorem]{Definition}}
{\theoremstyle{remark}
\newtheorem{remark}[theorem]{Remark}}
{\theoremstyle{remark}
\newtheorem{example}[theorem]{Example}
}

\numberwithin{equation}{section}

\begin{document}

\title[Markov chains and $\RR$-trivial monoids]{Markov chains, $\RR$-trivial monoids and representation theory}

\author[A. Ayyer]{Arvind Ayyer}
\address{Department of Mathematics, Department of Mathematics, Indian Institute of Science, Bangalore - 560012, India.}
\email{arvind@math.iisc.ernet.in}

\author[A. Schilling]{Anne Schilling}
\address{Department of Mathematics, UC Davis, One Shields Ave., Davis, CA 95616-8633, U.S.A.}
\email{anne@math.ucdavis.edu}

\author[B. Steinberg]{Benjamin Steinberg}
\address{Department of Mathematics, City College of New York, Convent Avenue at 138th Street,
New York, NY 10031, U.S.A.}
\email{bsteinberg@ccny.cuny.edu}

\author[N. M. Thi\'ery]{Nicolas M.~Thi\'ery}
\address{Univ Paris-Sud, Laboratoire de Math\'ematiques d'Orsay,
  Orsay, F-91405; CNRS, Orsay, F-91405, France}
\email{Nicolas.Thiery@u-psud.fr}

\dedicatory{Dedicated to Stuart Margolis on the occasion of his sixtieth birthday}

\begin{abstract}
We develop a general theory of Markov chains realizable as random walks on $\RR$-trivial monoids.
It provides explicit and simple formulas for the eigenvalues of the transition matrix, for multiplicities of the eigenvalues
via M\"obius inversion along a lattice, a condition for diagonalizability of the transition matrix and some techniques for bounding the mixing time.
In addition, we discuss several examples, such as Toom-Tsetlin models, an exchange walk for finite Coxeter groups,
as well as examples previously studied by the authors, such as nonabelian sandpile models and the promotion
Markov chain on posets. Many of these examples can be viewed as random walks on quotients of free tree monoids,
a new class of monoids whose combinatorics we develop.
\end{abstract}

\date{\today}

\maketitle
\tableofcontents

\section{Introduction}

A finite state Markov chain is a stochastic dynamical system where
the current state only depends on its history via the previous state.
The only data needed to define it is a finite set $\Omega$
and a transition matrix $T\colon \Omega\times \Omega\to \mathbb R$
that describes the probability to transition from one state to the next one
at each step. The matrix $T$ is required to be non-negative, with each column summing to $1$ (i.e., $T$ should be a
\emph{column stochastic} matrix).

As highlighted in~\cite{DiaconisLectures,silberstein}, the representation theory of
finite groups can be a powerful technique for analyzing the Markov chain when $T$
is bistochastic (meaning that each row and each column sums to $1$). The starting point
is to decompose the transition matrix $T$ as a convex combination $T=
\sum x_a \sigma_a$ of permutation matrices $\sigma_a$, and to consider
the finite permutation subgroup $G=\langle \sigma_a\rangle_a$ of the
symmetric group $\mathcal S_{\Omega}$ generated by the permutations
$\sigma_a$. At this point, the Markov chain can be interpreted as
arising from a random walk on $G$ or on cosets thereof.
By the classical Birkhoff-von Neumann theorem there always exists such a convex
decomposition of the transition matrix $T$ (see e.g.~\cite[Example~0.12]{Ziegler}).

The representation theory of $G$ allows one to decompose
the space $\mathbb C\Omega$ into a direct sum of irreducible representations of $G$, which are in particular
invariant subspaces for the operators $\sigma_a$ and therefore for the transition matrix $T$. This has the effect of
turning $T$ into a block diagonal matrix, where each
block can be analyzed separately using that the subspace is an
irreducible representation of $G$. Furthermore, \emph{character
  theory} can be employed to recover all irreducible constituents and
their multiplicities without explicitly block diagonalizing the
matrix.

Whether this approach is practical or not depends on how nice the representation theory of $G$
is, how fine the decomposition of $\mathbb C\Omega$ into irreducible
representations is, and on properties of $T$ itself. For instance,
when $G$ is abelian, each irreducible subrepresentation is
one-dimensional; therefore, the decomposition of $\mathbb C\Omega$ into
irreducibles diagonalizes $T$.  More generally, when the weights $x_a$ are
constant along conjugacy classes of $G$, Schur's lemma implies that
the decomposition of $\mathbb C\Omega$ into irreducibles diagonalizes
$T$, and one can completely analyze the Markov chain via representation
theory~\cite{DiaconisLectures,silberstein}.  Another particularly nice
case is when $G$ is the full symmetric group, giving
connections between card-shuffling Markov chains and symmetric
functions~\cite{DiaconisLectures,DiaconisSha.1981,DiaconisShah.1986,DiaconisRam.2012}.

When the transition matrix is only column stochastic, one can still
use an easy variant of Birkhoff-von Neumann theorem to decompose $T$
as a convex combination $T = \sum x_a \sigma_a$, where the $\sigma_a$
are operators on $\Omega$ which we sometimes refer to informally as the generators of the Markov chain. These
operators are not necessarily invertible anymore and in general they generate a monoid $M=\langle
\sigma_a\rangle_a$ instead of a group. Nevertheless one can still try to use its representation theory.

The representation theory of finite monoids is much less well
understood than that of groups (see~\cite[Chapter~5]{CP}
and~\cite{McAlisterCharacter,RhodesZalc}), but there has been much
recent progress, see for instance~\cite{Putcharep5,Putcharep3,myirreps,denton_hivert_schilling_thiery.2010,
hivert_schilling_thiery.2013,Saliola,MargolisSteinbergQuiver,rrbg,mobius1,mobius2}. The analysis
of random walks on hyperplane arrangements~\cite{BidigareThesis, BHR,DiaconisBrown},
and in particular the Tsetlin library, provided motivation for Brown to develop a successful analysis
of Markov chains via the representation theory of \emph{left regular bands}, which are semigroups satisfying a
certain ``deletion property''~\cite{BrownLRB} (for details see Section~\ref{section.semigroups}).
This theory has been further developed and applied
in~\cite{DiaconisBrown,DiaconisBrown2,bjorner1,bjorner2,DiaconisAthan,GrahamChung,saliolaeigen}.

In his 1998 ICM address, Diaconis~\cite{DiaconisICM} asked for the ultimate generalization of these monoid techniques.
In this paper, we make progress toward answering this question by generalizing Brown's theory of Markov chains on
left regular bands~\cite{BrownLRB} to Markov chains on \emph{$\RR$-trivial monoids}. This large generalization
potentially finds applications in combinatorics, statistical physics, and computer science. We remark that left regular bands
are precisely those $\RR$-trivial monoids whose elements are all idempotent.
From the point of view of combinatorics, natural Markov chains on objects such as permutations
(i.e., the Tsetlin library)~\cite{hendricks1}, hyperplane arrangements~\cite{BHR} and linear
extensions~\cite{ayyer_klee_schilling.2012} are of intrinsic interest. As in the case of left regular bands,
combinatorial sequences such as derangement numbers arise as the multiplicities of eigenvalues of the
transition matrices of these Markov chains which deserves some
uniform explanation.

Statistical physicists and computer scientists model real-life phenomena probabilistically as Markov chains
and are interested in both the stationary distribution of the chain (given by the eigenvector of the transition matrix
with eigenvalue $1$) and the time to approach stationarity (which for reversible chains is controlled by the
second-largest eigenvalue, or spectral gap). Recently many interesting Markov chains have emerged
which fit into the $\RR$-trivial monoid theory~\cite{ayyer_strehl_2010,ayyer_2011,ayyer_strehl_2013, ayyer_klee_schilling.2012,
ayyer_schilling_steinberg_thiery.sandpile.2013}.

In Sections~\ref{preliminaries}--\ref{rtrivial-markov-chain} we develop the general theory of Markov chains which
are random walks on $\RR$-trivial monoids and describe how the unified approach of \emph{$\RR$-trivial monoids}
provides techniques for the  calculation of these quantities.

Let us briefly summarize how the representation theory of $\RR$-trivial monoids compares
to that of groups. First of all, we lose semisimplicity (or complete reducibility) of representations, which means
that the transition matrix can no longer be put in a block diagonal form,
but rather in block triangular form. On the other hand, the irreducible representations
are one-dimensional, which means that the transition matrix can
actually be transformed into upper triangular form. For example, this makes it
easy to recover the eigenvalues using character theory, and to determine the irreducible constituents via M\"obius
inversion. In fact, the eigenvalues take a particularly nice form, given as a sum of a subset of the probabilities
$x_a$ assigned to each generator~\cite{mobius1, mobius2}.  Note that, in the group case, it is non-trivial to compute eigenvalues of random
walks unless the probability measure is constant on conjugacy classes (e.g., for abelian groups).  For instance,
it is easier to compute the eigenvalues for the top-to-random shuffle as a left regular band walk~\cite{BHR} than
as a symmetric group walk.


As this is a long paper, it seems worthwhile to informally describe
some of the Markov chains that we analyze in
Sections~\ref{section.free tree monoid}--\ref{section.coxeter}, and
have analyzed in previous papers, using $\RR$-trivial monoid techniques.
The reader should also consult Brown~\cite{BrownLRB} for numerous examples using the particular case of left regular
bands.  See also~\cite{GrahamChung,bjorner1,bjorner2,DiaconisAthan} for further left regular band random walks.

\smallskip
\paragraph{\textbf{The Toom--Tsetlin model}}
The classical Tsetlin library Markov chain \cite{hendricks1,Fill1,Fill2,BHR} consists of a shelf of books with an imposed
self-organizing system for the books.  Each time a book is removed from the shelf, it is placed back to
the front of the shelf. This way, eventually the most commonly used books will be toward the front of the shelf
while the least commonly used books will be toward the back.  The Tsetlin library is one of the first chains to be analyzed from the
$\RR$-trivial monoid point-of-view (actually from the left regular band point of view)~\cite{BHR,BrownLRB,DiaconisBrown}.
Using these tools one can explicitly compute the eigenvalues (which are the probabilities of picking a book from a
given subset of the books) and their multiplicities (which are derangement numbers), a bound on the mixing time,
and an explicit formula for the stationary distribution.

In this paper, we consider a generalization called the \emph{Toom--Tsetlin model}.  There are two versions, but we discuss
here only the first one  and refer the reader to Section~\ref{section.toom} for the second variant and details.
In this model one has $n_i\geq 1$ copies of book $b_i$ on the shelf.  When the $j^{th}$ copy of $b_i$ is removed
from the shelf, it is replaced immediately after the $(j-1)^{st}$ copy of $b_i$ (where if $j=1$, it is
simply placed at the front of the shelf). The Tsetlin library is the special case where there is only one
copy of each book. For the Toom--Tsetlin Markov chain we explicitly compute the eigenvalues (which again are probabilities
of choosing a book from a certain subset of books) and their multiplicities (which are derangement numbers
for words, or multipermutations). See Theorem~\ref{theorem.toom.interval.eigenvalues}.

\smallskip
\paragraph{\textbf{The landslide sandpile model}}
The abelian sandpile model~\cite{dhar1990,dhar1999} has proved influential in understanding the phenomenon
of self-organized criticality~\cite{btw1987}. The model can be thought of as a discrete-time Markov chain. It is
defined on any finite  directed graph with a global sink
(a sink is a vertex with out-degree zero; it is furthermore a global sink
if there is a directed path from any vertex to it). The abelian sandpile model works as follows.
The state space of the system is \[\Omega = \{ (t_v)_{v\in V} \mid 0\le t_v \le \mathrm{outdeg}(v)\},\]
where $V$ is the vertex set of the underlying graph and $\mathrm{outdeg}(v)$ is the out-degree of the vertex $v$.
The variable $t_v$ is a nonnegative integer which denotes the number of grains at vertex $v$. Notice in particular, that
sinks can carry no grains of sand. Moreover, whenever a grain of sand enters a sink it is considered removed from the system.

At every time step, a process of toppling and stabilization occurs. Suppose that one is in configuration $(t_v)_{v\in V}\in \Omega$.
One randomly deposits a grain of sand at one of the vertices $w$. If the total number of grains at $w$ after adding this grain is
below its out-degree, then we are done with this step;  if on the other hand the total number exceeds the out-degree of $w$, then
$w$ {\em topples}, sending one grain along each of its outgoing edges to the edge's other endpoint. This may then force some
of these endpoints to topple.   Because there is a global sink, after some sequence of topples one will reach a valid configuration
in $\Omega$ (this process is called stabilization in the literature).  The resulting configuration turns out to be independent of the
order in which the topples are performed. Let $\theta_w\colon \Omega\to \Omega$ be the operator of adding a grain of sand
at $w$ and then performing topples until stabilization occurs.

It is not completely obvious, but nonetheless true, that the operators $\theta_v$ with $v\in V$ commute and hence generate a
finite commutative monoid $M$. The minimal ideal of this commutative monoid is an abelian group $A$ which acts freely and
transitively on the recurrent states of the abelian sandpile Markov chain, which are precisely the fixed points of the identity $e$
of $A$ on $\Omega$.  Moreover, the operators $\theta_v$ act as permutations on the recurrent states since $\theta_v$ acts as
$\theta_ve$.  The group $A$ is called the sandpile group in the literature.

We study the following variant of the abelian sandpile model, called the landslide sandpile model, which is
\emph{nonabelian} in the sense that the generators of the model do not commute. The model is defined on a
\emph{directed} tree or arborescence and we analyzed this model
using monoid theoretic methods in~\cite{ayyer_schilling_steinberg_thiery.sandpile.2013}.

One has a directed rooted tree with all edges oriented toward the root.  Each vertex $v$ (hereafter called a site)
is assigned a threshold $T_v$, which is the number of grains of sand it can hold, and it contains some number of grains
up to its threshold.  At each time step, one of two things can happen: either a new particle can enter the system at a leaf, filling
the first available site along the geodesic from the leaf to the root (and if none are available, then it leaves
the system); or a site can topple, moving its grains along the geodesic
to the root and filling the first available sites (possibly some grains will leave the system).

Using the techniques developed in this paper, we computed the eigenvalues with multiplicities
and a reasonable upper bound on the mixing time.  A key ingredient was proving the $\RR$-triviality of the
monoid corresponding to the landslide nonabelian directed sandpile model. In Section~\ref{sandpile} we provide
an alternative proof of this fact. When all thresholds are $1$, we proved that the stationary distribution admits an explicit product measure.
See~\cite{ayyer_schilling_steinberg_thiery.sandpile.2013} for details.

\smallskip
\paragraph{\textbf{The exchange walk on a finite Coxeter group}}
In Section~\ref{section.coxeter} we examine another generalization of the Tsetlin library, this time associated to
a finite Coxeter system $(W,S)$~\cite{bjornerbrenti}.  The state space
for this finite Markov chain consists of all reduced
decompositions for the longest element $w_0$ of $W$. The transitions, called \emph{exchange moves}, are as follows.
If the system is in state $s_1\cdots s_n$, then a generator $s\in S$ is randomly chosen and a transition is made to the new state
$ss_1\cdots s_{i-1}\widehat s_is_{i+1}\cdots s_n$ where $\widehat s_i$ means omit $s_i$.  The generator $s_i$ to
omit in order to obtain a reduced decomposition of $w_0$ is unique according to the Exchange Condition
for Coxeter groups~\cite{bjornerbrenti}.

For example, if $W=(\mathbb Z/2\mathbb Z)^n$ and $S$ is the standard basis for $W$, then $(W,S)$ is a
Coxeter system, $w_0$ is the all-ones vector, the reduced decompositions for $w_0$ are those words over
$S$ containing all letters and no repetitions (i.e., the permutations of $S$) and an exchange move is just
move-to-front. So we recover the Tsetlin library in this case.

When $W=\mathcal S_n$ is the symmetric group and $S$ is the set of adjacent transpositions, then $(W,S)$
is a Coxeter system.  The longest element $w_0$ is $i\mapsto n-i+1$ (in one-line notation it is $n,n-1,\ldots,1$).
A well known result of Stanley~\cite{stanley.1984} says that reduced decompositions of
$w_0$ are equinumerous with standard tableaux of staircase shape.  An explicit bijection was given by Edelman
and Greene~\cite{edelmangreen}. So this chain can be viewed as a stochastic process on such tableaux.

Using the techniques of $\RR$-trivial monoids, we are able to compute the eigenvalues with multiplicities,
give a simple formula for the stationary distribution and provide an upper bound on the mixing time for the exchange walk
on a finite Coxeter group.

\smallskip
\paragraph{\textbf{Promotion chains}}
In~\cite{ayyer_klee_schilling.2012}, the Tsetlin library Markov chain was generalized by looking at
linear extensions $\mathcal{L}$ of a finite poset $P$ of size $n$. The transition between two linear extensions is given by
a variant of the promotion operator on posets~\cite{schuetzenberger.1972}.

For a linear extension $\pi=\pi_1\cdots \pi_n \in \mathcal L$ in one-line notation, the generalized promotion
operator $\partial_i$ for $1\le i<n$ can be defined as~\cite{haiman.1992,malvenuto_reutenauer.1994,stanley.2009}
\[
	\partial_i(\pi) = \tau_{n-1} \tau_{n-2} \cdots \tau_i(\pi)\;.
\]
Here $\tau_i$ acts on $\pi$ by interchanging $\pi_i$ and $\pi_{i+1}$ if $\pi_i$ and $\pi_{i+1}$ are not comparable
in $P$. Otherwise, it acts as the identity.
Define $\hat{\partial}_i(\pi) = \partial_{\pi_i^{-1}}(\pi)$. Assigning probability $x_i$ to the operator $\hat{\partial}_i$ defines
the promotion Markov chain on $\mathcal L$. For any poset $P$, the stationary distribution of the Markov chain was given
by an explicit product formula~\cite{ayyer_klee_schilling.2012}.

When $P$ is the antichain on $n$ vertices (that is, there are no imposed ordering relations between
any of the vertices), then $\mathcal L$ is the set of all linear orderings and the promotion Markov chain reduces
to the Tsetlin library (where now books are moved to the end of the stack instead of the front due to a difference
in conventions).

For special posets, called rooted forests, the eigenvalues and their multiplicities of the transition matrices can also
be computed explicitly. Recall that a rooted forest is a poset where each vertex has at most one successor.
It was shown~\cite{ayyer_klee_schilling.2012} that in that case, the underlying transition monoid is $\RR$-trivial. The eigenvalues
can then be computed using the techniques presented in this paper.
In~\cite{ayyer_klee_schilling.2013} the mixing time for this Markov chain was also estimated using monoid techniques.

\smallskip
\paragraph{\textbf{Structure of the paper}}
Let us now describe the content of each section in more detail.
Since this paper is intended for an audience of algebraists,
combinatorialists and probabilists, we include in
Section~\ref{preliminaries} some background about each of these areas.

In Section~\ref{section.random walks on monoids}, we present general
results for random walks on monoids before specializing to
$\RR$-trivial Markov chains in Section~\ref{rtrivial-markov-chain}.
In particular, for $\RR$-trivial monoids, we describe combinatorially the eigenvalues by character theory,
or equivalently through inclusion-exclusion on a lattice (Theorem~\ref{theorem.eigenvalues}),
give a sufficient condition for diagonalizability (Theorem~\ref{diagonal}) generalizing the result of
Brown~\cite{BrownLRB} (see also~\cite{DiaconisBrown}),
provide a formula for the stationary distribution (Theorems~\ref{productformula} and~\ref{productformula2}),
relate the rate of convergence with some properties of the monoid
(Corollary~\ref{lrbcaseofconvergencetime}), and conclude with a bound on the mixing time (Corollary~\ref{corollary.mixing time}).
This theory subsumes that of left regular band random
walks developed in~\cite{BrownLRB}.

When investigating examples, we discovered that the generators of the
transition monoids often satisfy certain types of relations,
reminiscent of the plactic relations~\cite[Chapter~5]{lothaire2}. In Section~\ref{section.free tree monoid}, we
study the largest such monoid. The relations admit a nice Knuth--Bendix completion, and it follows that its combinatorics
is governed by a certain class of trees, which motivates its name: the \emph{\fhb}.
One of the main results is that the \fhb{} is $\RR$-trivial (Corollary~\ref{corollary.fhb.r_trivial}).
The lattice of regular $\mathscr L$-classes of idempotents of the \fhb{} is the
Boolean lattice and we provide a simple transversal of idempotents.

In the remaining sections, we study several examples of $\RR$-trivial
Markov chains, applying results of Section~\ref{rtrivial-markov-chain}, and using the
\fhb{} on several occasions for concise proofs of $\RR$-triviality and using its representation theory
in order to benefit from its simple combinatorics.

In Section~\ref{section.toom} we consider two new generalizations of
the Tsetlin library, with multiple copies of books and with storage or interlibrary loan, respectively.
This model can also be regarded as a generalization
of the Toom model~\cite{toom1980,jnr1996} to finite size as well as arbitrary particles.
Theorems~\ref{theorem.toom.interval.eigenvalues} and~\ref{theorem.toom.window.eigenvalues}
provide the spectra of these models.
In Section~\ref{sandpile} we provide a short proof of the
$\RR$-triviality of the landslide nonabelian directed sandpile model
of~\cite{ayyer_schilling_steinberg_thiery.sandpile.2013} using the \fhb{} of
Section~\ref{section.free tree monoid}. Finally, in
Section~\ref{section.coxeter}, we consider a Markov chain on the set of
reduced words of the longest element of a finite Coxeter group and provide its spectrum and stationary
distribution, as well as an upper bound on its mixing time. This model is also a generalization of the Tsetlin library,
which appears in the case of a finite right-angled Coxeter group.

\subsection*{Acknowledgments}
We would like to thank Persi Diaconis, Lionel Levine, John Pike, John Rhodes, and Dan Romik
for helpful discussions, as well as the organizers of the conference ``International Conference on Geometric, Combinatorial
and Dynamics aspects of Semigroup and Group Theory" in June 2013 at Bar Ilan University, Israel, where part of this research
was presented. Thanks to Zachary Hamaker for pointing out the relation of the exchange walk in Section~\ref{section.coxeter}
with~\cite{BBHM06,AHR09}.
This research was driven by computer exploration using Maple${}^{{\text{TM}}}$~\cite{Maple}, KBMag~\cite{KBMag},
{\sc Sage}~\cite{sage} and {\sc Sage-combinat}~\cite{sage-combinat}. The Maple package for the Toom-Tsetlin
model \texttt{ToomTsetlin} is available on the first author's webpage and as an ancillary file in the arXiv source.
Maple packages for the promotion chain \texttt{LinearExtensions} and the sandpile model \texttt{Nonabelian\-Sandpiles}
are available from the first author's webpage.

All the authors would like to thank ICERM, where part of this work was performed, for its hospitality.
This work was partially supported by a grant from the Simons Foundation (\#245268
to Benjamin Steinberg) and the Binational Science Foundation of Israel and the US (\#2012080 to Benjamin Steinberg).
AS was partially supported by NSF grants DMS--1001256, OCI--1147247,
and a grant from the Simons Foundation (\#226108 to Anne Schilling).

\section{Background on Markov chains and monoids}
\label{preliminaries}
Since this paper is intended for an audience of algebraists, combinatorialists and probabilists, we include some
background about each of these areas.

\subsection{Markov chains}

We recall here some basic notions from Markov chain theory.  Details
can be found in e.g.~\cite{Markovmixing}.
Let $\Omega$ be a finite set.  A \emph{probability distribution} (or simply a probability) on $\Omega$ is a
mapping $P\colon \Omega\to \mathbb R$ such that $P(\omega)\geq 0$ for all $\omega\in \Omega$ and
\[
	\sum_{\omega\in \Omega}P(\omega)=1\,.
\]
The probability that an element of $\Omega$ chosen randomly according
to $P$ belongs to some subset $A\subseteq\Omega$ is given by
\[
	P(A) = \sum_{\omega\in A}P(\omega)\,.
\]

A (finite state) \emph{Markov chain} is a pair $\mathcal M=(\Omega,T)$ consisting of a (finite) state space
$\Omega$ and a (column) \emph{stochastic matrix} $T\colon \Omega\times \Omega\to \mathbb R$.  Recall
that $T$ is stochastic if:
\begin{enumerate}
\item $T(\alpha,\beta)\geq 0$ for all $\alpha,\beta\in \Omega$;
\item for all $\beta\in \Omega$,
\[
	\sum_{\alpha\in \Omega}T(\alpha,\beta)=1.
\]
\end{enumerate}
One calls $T$  the \emph{transition matrix} of the chain.  The intuition is that if you are in state $\beta$,
then with probability $T(\alpha,\beta)$ you move to state $\alpha$.\footnote{Note that some authors prefer
to denote this probability as $T(\beta,\alpha)$.}

We can view $T$ as an operator on $\mathbb R^{\Omega}$ in the usual way:
$Tf(\omega) = \sum_{\beta\in \Omega}T(\omega,\beta)f(\beta)$.  It is easy to see that $T$ preserves
probability distributions and so if $\nu$ is an initial distribution, then $T^n\nu$ is a probability distribution
known as the \emph{$n^{th}$-step} distribution of the Markov chain.  That is, $T^n\nu(\omega)$ is the
probability of being in state $\omega$ on the $n^{th}$-step of the chain if the chain starts with initial
distribution $\nu$.

We say that $\pi$ is a \emph{stationary distribution} for $T$ if $T\pi=\pi$.  It is a consequence of the
Perron-Frobenius theorem that each Markov chain has at least one stationary distribution. A Markov chain
$\mathcal M=(\Omega,T)$ is called \emph{irreducible} if, for each $\alpha,\beta\in \Omega$, there
exists $n\geq 0$ such that $T^n(\alpha,\beta)>0$. In the language of
graphs, this translates as follows.
Define a digraph $\Gamma(\mathcal M)$ with vertex set $\Omega$ and a directed edge
$\beta\to \alpha$ if $T(\alpha,\beta)> 0$.  Then $\mathcal M$ is irreducible if and only if
$\Gamma(\mathcal M)$ is strongly connected.  Irreducible Markov chains have a unique stationary
distribution $\pi$ and moreover $\pi>0$ (has strictly positive entries).  The
Markov chain $\mathcal M$ is said to be \emph{ergodic} if $T^n>0$ for
some $n\geq 0$, or equivalently for any large enough $n$.  It is
well known that this is equivalent to asking that the chain be irreducible and that the greatest common
divisor of the lengths of the cycles of $\Gamma(\mathcal M)$ be $1$ (that is, the associated digraph is \emph{primitive}).
In this case, for any initial distribution $\nu$, the sequence $T^n\nu$ converges to the stationary distribution $\pi$.

Strongly connected components of $\Gamma(\mathcal M)$ are called \emph{communicating classes} in Markov chain
theory.  A communicating class is called \emph{essential} if the corresponding strong component is
minimal under the ordering on strongly connected components defined by $C\leq C'$ if there is a directed path
from  $C'$ to $C$. States which belong to essential communicating classes are said to be \emph{recurrent};
the remaining states are called \emph{transient}. It is well known and easy to see that
$\lim_{n\to \infty} T^n(\alpha,\beta)=0$ if $\alpha,\beta$ do not
belong to the same essential
communicating class, and that if $\pi$ is a stationary distribution of $T$, then $\pi(\omega)=0$ for
each transient state $\omega\in \Omega$ (see~\cite[Section~1.7]{Markovmixing}).

In Markov chain theory, one usually measures the rate of convergence in terms of the total variation
distance.  Recall that $\mathbb R^{\Omega}$ is a real Banach space with the $\ell^1$-norm:
$\|f\|_1 = \sum_{\omega\in \Omega} |f(\omega)|$.  Let $\mathcal P(\Omega)$ be the space of
probability distributions on $\Omega$; it is a compact subspace of the $\ell^1$-unit ball.
The \emph{total variation distance} between two probability distributions $\nu,\mu$ is defined by
\[
	\|\nu-\mu\|_{TV} = \max_{A\subseteq \Omega}|\nu(A)-\mu(A)|.
\]
The following equivalent expressions are extremely useful.
\begin{proposition}[See e.g. Proposition~4.2 and Remark~4.3 of\cite{Markovmixing}]\label{p:totalvariationdist}
Let $\nu,\mu$ be probabilities on $\Omega$, and $A=\{\omega\in
\Omega\mid \nu(\omega)\geq \mu(\omega)\}$. Then,
\[
	\|\nu-\mu\|_{TV} = \frac{1}{2}\|\nu -\mu\|_1 = \nu(A)-\mu(A)\,.
\]
\end{proposition}

Let $\mathcal M=(\Omega,T)$ be a finite state ergodic Markov chain with stationary distribution $\pi$. Let
$d(n) = \sup_{\nu\in \mathcal P(\Omega)} \|T^n\nu-\pi\|_{TV}$.  Then, for $\varepsilon >0$, the \emph{mixing time}
of $\mathcal M$ is $t_{mix}(\varepsilon) =\min\{n\mid d(n)\leq \varepsilon\}$~\cite{Markovmixing}.  Often
authors choose $\varepsilon =e^{-1}$ or $\varepsilon = 1/4$ to define the mixing time.  We usually try to
bound, for $c>0$, when $\|T^n\nu -\pi\|_{TV}\leq e^{-c}$.

\subsection{Semigroups and monoids}
\label{section.semigroups}
We recall here some basic notions from semigroup theory.  Details can be found
in~\cite{CP,Arbib,Howie,Almeida:book,EilenbergB,Pin.MPRI} or~\cite[Appendix A]{qtheor}.

A \emph{semigroup} $S$ is a set with an associative multiplication $S\times S\to S$.  It is called a
\emph{monoid} if additionally it contains an identity element, usually denoted $1$.

An element $e$ of a semigroup $S$ is \emph{idempotent} if $e^2=e$. The set of idempotents
is denoted $E(S)$.   Each element  $s$ of a finite semigroup has unique idempotent (positive) power,
traditionally denoted $s^{\omega}$. In particular, every non-empty finite semigroup contains an idempotent.

A finite semigroup $S$ is said to be \emph{aperiodic} if
$s^{\omega}s=s^{\omega}$, for all $s\in S$. Equivalently, $S$ is aperiodic if
there is a positive integer $n$ such that $s^n=s^{n+1}$ for any $s\in S$. Trivially, any subsemigroup
or homomorphic image of an aperiodic semigroup is aperiodic.

An \emph{ideal} of a monoid $M$ is a non-empty subset $I$ such that $MIM\subseteq I$.  Left ideals and right
ideals are defined analogously. If $I,J$ are ideals of a monoid $M$, then $IJ\subseteq I\cap J$ and hence
$I\cap J\neq \emptyset$. It follows that every finite monoid has a unique minimal ideal. Let $M$ be a finite monoid.
Any ideal of $M$ is a subsemigroup and hence contains an idempotent. If $I$ is the minimal ideal of $M$ and $e\in E(I)$,
then $eMe=eIe$ is a group with identity $e$.  In particular, if $I$ is aperiodic, then $eMe=\{e\}$.

We now introduce two of Green's relations~\cite{Green}, namely $\LL$ and $\RR$.
Let $M$ be a monoid.  Then the \emph{principal right} ideal generated by
$m\in M$ is $mM$.  One defines a preorder on $M$ by putting
$m\leq_{\RR} m'$ if $mM\subseteq m'M$.  One defines
$m\mathrel{\RR} m'$ if $m\leq_{\RR} m'$ and $m'\leq_{\RR} m$  (i.e., $mM=m'M$).
The classes for this relation are called $\RR$-classes; they are
the strongly connected components of the right Cayley graph of $M$ with respect to any generating set.

A monoid $M$ is \emph{$\RR$-trivial} if Green's relation $\RR$ is trivial, that is, if $m\mathrel{\RR} m'$
(i.e., $mM=m'M$) implies $m=m'$. Equivalently, $M$ is $\RR$-trivial if the right Cayley
graph of $M$ with respect to any generating set is acyclic. In this case $\leq_{\RR}$ is a partial order on $M$. Note that $\leq_{\RR}$
is compatible with left multiplication, that is, $m\leq_{\RR} m'$ implies $nm\leq_{\RR} nm'$ for all $n\in M$.
A finite $\RR$-trivial monoid is necessarily aperiodic since
$s^{\omega}\mathrel{\RR} s^{\omega}s$ in any finite monoid. The class of finite $\mathscr R$-trivial
monoids is closed under taking finite direct products, submonoids, and homomorphic images.

Green's relation $\LL$ and $\LL$-trivial monoids are defined
symmetrically on the left.

A \emph{left zero semigroup} is a semigroup $S$ satisfying the identity
$xy=x$ for all $x,y\in S$. Let $L$ be the $\LL$-class of an
idempotent of an $\RR$-trivial monoid; such an $\LL$-class is called a \emph{regular} $\LL$-class. A regular $\LL$-class $L$ is always a left zero
semigroup; more generally, for any $x\in L$ and $y\in M$, one has $xy=x$ if and only
if $Mx\subseteq My$.  The minimal ideal of an $\RR$-trivial monoid $M$ is a left zero semigroup and is
the unique minimal left ideal of $M$.

A monoid $M$ is called a \emph{left regular band} if $x^2=x$ and $xyx=xy$ for all $x,y\in M$.
Left regular bands are $\RR$-trivial, which can be seen as follows. Suppose $x$ and $y$ are in the
same $\RR$-class, that is, there exist $u,v\in M$ such that $xu=y$ and $yv=x$. Then
\[
	x=yv=xuv=xuvu=yvu=xu=y
\]
since $uv=uvu$. Hence all $\RR$-classes are singletons. More generally, a finite monoid $M$ is $\RR$-trivial
if and only if $(xy)^{\omega}x=(xy)^{\omega}$ for all $x,y\in M$.

\subsection{Random mapping representations}
A \emph{left action}
of a monoid on a set $\Omega$ is a mapping $M\times \Omega\to \Omega$, written
$(m,\omega)\mapsto m\omega$, such that
\begin{enumerate}
\item $m(m'\omega)=(mm')\omega$
\item $1\omega=\omega$
\end{enumerate}
for all $m,m'\in M$ and $\omega\in \Omega$. Right actions are defined symmetrically.

If $X\subseteq M$, then the \emph{Cayley digraph} of
the action of $X$ on $\Omega$ is the digraph $\Gamma(\Omega,X)$ with vertex set $\Omega$ and
edges $\omega\to x\omega$ for $\omega\in \Omega$ and $x\in X$ (sometimes we use the Cayley digraph
with labelled edges $\omega\xrightarrow{\,\,x\,\,} x\omega$).

If $X\subseteq M$, then $\langle X\rangle$ denotes the submonoid of $M$ generated by $X$,
that is, the smallest submonoid of $M$ containing $X$.

Let $M$ be a (finite) monoid acting on the left of a (finite) set $\Omega$.  Suppose that $P$ is a
probability on $M$.  Then we have an induced Markov chain $\mathcal M=(\Omega,T)$ where
\[
	T(\alpha,\beta) =\sum_{m\beta=\alpha}P(m) = P(\{m\in M\mid m\beta=\alpha\})\,.
\]
We call this Markov chain the \emph{random walk of $M$ on $\Omega$ driven by $P$}. The fact that $T$ is
stochastic is simply the computation
\begin{align*}
\sum_{\alpha\in \Omega}T(\alpha,\beta)&= \sum_{\alpha\in \Omega}\sum_{m\in M}P(m)\delta_{\alpha,m\beta}\\
&= \sum_{m\in M}P(m)\sum_{\alpha\in \Omega} \delta_{\alpha,m\beta} = \sum_{m\in M}P(m)=1\,.
\end{align*}
The data consisting of the action of $M\times \Omega\to \Omega$ and the probability $P$ on $M$
is called a \emph{random mapping representation} of the Markov chain $\mathcal M$.

A matrix $A\colon \Omega\times \Omega\to \mathbb R$ is
called \emph{column monomial} if each column of $A$ is a standard basis vector (i.e., contains
exactly one non-zero entry, which must be a one).  Note that such a column monomial matrix is
stochastic. Column monomial matrices are exactly the linear operators induced by mappings
$f\colon \Omega\to \Omega$, the corresponding column monomial matrix $A_f$ being given by
\[
	A_f(\alpha,\beta) = \begin{cases} 1, & \text{if}\ f(\beta)=\alpha,\\ 0, & \text{else.}\end{cases}
\]
To prove that every Markov chain has a random mapping representation we use the following well-known lemma.

\begin{lemma}\label{l:convexcombocolumnmonomial}
Every stochastic matrix is a convex combination of column monomial matrices.
\end{lemma}
\begin{proof}
The set $\mathcal S$ of stochastic matrices is a polytope whose vertices are the column monomial matrices
(cf.~the discussion after~\cite[Theorem~5.3]{BermanPlemmons}). As each point of a polytope is a convex combination
of vertices, the lemma follows.
\end{proof}

\begin{theorem}\label{t:hasmappingrep}
Every finite state Markov chain has a random mapping representation.
\end{theorem}
\begin{proof}
Let $\mathcal M=(\Omega,T)$ be a Markov chain.  Then $T$ can be written as a convex combination of column monomial
matrices by Lemma~\ref{l:convexcombocolumnmonomial}.  If $M$ is the monoid of all mappings on $\Omega$ and if $P$
is the probability on $M$ that gives a mapping $f$ the same weight that its corresponding column monomial matrix  $A_f$ receives
in the above convex combination expressing $T$, then $\mathcal M$ is the random walk of $M$ on $\Omega$ driven by $P$.
 \end{proof}

\begin{remark}
Theorem~\ref{t:hasmappingrep} is a basic fact of probability theory, although
it is usually stated in a different language: namely, it is equivalent with~\cite[Proposition~1.5]{Markovmixing}.  In~\cite{Markovmixing}
a random mapping representation of $\mathcal M=(\Omega,T)$ is defined as consisting of a mapping
$f\colon \Lambda\times \Omega\to \Omega$ and a $\Lambda$-valued random variable $Z$ such that the probability that
$f(Z,\beta)=\alpha$ is $T(\alpha,\beta)$ for all $\alpha,\beta\in \Omega$. (Actually, since~\cite{Markovmixing} uses row
stochastic matrices, they provide a dual formulation.)

In other words, a random mapping representation of $\mathcal M=(\Omega,T)$ in the sense of~\cite{Markovmixing}
consists of a deterministic automaton with state set $\Omega$ and input alphabet $\Lambda$ (but no initial or accepting
states) together with a $\Lambda$-valued random variable $Z$.  If you are in state $\omega\in \Omega$, the Markov chain
works by choosing a random letter $a\in \Lambda$, distributed identically to $Z$, and performing the transition
$\omega\xrightarrow{\,\,a\,\,} a\omega$ in the automaton. (Note that our convention is to process words in an automaton from right-to-left.)

Given such a mapping $f$, we can define a mapping $F\colon \Lambda\to M$,
where $M$ is the monoid of all mappings on $\Omega$, by currying: $F(\lambda)(\alpha)=f(\lambda,\alpha)$.  Then $F(Z)$ is an
$M$-valued random variable given by some probability distribution $P$ on $M$. It is straightforward to verify that $\mathcal M$
is the random walk of $M$ on $\Omega$ driven by $P$.

Conversely, if $f\colon M\times \Omega\to \Omega$ is an action and $P$ is a probability on $M$, then the random walk
$\mathcal M=(\Omega,T)$ of $M$ on $\Omega$ driven by the probability $P$ has a random mapping representation in
the sense of~\cite{Markovmixing} by taking $\Lambda=M$, $f$ to be the action and $Z$ to be the $M$-valued random variable
with distribution $P$.

In summary, a random mapping representation of a Markov chain can also be specified by giving a collection $S$ of mappings
on the state space $\Omega$ and a probability distribution $P$ on $S$.
We can then take $M$ to be the monoid of mappings on $\Omega$ generated by $S$ and view $P$ as a probability
on $M$.  Sometimes we refer informally to $S$ as the \emph{generators} of the Markov chain.
\end{remark}

\subsection{Random mapping representations with constants}
If $P,Q$ are probabilities on a monoid $M$, their convolution is the probability
\[P\ast Q(m) = \sum_{m_1m_2=m}P(m_1)Q(m_2)\] on $M$.
Recall that the \emph{support} of a probability $P$  on $M$ is the set
\[
	\supp P=\{m\in M\mid P(m)>0\}\,.
\]
Denote by $P^{\ast n}$ the $n^{th}$-convolution power of $P$. It is the distribution of $X_nX_{n-1}\cdots X_1$
where $X_1,\ldots, X_n$ are independent random variables distributed according to $P$.
Then $P^{\ast n}(m)>0$ for some $n\geq 0$ if and only if $m$ is in the
submonoid $\langle \supp P\rangle$ generated by the support of $P$.

Let $\mathcal M$ be a Markov chain with random mapping representation $M\times \Omega\to \Omega$ driven by a probability $P$.
Then, $\Gamma(\mathcal M)$ is the Cayley digraph of the
action of $M$ on $\Omega$ with respect to the set $X=\supp P$. In
particular, $\mathcal M$ is irreducible (that is, $\Gamma(\mathcal M)$
is strongly connected) if and only if the action of $\langle \supp
P\rangle$ is transitive on $\Omega$ (that is, for any $\alpha,\beta
\in \Omega$, there exists $n\in \langle\supp P\rangle$ with $n\alpha=\beta$).

The following proposition is folklore.
\begin{proposition}\label{p:hasallconstant}
Let $\mathcal M=(\Omega,T)$ be an irreducible Markov chain with a random mapping
representation $M\times \Omega\to \Omega$ driven by a probability $P$. Let
$N=\langle \supp P\rangle$ and suppose that some $m\in N$ acts as a constant map on
$\Omega$.  Then $N$ contains all constant maps on $\Omega$ and the Markov chain
$\mathcal M$ is ergodic.
\end{proposition}
\begin{proof}
By irreducibility $N$ acts transitively on $\Omega$. If $m$ acts as a constant mapping with image
$\omega$ and $m'\omega=\alpha$, then $m'm$ acts as the constant map to $\alpha$ and hence
$N$ contains all constant maps on $\omega$.

Note that, if the constant map to $\beta\in \Omega$ can be represented by a product $m_1\cdots m_k$
of $k$ elements of $\supp P$, then for any $m\in \supp P$ one has that $m_1\cdots m_k m$ acts as the
constant map to $\beta$.  Thus the constant map to $\beta$ can be represented as a product of $r$
elements of $\supp P$ for any $r\geq k$. It now follows that there exists $t\geq 0$ such that the constant
map on $\Omega$ with image $\alpha$ can be represented by a product $m_{\alpha}$ of $t$ elements
of $\supp P$ for all $\alpha\in \Omega$. But then $T^t(\alpha,\beta)\geq P^{\ast t}(m_{\alpha})>0$ and so
$\mathcal M$ is ergodic.
\end{proof}

Note that, under any action of a monoid $M$
on a set $\Omega$, the fixed-point set of an idempotent $e$ is its image $e\Omega$.

The following result is well known to automata theorists.

\begin{proposition}\label{p:aperiodichasconstants}
Let $M$ be a monoid acting transitively on a set $\Omega$ and suppose that the minimal ideal $I$
of $M$ is aperiodic. Then, for any $\omega\in \Omega$, there is an element $m\in I$ acting as a
constant map to $\omega$.
\end{proposition}
\begin{proof}
It suffices by the proof of Proposition~\ref{p:hasallconstant} to show that $I$ contains some element
$m$ acting as constant map, since then $Mm\subseteq I$ will contain all the constant maps by
transitivity.  Let $e\in I$ be an idempotent. Suppose that $\alpha,\beta\in e\Omega$.  By transitivity
there exists $m\in M$ with $m\alpha=\beta$.  As $eMe=\{e\}$ by aperiodicity of $I$, we conclude
that $\beta=e\beta=em\alpha =eme\alpha=e\alpha=\alpha$.  Thus $e$ acts as a constant map.
\end{proof}

As an immediate corollary of the preceding results we obtain the following result.

\begin{corollary}\label{c:aperiodicchains}
Let $\mathcal M=(\Omega,T)$ be an irreducible Markov chain with a random mapping
representation $M\times \Omega\to \Omega$ driven by a probability $P$. Suppose further that $M$ is aperiodic
or, more generally, the minimal ideal $I$ of the submonoid $N$ generated by the support of $P$ is
aperiodic.  Then, all elements of $I$ act on $\Omega$ as constant
maps and every constant map on $\Omega$ is obtained via the action of some element of $I$.
In particular $\mathcal M$ is ergodic.
\end{corollary}

\begin{proof}
Since $\mathcal M$ is irreducible, the submonoid $N=\langle \supp P\rangle$ acts transitively on $\Omega$.
Assume that its minimal ideal $I$ is aperiodic. By Proposition~\ref{p:aperiodichasconstants}, $I$ contains
elements acting as every constant map on $\Omega$.  Note that if $m\in I$ acts as a constant map, then
each element of $ImI$ also acts as a constant map because any map factoring through a constant map
is a constant map.  As $ImI$ is an ideal contained in $I$, we must have $ImI=I$ by minimality of $I$.  Thus
every element of $I$ acts on $\Omega$ as a constant map. Ergodicity follows from Proposition~\ref{p:hasallconstant}.

Finally, note that if $M$ is aperiodic, then so is any of its subsemigroups.  In particular, $I$ will be aperiodic.
\end{proof}

\begin{remark}
Assume that the action of $M$ is faithful, as will be the case in
most of our examples. Then the minimal ideal $I$ of submonoid $N$ generated by the support of $P$ consists
precisely of the constant maps on $\Omega$ and hence is canonically in bijection with $\Omega$ (by sending
a constant map to its image) and, moreover, that bijection is an isomorphism of the action of $N$ on the left
of $I$ with the action of $N$ on $\Omega$.
\end{remark}

\section{Random walks on monoids}
\label{section.random walks on monoids}
A number of results from this section can be viewed as special cases of results about probability measures
on compact semigroups~\cite{mukherjea}, but it seems better in our context to just prove them.
Let $M$ be a finite monoid.  Denote by $\LO M$ the vector space of all functions $f\colon M\to \mathbb R$
equipped with the $\ell^1$-norm $\|\cdot\|_1$.  Then $\LO M$ is a finite-dimensional real Banach
algebra with respect to the convolution product
\[
	(f\ast g)(m) = \sum_{xy=m}f(x)g(y)\,.
\]
As an algebra, we can identify $\LO M$ with the monoid algebra $\mathbb RM$ via
$f\mapsto \sum_{m\in M}f(m)m$ and we shall do this when convenient.

A probability distribution $P$ on $M$ can be viewed as an element of $\LO M$.
The probability distributions form a compact multiplicative submonoid of $\LO M$.  Notice that if
$X$ and $Y$ are independent $M$-valued random variables with respective distributions $\nu$
and $\mu$, then the distribution of the random variable $X\cdot Y$ is $\nu\ast \mu$.
Recall that $P^{\ast n}$ denotes the $n^{th}$-convolution power of $P$, which is the distribution of
$X_nX_{n-1}\cdots X_1$ where $X_1,\ldots, X_n$ are independent random variables distributed according to $P$.

The \emph{left random walk on $M$ driven by $P$} is the Markov chain with random mapping representation
coming from the action of $M$ on itself by left multiplication and the probability $P$.  The right random walk is defined dually.

Suppose that $M$ acts on a finite set $\Omega$.  We can identify $\mathbb R^{\Omega}$ with
$\mathbb R\Omega$. We then have a natural $\LO M$-module structure on $\mathbb R\Omega$
given by having $f\in \LO M$ act on a basis element $\omega\in \Omega$ by
\[
	f\cdot \omega = \sum_{m\in M}f(m)m\omega\,.
\]
From the point of view of functions, for $f\in \LO M$, $g\in \mathbb R^{\Omega}$ and $\omega\in \Omega$, the module
structure is given by
\[
	(f\cdot g)(\omega) =\sum_{m\in M}\sum_{m\alpha=\omega}f(m)g(\alpha)\,.
\]
The following proposition is well known, but important.

\begin{proposition}\label{p:transitionmatrix}
Let $M$ act on $\Omega$ and let $P$ be a probability on $M$ (viewed as
an element of $\LO M$). Then, the transition matrix $T$
of the random walk of $M$ on $\Omega$ driven by $P$ is the matrix with respect to the basis $\Omega$
of the operator on $\mathbb R\Omega$ defined by $v\mapsto Pv$.  It follows that, if $\nu$ is a probability on
$\Omega$ (viewed as an element of $\mathbb R\Omega$), then $T^n\nu = P^{\ast n}\nu$.
\end{proposition}

\begin{proof}
We have $P\beta =\sum_{m\in M}P(m)m\beta$ and thus the coefficient of $\alpha$ in
$P\beta$ is $\sum_{m\beta=\alpha}P(m)=T(\alpha,\beta)$.
\end{proof}

A crucial consequence of the proposition is that any $\LO M$-submodule of $\mathbb R\Omega$
is an invariant subspace for the transition matrix $T$.

Recall that the minimal ideal $I$ of a finite monoid $M$ is the disjoint union of all the minimal left
ideals of $M$~\cite{CP,Arbib}. Let us say that a probability $P$ on $M$ is \emph{adapted} if the submonoid
generated by the support of $P$ contains the minimal ideal.  Note that a probability on a group
is adapted if and only if the support generates the group, which is the usual definition in that context.
In general, if the support generates the monoid, then the probability is adapted but the converse
need not be true. The following result is straightforward and well known~\cite{mukherjea}, but
we include it for completeness.

\begin{proposition}\label{trivium}
Let $M$ be a finite monoid with minimal ideal $I$ and let $P$ be an adapted probability on $M$.
Then the recurrent states of the left random walk on $M$ driven by $P$ are the elements of $I$.
The essential communicating classes of the chain are the minimal left ideals of $M$.  The restriction
of the random walk to any minimal left ideal is irreducible.  Moreover, the
chain so obtained is independent of which minimal left ideal is chosen.
\end{proposition}

\begin{proof}
Because $\langle \supp P\rangle$ contains $I$ and each minimal left ideal of $M$ is a left
zero semigroup, it follows that the minimal left ideals are precisely the minimal strong
components of the left Cayley digraph of $M$ with respect to the set $\supp P$.   This explains
the recurrent elements and the essential communicating classes.  By Green's lemma~\cite{Green},
any two minimal left ideals are isomorphic via right multiplication by a monoid element.  This
gives an isomorphism of the corresponding Markov chains.
\end{proof}

Let us assume from now that $M$ is a monoid whose minimal ideal is a left zero semigroup $\wh 0$,
that is, $mt=m$ for all $m\in \wh 0$ and $t\in M$.  Equivalently, the minimal ideal of $M$ is the unique
minimal left ideal of $M$ and has a trivial maximal subgroup.  As we have seen, this is the case for
aperiodic monoids acting faithfully and transitively on the left of a finite set. It is also the case for
$\mathscr R$-trivial monoids, which form the primary object of study for most of the paper.

If $\pi$ is a probability with support contained in $\wh 0$ and $P$ is any probability, then
since $\wh 0$ is a two-sided ideal, $\pi\ast P$ is supported on $\wh 0$ and one has, for
$m\in \wh 0$,
\begin{equation*}
	(\pi \ast P)(m) = \sum_{xy=m}\pi(x)P(y) =\pi(m)\sum_{y\in M}P(y)=\pi(m).
\end{equation*}
Thus we have proved:

\begin{lemma}\label{trivialobs}
If $M$ is a monoid whose minimal ideal $\wh 0$ is a left zero semigroup and if $\pi$ is a probability on
$M$ with support contained in $\wh 0$, then $\pi\ast P=\pi$ for any probability $P$ on $M$. In particular,
$\pi$ is idempotent.
\end{lemma}

We can now describe in the following theorem the stationary
distribution for a random walk on a monoid whose minimal ideal is a
left zero semigroup, and derive in
Corollary~\ref{t:couplingfromthepast} a bound on mixing times of
Markov chains with a random mapping representation containing constant
maps. Roughly speaking, the mixing time is bounded by the
probability that a product of $n$ elements does not act as a
constant. This is essentially a variation of a technique that goes
under the name ``coupling from the past'' in the literature and can be
found in~\cite{DiaconisBrown} for the case when the action of $M$ is
faithful. It is the key tool we shall use to obtain mixing times.

\begin{theorem}\label{idempotentconvergence}
  Let $M$ be a finite monoid whose minimal ideal $\wh 0$ is a left
  zero semigroup, and let $P$ be an adapted probability on $M$. Then,
\begin{enumerate}
\item The sequence $P^{\ast n}$ converges to an idempotent probability $\pi$ with support $\wh 0$ and
\begin{equation}\label{firstconv}
  \|P^{\ast n}-\pi\|_{TV} = P^{\ast n}(M\setminus \wh 0)\,.
\end{equation}
\item The random walk on $\wh 0$ driven by $P$ is ergodic with $\pi$
  as stationary distribution. Moreover, for any distribution $\nu$ on
  $\wh 0$,
\begin{equation}\label{sndconv}
  \|P^{\ast n}\ast\nu-\pi\|_{TV}\leq P^{\ast n}(M\setminus \wh 0)\,.
\end{equation}
\end{enumerate}
\end{theorem}
\begin{proof}
  Recall that, by Proposition~\ref{trivium}, $\wh 0$ is the set of
  recurrent elements for the left random walk on $M$ driven by
  $P$. Therefore, for $m\notin \wh 0$, the sequence $P^{\ast n}(m)$
  converges to zero as $n \to \infty$. Take now $m\in \wh 0$. The
  sequence $P^{\ast n}(m)$ is non-decreasing: indeed, since $mt=m$ for
  all $t\in M$,
  \[
  P^{\ast (n+1)}(m)\geq P^{\ast n}(m)\sum_{t\in M}P(t) = P^{\ast n}(m)\,.
  \]
  Moreover, since $P$ is adapted, there exists $n>0$ such that
  $P^{\ast n}(m)>0$. Finally, the sequence $P^{\ast n}(m)$ is bounded
  by $1$ and therefore converges to some real number $\pi(m)$ with
  $0<\pi(m)\leq 1$.
  Altogether, using that probability distributions are closed in $\LO
  M$ in conjunction with Lemma~\ref{trivialobs}, we obtain that $P^{\ast n}$
  converges to an idempotent probability $\pi$ with support $\wh 0$.

  Observe that the set $A$ of elements of $M$ on which $P^{\ast n}$ is
  greater than $\pi$ is precisely $M\setminus \wh 0$.
  Proposition~\ref{p:totalvariationdist} then implies that
  \begin{equation}\label{bound}
    \|P^{\ast n}-\pi\|_{TV} = P^{\ast n}(M\setminus \wh 0)-\pi(M\setminus \wh 0) = P^{\ast n}(M\setminus \wh 0)\,.
  \end{equation}

  Let us now turn to (2). Since $P^{\ast n}\to \pi$ and multiplication
  in $\LO M$ is norm-continuous, $\pi$ commutes with $P$. Combining
  this with Lemma~\ref{trivialobs} gives that $P\ast \pi = \pi\ast
  P=\pi$. Therefore, $\pi$ is the unique stationary distribution for
  the left random walk on $\wh 0$ driven by $P$ (uniqueness is given
  by the irreducibly of the walk and Proposition~\ref{trivium}). The
  random walk is ergodic by Proposition~\ref{p:hasallconstant}.

  To conclude, take any initial distribution $\nu$ on $\wh 0$. Using
  successively Lemma~\ref{trivialobs}, that the $\ell^1$-norm is
  submultiplicative, that probabilities have $\ell^1$-norm $1$, and
  Equation~\eqref{bound} we obtain as desired:
\begin{equation}\label{equationstationary}
\begin{split}
	\|P^{\ast n}\ast\nu -\pi\|_{TV} &= \|P^{\ast n}\ast\nu -\pi\ast\nu\|_{TV}\\
        &= \frac{1}{2}\|P^{\ast n}\ast\nu -\pi\ast\nu\|_1\\
	&\leq \frac{1}{2}\|P^{\ast n}-\pi\|_1\cdot \|\nu\|_1 \\
        &= \|P^{\ast n}-\pi\|_{TV}\\
        &= P^{\ast n}(M\setminus \wh 0)\,.\qedhere
\end{split}
  \end{equation}
\end{proof}

\begin{corollary}\label{t:couplingfromthepast}
Let $\mathcal M=(\Omega,T)$ be an irreducible Markov chain with random mapping representation
$M\times \Omega\to \Omega$ driven by a probability $P$. Suppose, moreover, that $M$ contains an
element acting as a constant on $\Omega$ (e.g., if the minimal ideal of $M$ is aperiodic) and that
$P$ is adapted.  Then, the following hold.
\begin{enumerate}
\item $\mathcal M$ is ergodic.
\item Let $\mu$ be the stationary distribution for the random walk of $M$ on a minimal left ideal $L$
driven by $P$ and let $\pi$ be the stationary distribution for $\mathcal M$. Then
\begin{equation}\label{lumpeddist}
	\pi(\omega) = \sum_{\{x\in L\mid x\omega=\omega\}}\mu(x)\,.
\end{equation}
\item Let $I$ be the ideal of those elements of $M$ acting as constant maps on $\Omega$.  Then, for
any probability distribution $\nu$ on $\Omega$, we have
\begin{equation}\label{sndconv2}
	\|T^n\nu-\pi\|_{TV}\leq P^{\ast n}(M\setminus I)\,.
\end{equation}
\end{enumerate}
\end{corollary}

\begin{proof}
The first item is part of Proposition~\ref{p:hasallconstant}.
The idea for the second item is that $\mathcal M$ is a lumping of the random walk of $M$ on $L$.
The set $I$ of elements of $M$ acting as a constant map is an ideal and hence contains the minimal
ideal (and consequently $L$). Let $\Psi\colon L\to \Omega$ be defined by $x\Omega=\{\Psi(x)\}$ for
$x\in L$.  It is easily checked that $\Psi(mx)=m\Psi(x)$ for all $x\in L$ and $m\in M$.  It follows that
$\Psi$ induces an $\LO M$-module homomorphism $\Psi\colon \mathbb RL\to \mathbb R\Omega$.
We claim that $\pi=\Psi(\mu)$.

First note that $\Psi(\mu)$ is a probability distribution.  Indeed, it is easy to check that
$\Psi(\mu)(\omega) = \mu(\Psi^{-1}(\omega))$, which is the right hand side of~\eqref{lumpeddist}.
Next we have that $T\Psi(\mu) = P\Psi(\mu)=\Psi(P\mu)=\Psi(\mu)$ and hence $\pi=\Psi(\mu)$.
This establishes the second item.

To prove the third item, observe that the action of $M$ on $\Omega$ induces a homomorphism
$\varphi\colon M\to \mathcal T_{\Omega}$. Let $N=\varphi(M)$.  Then $N$ acts faithfully on
$\Omega$ and, in particular, the minimal ideal $J$ of $N$ is a left zero semigroup consisting of
the constant maps on $\Omega$ (cf.~Proposition~\ref{p:hasallconstant}).  Let $Q$ be the probability
on $N$ defined by $Q(n) = P(\varphi\inv(n))$; so $Q(A)=P(\varphi\inv(A))$ for any $A\subseteq N$.
As a surjective monoid homomorphism maps minimal ideals onto minimal ideals, it follows that $Q$
is adapted.   Observe that if $\Phi\colon \LO M\to \LO N$ is the homomorphism induced by
$\delta_m\mapsto \delta_{\varphi(m)}$ (i.e., $(\Phi(f))(n) = \sum_{m\in \varphi^{-1}(n)}f(m)$), then $Q=\Phi(P)$.
It is then easy to see that $\mathcal M$ is the random walk of $N$ on $\Omega$ driven by $Q$, which is isomorphic to
the random walk of $N$ on $J$ driven by
$Q$. Theorem~\ref{idempotentconvergence} then yields as desired that, for any probability $\nu$ on $\Omega$,
\[
	\|T^n\nu-\pi\|_{TV}\leq Q^{\ast n}(N\setminus J)=P^{\ast n}(\varphi\inv(N\setminus J)) = P^{\ast n}(M\setminus I)\,.\qedhere
\]
\end{proof}

The following lemma provides a technique for applying Corollary~\ref{t:couplingfromthepast}.  It is based on the same
arguments as in~\cite{ayyer_schilling_steinberg_thiery.sandpile.2013}[Sections 2.4 and 5.3]
and \cite{ayyer_klee_schilling.2013}[Section 6].

\begin{lemma}\label{l:statisticbound}
Let $\mathcal M=(\Omega,T)$ be an irreducible Markov chain with random mapping representation
$M\times \Omega\to \Omega$ driven by a probability $P$. Let $\pi$ be the stationary distribution.  Suppose that $M$ contains
an element acting as a constant on $\Omega$ and that $P$ is adapted.  Let $f\colon M\to \mathbb N$ be a function, called
a \emph{statistic}, such that:
\begin{enumerate}
\item $f(mm')\leq f(m)$ for all $m,m'\in M$;
\item if $f(m)>0$, then there exists $m'\in M$ with $P(m')>0$ and $f(mm')<f(m)$;
\item $f(m)=0$ if and only if $m$ acts as a constant on $\Omega$.
\end{enumerate}
Then if $p=\min\{P(m)\mid m\in M, P(m)>0\}$  and $n=f(1)$, we have that
\[
  \|T^k\nu -\pi\|_{TV} \le \sum_{i=0}^{n-1} {k\choose i}p^i(1-p)^{k-i}
  \leq \exp\left(-\frac{(kp-(n-1))^2}{2kp}\right)\,,
  \]
for any probability distribution $\nu$ on $\Omega$, where the last inequality holds as long as $k\ge (n-1)/p$.
\end{lemma}
\begin{proof}
Corollary~\ref{t:couplingfromthepast} yields $\|T^k\nu-\pi\|_{TV}\leq P^{\ast k}(M\setminus I)$ where $I$ is the ideal of
elements of $M$ acting as constant maps.  Consider the right random walk on $M$ driven by $P$,  that is, the Markov
chain whose state set is $M$ and if we are in state $m$, then we choose a random element $m'\in M$ distributed according
to $P$ and move to $mm'$.  Then $P^{\ast k}(M\setminus I)$ is the probability that if we start at $1$, then on step $k$ of the
right random walk on $M$ we are not in $I$.

Let us call a step $m_i\mapsto m_{i+1}$ in the right random walk on $M$ \emph{successful} if either $m_i\in I$ or
$f(m_{i+1})<f(m_i)$. Property~1 of $f$ implies that $f(m_i)=f(m_{i+1})$ if the step is not successful. By Property~3, if
$m$ is the current state after $k$-steps of the walk, then the probability $f(m)>0$ is precisely $P^{\ast k}(M\setminus I)$.
The probability that $f(m)>0$ after $k$ steps of the right random walk on $M$ is less than the probability of having at most
$n-1$ successful steps in the first $k$ steps.

Property~2 of $f$ says that each step has probability at least $p$ to be successful.  Therefore, the probability that $f(m)>0$
after $k$ steps of the right random walk on $M$ is bounded above by the probability of having at most $n-1$ successes
in $k$ Bernoulli trials with success probability $p$.

Using Chernoff's inequality for the cumulative distribution function of a binomial random variable we obtain
that (see for example~\cite[after Theorem~2.1]{devroye_lugosi.2001})
\[
  ||P^k-\pi ||_{TV} \le \sum_{i=0}^{n-1} {k\choose i}p^i(1-p)^{k-i}
  \leq \exp\left(-\frac{(kp-(n-1))^2}{2kp}\right)\,,
\]
where the last inequality holds as long as $k\ge (n-1)/p$.
\end{proof}

\section{Generalities on $\RR$-trivial random walks}
\label{rtrivial-markov-chain}

From now on we confine our attention to $\RR$-trivial monoids, which
form a class rich enough to contain many interesting examples, but
restrictive enough to provide a workable theory. In particular, this
theory subsumes the left regular band theory of Brown~\cite{BrownLRB}.

\subsection{The spectrum of the transition matrix}
\label{subsection.spectrum}

The spectra of random walks on minimal left ideals of a fairly general
class of monoids -- those with simple modules of dimension $1$ -- was
computed in~\cite{mobius1,mobius2}. We recap here for completeness the
special case of $\mathscr R$-trivial monoids, where no group theoretic considerations intervene.

Suppose that $M$ is a finite $\mathscr R$-trivial monoid. Let
\[
\LeftIdeal(M)=\{Mm\mid m\in M\}
\] be the poset
of principal left ideals of $M$ ordered by inclusion.  Note that
$M/{\mathscr L}$ is partially ordered by $\leq_{\mathscr L}$ and is isomorphic to $\LeftIdeal(M)$.

Let
\[
	\Lambda(M)=\{Me\mid e\in E(M)\}
\]
be the subposet of idempotent-generated principal left ideals.
It is well known that $\Lambda(M)$ is a lattice and that $Me\wedge Mf=M(ef)^{\omega}$.  Moreover,
the mapping $c\colon M\to \Lambda(M)$ defined by $c(m)= Mm^{\omega}$ is a homomorphism
(details can be found, for example, in~\cite{MargolisSteinbergQuiver}). Sometimes $c$ is called the \emph{content map}.

Define $d\colon M\to \Lambda(M)$ by $d(m)=Me$ where $e$ is any element of the minimal ideal
of the right stabilizer of $m$.  One has that $mt=m$ if and only if $c(t)\geq d(m)$. Sometimes $d$ is called the \emph{right descent map}.

The mappings $c,d$ descend to order
preserving maps $c,d\colon \LeftIdeal(M)\to \Lambda(M)$ with
\begin{align*}
c(Mm) &= \bigvee \{Me\in \Lambda(M)\mid Me\leq Mm\}\\
d(Mm) &= \bigwedge \{Me\in \Lambda(M)\mid Mm\leq Me\}
\end{align*}
and so in particular $c(Mm)\leq Mm\leq d(Mm)$ and $c(Mm)=d(Mm)$ if and only if $Mm\in \Lambda(M)$.
Thus one has $c=d$ if and only if $M$ is a left regular band.

\begin{remark}
For the categorically minded, we observe that if $e\in E$, then $Me\leq Mm$ if and only if $Me\leq c(Mm)$
and $Mm\leq Me$ if and only if $d(Mm)\leq Me$ and therefore $c,d$ are right and left adjoints, respectively,
of the inclusion of $\Lambda(M)$ into $\LeftIdeal(M)$.
\end{remark}

It is well known that, if $M$ is $\mathscr R$-trivial, then every simple $\mathbb RM$-module is
one-dimensional, cf.~\cite{myirreps,AMSV}.  More precisely, there is one irreducible character
$\chi_X\colon M\to \mathbb R$ for each $X\in \Lambda(M)$ given by
\[
	\chi_X(m) =\begin{cases} 1, & \text{if}\ Mm\geq X\ \text{(i.e., $c(m)\geq X$)},\\ 0, &\text{else.}\end{cases}
\]

The following is a reformulation of a theorem of the third author
from~\cite{mobius1} to a slightly more general setting. It generalizes
straightforwardly to any monoid whose regular $\mathscr J$-classes are aperiodic semigroups.
For representation theorists this theorem and its proof can be
summarized as follows: the multiplicities of the eigenvalues are given
by the multiplicities of the isomorphism types of simple modules in
the composition factors of $\mathbb{R}\Omega$; the later can be
computed by character theory, counting fixed points of appropriate
elements of the monoid and inverting the character table. This last
step boils down to a M\"obius inversion since the character table is
given by the incidence matrix of the poset $\Lambda(M)$.

\begin{theorem}[Steinberg~\cite{mobius1}]
\label{theorem.eigenvalues}
Let $P$ be a probability on an $\mathscr R$-trivial monoid $M$ and let $M$ act on $\Omega$.  Let
$T$ be the transition matrix for the random walk of $M$ on $\Omega$ driven by $P$.  Fix, for each
$X\in \Lambda(M)$, an idempotent $e_X$ with $X=Me_X$ and let $\mu$ be the M\"obius function of $\Lambda(M)$.
Then each $X\in \Lambda(M)$ contributes an eigenvalue
\begin{equation}\label{eigenvalueseq}
	\lambda_X = \sum_{Mm\geq X}P(m)=\sum_{c(m)\geq X}P(m)\,,
\end{equation}
with (possibly null) multiplicity given by
\[
	m_X =\sum_{Y\leq X}|e_Y\Omega|\cdot \mu(Y,X) \,.
\]
All eigenvalues of $T$ are obtained this way.
\end{theorem}
\begin{proof}
In what follows we identify $\LO M$ with $\mathbb RM$.  Choose a composition series for the $\mathbb RM$-module
$\mathbb R\Omega=V_n\supseteq V_{n-1}\supseteq \cdots\supseteq V_0=\{0\}$.
Each simple $\mathbb RM$-module $V_j/V_{j-1}$ is one-dimensional.  As each $V_j$
is an invariant subspace for $T$ (which acts on $\mathbb R\Omega$ as $P$),  we see, by choosing a
basis adapted to this composition series, that $T$ is similar to an upper triangular matrix of the form
\begin{equation}\label{uppdecomp}
\begin{bmatrix}\chi_1(P)     & \ast   &\cdots  &\ast\\
                    0       & \chi_2(P)    &\ddots  &\vdots\\
                    \vdots  & \ddots &\ddots  & \ast\\
                    0       & \cdots &0       &\chi_{|\Omega|}(P)
\end{bmatrix}\,,
\end{equation}
where the $\chi_i$ are characters of $M$.  Therefore, the eigenvalues
are given by the $\chi_i(P)$.  If $\chi_i$ is
the character $\chi_X$ corresponding to $X\in \Lambda(M)$, then
\[
	\chi_i(P) =\sum_{m\in M}P(m)\chi_X(m) = \sum_{Mm\geq X} P(m)=\lambda_X\,.
\]
To compute the multiplicity of $\lambda_X$, observe that the character $\theta$ of the module
$\mathbb R\Omega$ counts the number of fixed points, that is, for
$m\in M$,
\[
	\theta(m)=|\{\omega\in \Omega\mid m\omega=\omega\}|\,.
\]
In particular, $\theta(e_X)=|e_X\Omega|$.  On the other hand,
$\theta(e_X)=\sum_{i=1}^n\chi_i(e_X)$, and using that
\[
	\chi_Y(e_X)=\begin{cases} 1, & \text{if}\ Y\leq X,\\ 0, & \text{else,}\end{cases}
\]
we get
\[
	|e_X\Omega|=\theta(e_X)=\sum_{Y\leq X}m_Y\,.
\]
M\"obius inversion then yields the desired multiplicity:
 \[m_X=\sum_{Y\leq X}|e_Y\Omega|\cdot \mu(Y,X)\,.\qedhere\]
\end{proof}

\subsection{A sufficient condition for diagonalizability}
Let $P$ be a probability on an $\mathscr R$-trivial monoid $M$.  We give a sufficient condition for
diagonalizability of $P$ as an operator on $\LO M$.  This implies the diagonalizablity of the transition matrix
$T$ of any random walk of $M$ on a set $\Omega$ driven by $P$.  This is because the subalgebra
$\mathbb R[P]$ of $\LO M$ generated by $P$ will be split semisimple and thus its quotient algebra
$\mathbb R[T]$ will also be split semisimple, which is the same thing as saying that $T$ is diagonalizable.

This generalizes Brown's diagonalizability result~\cite{BrownLRB} for left regular band walks.
In what follows we write $m$ for $\delta_m$ and omit the $\ast$
for convolution (i.e., we identify $\LO M$ with $\mathbb RM$).

\begin{theorem}\label{diagonal}
  Let $P$ be a probability on an $\mathscr R$-trivial monoid $M$ and
  let $N$ be the submonoid generated by the support of $P$. Recall from
  Theorem~\ref{theorem.eigenvalues} that the eigenvalues of $P$ are of
  the form
  \[
	\lambda_X=\sum_{c(m)\geq X} P(m)\,,
  \]
  where $X\in \Lambda(M)$.

  Assume that $\lambda_{d(m)}\neq \lambda_{d(m')}$ whenever $m\in M$,
  $m'\in mN$ and $m'\neq m$. Then, $P(m)$ is diagonalizable as an operator on the left of $\LO M$ and hence the transition
  matrix of any random walk of $M$ on a finite set $\Omega$ driven by $P$ is diagonalizable.
\end{theorem}

Before proving the theorem, we recover Brown's theorem on diagonalizability of left regular band walks~\cite{Brown1}.

\begin{corollary}
  A random walk of a left regular band $M$ on a set $\Omega$ has a diagonalizable transition
  matrix.
\end{corollary}
\begin{proof}
We verify the criterion in Theorem~\ref{diagonal} applies.
  Take $m\in M$ and $m'=mn\in mN$ such that $m'\ne m$.
  Then, $c(n)\geq c(m')=d(m')$ and $c(n)\ngeq d(m)=c(m)$.  On the
  other hand $d(m')=c(m')\leq c(m)=d(m)$. Thus $\lambda_{d(m')}\geq
  \lambda_{d(m)}+P(n)>\lambda_{d(m)}$.
\end{proof}

\begin{proof}[Proof of Theorem~\ref{diagonal}]
  We will prove that the minimal polynomial $q$ of $P$ is square-free.
  Note that $q$ coincides with the minimal polynomial of $P$ acting on
  the left and on the right: indeed, $q(P)=0$ if and only if
  $0=q(P)1=1q(P)$. We consider here the action of $P$ on the right of
  $\LO M$ to exploit the $\mathscr R$-triviality of $M$, i.e., that $\leq_\RR$ is
  a partial order.  Define a partial order $\preceq$ on $M$ by $m'\preceq m$ if $m'\in mN$.
  Note that $m'\preceq m$ implies $m'\leq_\RR m$ and so $\preceq$ is indeed a partial order.
  We write $m'\prec m$ if $m'\preceq m$ and $m'\neq m$.

\begin{lemma}
\label{theformula}
Let $m\in M$. Then,
\[
	m(P-\lambda_{d(m)}) = \sum_{c(t)\ngeq d(m)} P(t)mt\,,
\]
  with all the non-zero terms of the summand on the right hand side satisfying $mt\prec m$.
\end{lemma}
\begin{proof}
  Recall that $c(t)\geq d(m)$ if and only if $mt=m$, and otherwise
  $mt<_\RR m$ by $\RR$-triviality. Therefore,
\begin{align*}
mP
   &= \sum_{c(t)\geq d(m)}P(t)m+\sum_{c(t)\ngeq d(m)}P(t)mt\\
   &= \lambda_{d(m)}m+\sum_{c(t)\ngeq d(m)}P(t)mt\,.\qedhere
\end{align*}
Since $P(t)>0$ implies $t\in N$, the lemma follows.
\end{proof}

For $m\in M$, let $\sigma(m) = \{\lambda_{d(mn)}\mid n\in N\}$ and
consider the square-free polyomials
\begin{equation*}
	q_m(x) = \prod_{\lambda\in \sigma(m)}(x-\lambda) \qquad \text{and} \qquad
	Q_m(x) = \frac{q_m(x)}{x-\lambda_{d(m)}}\,.
\end{equation*}
By our hypothesis on $P$, $q_{m'}(x)$ divides $Q_m(x)$ whenever
$m'\prec m$ because $\sigma(m')\subseteq \sigma(m)\setminus \{\lambda_{d(m)}\}$.

\begin{lemma}
\label{kill}
If $m\in M$, then $m\cdot q_m(P)=0$.
\end{lemma}
\begin{proof}
  The proof is by induction on the order $\preceq$. Suppose first
  that $m$ is $\preceq$-minimal. Then, $m=mn$ for all $n\in N$, i.e.,
  $c(n)\geq d(m)$ for all $n\in N$. Hence,
  $\sigma(m)=\{\lambda_{d(m)}\}$, and Lemma~\ref{theformula}
  immediately yields $m\cdot q_m(P) = m(P-\lambda_{d(m)}) = 0$.

  In general, assume that the lemma holds for all $m'\in M$ with
  $m'\prec m$. Since $q_{m'}(P)$ divides $Q_m(P)$, this implies
  $m'Q_m(P)=0$. Therefore, using Lemma~\ref{theformula},
\begin{align*}
	m\cdot q_{m}(P)&= m\cdot (P-\lambda_{d(m)})\cdot Q_m(P)
	\\ &= \sum_{c(t)\ngeq d(m)}P(t)mt\cdot Q_m(P)
	=0
\end{align*}
(since $c(t)\ngeq d(m)$ and $P(t)>0$ implies $mt\prec m$).
\end{proof}
The theorem follows by taking $m=1$: since $1\cdot q_1(P)=0$, the
minimal polynomial $q$ of $P$ divides $q_1$ and is therefore
square-free.
\end{proof}
Note that the above proof does not use that $P$ is a probability. In
fact, independently of the ground field, Theorem~\ref{diagonal}
applies to any element of the algebra of an $\RR$-trivial monoid.

Let us define a probability $P$ on $M$ to be \emph{generic} if, for all $X\neq Y\in \Lambda(M)$, we
have that
\[
	\lambda_X=\sum_{Mm\geq X} P(m)\neq \sum_{Mm\geq Y} P(m)=\lambda_Y.
\]
Note that generic probabilities are those probabilities that do no lie on a certain finite set of hyperplanes and
hence are generic in all reasonable senses of the word.

\begin{corollary}\label{c:generic}
Suppose that $M$ is an $\mathscr R$-trivial monoid such that $m>_{\mathscr R} m'$ implies that
$d(m)\neq d(m')$.  Then every generic probability $P$ is diagonalizable as an operator on $\LO M$ and
consequently, the transition matrix of any random walk of $M$ on a set
driven by a
generic probability is diagonalizable.
\end{corollary}

\begin{proof}
The result is immediate from Theorem~\ref{diagonal} since for a generic probability we have $d(m)\neq d(m')$
implies $\lambda_{d(m)}\neq \lambda_{d(m')}$.
\end{proof}

\subsection{A formula for the stationary distribution for $\mathscr R$-trivial monoids}

We continue to assume that $M$ is an $\mathscr R$-trivial monoid with minimal ideal $\wh 0$ and let $P$
be an adapted probability on $M$.  Our goal is to give an explicit formula for the stationary distribution of the random
walk on $\wh 0$ driven by $P$. We continue to use the notation \eqref{eigenvalueseq}.

Let $T$ be the transition matrix for the \emph{right} random walk on $M$ driven by $P$.  So $T$ is a row stochastic
$M\times M$-matrix with
\[
	T(m,t) = \sum_{mx=t}P(x)\,.
\]
Note that $T(m,m) = \lambda_{d(m)}$ and that
\begin{equation}\label{convolved}
	T^n(1,m)=P^{\ast n}(m)\,.
\end{equation}

Also observe that $T$ belongs to the incidence algebra of $(M,
\geq_\RR$) (recall that the \emph{incidence algebra} of a
finite poset $\Poset$ is the algebra of all upper triangular
$\Poset\times \Poset$-matrices over $\mathbb R$; that is, all $A\colon
\Poset\times \Poset\to \mathbb R$ such that $A(p,q)=0$ if $p\nleq q$).
In particular, $T$ is an upper triangular matrix if we order $M$ along
a linear extension of $\geq_\RR$.

We recall that if $\Poset$ is a finite poset, then the \emph{order complex} of $\Poset$ is the simplicial complex
whose vertex set is $\Poset$ and whose $q$-simplices are strictly decreasing chains
$\sigma=\sigma_0>\sigma_1>\cdots>\sigma_q$ of elements of $\Poset$.

Let $\Delta(M)$ be the order complex of $(M,\leq_\RR)$.  Let $\mathrm{St}(1)$ be the \emph{star} of $1$; it consists
of all simplices $\sigma$ containing $1$ as a vertex.  If $m\in M$, let $N(m)$ be the set of all simplices in
$\mathrm{St}(1)$ with minimal vertex $m$, i.e., it consists of all strictly decreasing chains
$1=\sigma_0>_{\mathscr R}\cdots>_{\mathscr R}\sigma_q=m$. A simplex $\sigma\in \mathrm{St}(1)$ will
always be written $\sigma = (\sigma_0,\sigma_1,\ldots, \sigma_q)$ where $q=\dim \sigma$, $\sigma_0=1$
and $\sigma_i>_{\mathscr R} \sigma_{i+1}$. Let us put
\[
	P(\sigma)=\prod_{i=1}^qT(\sigma_{i-1},\sigma_i)\;.
\]
Notice that $P(\sigma)$ will be $0$ unless there is a product of elements in the support of $P$
which visits precisely the $\mathscr R$-classes of $\sigma$.

The \emph{complete homogeneous symmetric polynomial} of degree $j$ in variables $x_1,\ldots, x_n$ is
denoted $h_j(x_1,\ldots, x_n)$; it is the sum of all monomials of degree $j$.

\begin{proposition}\label{explicitcomputation}
  Let $m\in M$.  Then,
\[
	P^{\ast n}(m) = \sum_{\substack{\sigma\in N(m)\\
	\dim \sigma\leq n}} P(\sigma)h_{n-\dim \sigma}(\lambda_{d(\sigma_0)},\ldots,\lambda_{d(\sigma_{\dim \sigma})}).
\]
\end{proposition}

\begin{proof}
We have that $P^{\ast n}(m) = T^n(1,m)$.  As $T^n(1,m)$ is in the incidence algebra of $(M,\geq_{\mathscr R})$, it
follows (using $T(m,m) = \lambda_{d(m)}$) that
\begin{equation}\label{simplecomputation}
	P^{\ast n}(m) = \sum\sum \lambda_{d(\sigma_0)}^{r_0}T(\sigma_0,\sigma_1)\lambda_{d(\sigma_1)}^{r_1}\cdots
	T(\sigma_{q-1},\sigma_q)\lambda_{d(\sigma_q)}^{r_q}
\end{equation}
where the sum runs over all $\sigma=(\sigma_0,\sigma_1,\ldots, \sigma_q)\in N(m)$
with $q\leq n$ and $r_0+\cdots+r_q=n-q$. As desired, this gives:
\[
	\sum_{\substack{\sigma\in N(m)\\\dim \sigma\leq n}} P(\sigma)
	h_{n-\dim \sigma}(\lambda_{d(\sigma_0)},\lambda_{d(\sigma_1)},\ldots,\lambda_{d(\sigma_{\dim \sigma})})\,.\qedhere
\]
\end{proof}

If $m\in \wh 0$, then $c(m)=d(m) = \wh 0$ and $\lambda_{d(m)} =1$.  Thus we have the following specialization of
Proposition~\ref{explicitcomputation} for $m\in \wh 0$.

\begin{corollary}\label{minimalidealcase}
Let $m\in \wh 0$.  Then
\[
	P^{\ast n}(m) = \sum_{\substack{\sigma\in N(m)\\\dim \sigma\leq n}}P(\sigma)\cdot
	\sum_{r\leq n-\dim \sigma}h_{r}(\lambda_{d(\sigma_0)},\ldots,\lambda_{d(\sigma_{\dim \sigma-1})}).
\]
\end{corollary}

We now can compute a formula for the stationary distribution.

\begin{theorem}\label{productformula}
Let $P$ be an adapted probability on a finite $\mathscr R$-trivial monoid $M$ with minimal ideal $\wh 0$.
Then the stationary distribution $\pi$ of the random walk on $\wh 0$ driven by $P$ is given by
\[
	\pi(m) = \sum_{\sigma\in N(m)} \prod_{i=1}^{\dim \sigma} \frac{T(\sigma_{i-1},\sigma_i)}{1-\lambda_{d(\sigma_{i-1})}}
	= \sum_{\sigma\in N(m)} \prod_{i=1}^{\dim \sigma} \frac{\displaystyle{\sum_{\sigma_{i-1}x=\sigma_i}
	P(x)}}{\displaystyle{1-\sum_{c(x)\geq d(\sigma_{i-1})}P(x)}} \;,
\]
where $N(m)$ consists of all chains $1=\sigma_0>_{\mathscr R}\sigma_1>_{\mathscr R}\cdots>_{\mathscr R}\sigma_q=m$.
\end{theorem}

\begin{proof}
By Theorem~\ref{idempotentconvergence} we know that $\pi(m) = \lim_{n\to \infty} P^{\ast n}(m)$.  By
Corollary~\ref{minimalidealcase}
\begin{align*}
\lim_{n\to \infty} P^{\ast n}(m) &=
	\sum_{\sigma\in N(m)}P(\sigma)\cdot \sum_{r=0}^{\infty}h_{r}(\lambda_{d(\sigma_0)},\ldots,\lambda_{d(\sigma_{\dim \sigma-1})})\\
	&= \sum_{\sigma\in N(m)}P(\sigma)\cdot \prod_{i=0}^{\dim\sigma-1}\sum_{j=0}^{\infty}\lambda_{d(\sigma_i)}^j\\
	&= \sum_{\sigma\in N(m)}P(\sigma)\cdot \prod_{i=0}^{\dim\sigma-1}\frac{1}{1-\lambda_{d(\sigma_i)}}\\
	&=  \sum_{\sigma\in N(m)} \prod_{i=1}^{\dim \sigma} \frac{T(\sigma_{i-1},\sigma_i)}{1-\lambda_{d(\sigma_{i-1})}}\,.\qedhere
\end{align*}
\end{proof}

\begin{remark}
The stationary distribution $\pi$ admits the following probabilistic interpretation. It is the probability of obtaining $m$ via the
following process. You start at the identity and continue the process until you arrive at the minimal ideal $\wh 0$ at which
point you stop.  If you are at $t\in M$, then you remove from the support of $P$ all elements in the right stabilizer of $t$
and then renormalize to obtain a probability $Q_t$.  Select an element $x$ of $S$ according to $Q_t$ and move to $tx$.

Equivalently, this is the usual right random walk on the monoid,
except one rejects each step that does not go strictly down in the $\RR$-order.
\end{remark}

\subsection{Reduced words and product formulas}
Let $P$ be an adapted  probability on an $\mathscr R$-trivial monoid $M$ with minimal left ideal $\wh 0$ and denote by $X$
the support of $P$.  We write $[w]_M$ for the image in $M$ of a word $w$ in the free monoid $X^*$.  If $w=w_1\cdots w_n$
is in $X^*$, let $\sigma(w)$ be the simplex of $\Delta(M)$ given by the set
\[
	\sigma(w)=\left\{1,[w_1]_M,[w_1w_2]_M,\ldots, [w_1\cdots w_n]_M\right\}.
\]
Note that the elements $[w_1\cdots w_i]_M$ with $i=0,\ldots, n$ need not be distinct; if they are we call the word
$w$ \emph{reduced}.  Define the reduction $\rho(w)$ of $w$ to be the word obtained by removing those letters
$w_i$ with $[w_1\cdots w_i]_M=[w_1\cdots w_{i-1}]_M$. It is easy to see that $[\rho(w)]_M=[w]_M$ and
$\sigma(w)=\sigma(\rho(w))$. For $m\in M$, denote by $\mathrm{Red}(m)$ the set of all reduced words $w\in X^*$
with $[w]_M=m$.  The reduced words are precisely the elements of the Karnofsky--Rhodes expansion of
$M$ with respect to the set $X$~\cite{Elston}; they were used by Brown in his proof of the diagonalizability of
left regular band walks~\cite{BrownLRB}.

It is immediate from the definition that if $\sigma$ is a simplex of $\Delta(M)$ and $R(\sigma)$ is the set of
reduced words $w$ with $\sigma(w)=\sigma$, then
\[
	P(\sigma) = \sum_{w\in R(\sigma)}P(w_1)\cdots P(w_{|w|})\,.
\]
In light of this, Theorem~\ref{productformula} admits the following reformulation.

\begin{theorem}\label{productformula2}
Let $P$ be an adapted probability on a finite $\mathscr R$-trivial monoid $M$ with minimal ideal $\wh 0$.  Then, the stationary
distribution $\pi$ of the random walk on $\wh 0$ driven
by $P$ is given by
\begin{align*}
	\pi(m) &= \sum_{w\in \mathrm{Red}(m)} \prod_{i=1}^{|w|} \frac{P(w_i)}{1-\lambda_{d([w_1\cdots w_{i-1}]_M)}} \\
	&= \sum_{w\in \mathrm{Red}(m)} \prod_{i=1}^{|w|} \frac{P(w_i)}{\displaystyle{1-\sum_{c(x)\geq
	d([w_1\cdots w_{i-1}]_M)}P(x)}}\;.
\end{align*}
\end{theorem}

Theorem~\ref{productformula2} reduces to a product formula in the special case that each element of the monoid
admits a unique reduced representative. In fact, much of the random walk theory becomes particularly simple in this
case.  So let $M$ be an $\mathscr R$-trivial monoid with generating set $X$.  We say that $M$ is \emph{Karnofsky--Rhodes
with respect to $X$} if each element of $M$ can be represented by a unique reduced word over $X$.  This is equivalent to
saying that the right Cayley digraph of $M$ with respect to $X$ becomes a directed rooted tree after removal of loop edges.  Free left regular
bands are examples, and we shall encounter others in this paper. Abusing notation slightly, we write $\mathrm{Red}(m)$
for the unique reduced word representing the element $m$. Notice that if $M$ is Karnofsky-Rhodes with respect to $X$,
then $m\leq_{\mathscr R}n$ if and only if $\mathrm{Red}(n)$ is a prefix of $\mathrm{Red}(m)$; in particular, if $e\in E(M)$,
then $em=m$ if and only if $\mathrm{Red}(e)$ is a prefix of $\mathrm{Red}(m)$.  The following corollary is immediate
from this discussion and Theorem~\ref{productformula2}.

\begin{corollary}\label{productformulacor}
Let $M$ be a finite $\RR$-trivial monoid which is Karnofsky--Rhodes with respect to a generating set $X$.
Let $P$ be a probability on $M$ with support $X$. Denote by $\wh 0$ the minimal ideal of $M$.  Let $\pi$ be the
stationary distribution of the random walk on $\wh 0$ driven by $P$. For an idempotent $e$, let $r_e$ be the number
of elements of $\wh 0$ whose reduced expression has $\mathrm{Red}(e)$ as a prefix.
\begin{enumerate}
\item If $e\in E(M)$, then the multiplicity of the eigenvalue of the transition matrix corresponding to $Me$ is
\[
\sum_{Mf\leq Me} r_f\mu(Mf,Me)\,
\]
where $\mu$ is the M\"obius function of $\Lambda(M)$.
\item If $m\in \wh 0$ with $\mathrm{Red}(m)=w_1\cdots w_n$, then
\[
	\pi(m) = \prod_{i=1}^{n} \frac{P(w_i)}{1-\lambda_{d([w_1\cdots w_{i-1}]_M)}}
	=\prod_{i=1}^{n} \frac{P(w_i)}{\displaystyle{1-\sum_{c(x)\geq d([w_1\cdots w_{i-1}]_M)}P(x)}}\,.
\]
\end{enumerate}
\end{corollary}

It is not hard to see how to recover the stationary distribution for the Tsetlin library from this corollary.
If $w=w_1\cdots w_n$ is a repetition-free word over an $n$-letter alphabet and we use the free LRB as the
monoid $M$, then $w$ itself is its only reduced representative.

\begin{remark}
  One more generally obtains a product formula as long as $\sigma(w)$
  is constant along the reduced words $w$ of each given element $m$.
\end{remark}

\subsection{Rates of convergence for $\mathscr R$-trivial monoids}

We continue to assume that $P$ is an adapted probability on an $\mathscr R$-trivial monoid $M$ with minimal left ideal
$\wh 0$.  In this section we give a crude upper bound on the rate of convergence to stationarity of the random walk on $\wh 0$.
Let $\nu$ be a probability on $\wh 0$. Then, by Theorem~\ref{idempotentconvergence}, we know that
\[
	\|P^{\ast n}\nu-\pi\|_{TV}\leq P^{\ast n}(M\setminus \wh 0)\,.
\]
We proceed by bounding the right hand side.

For $L$ an $\mathscr L$-class, let
\[
	M_{\geq_{\mathscr L}L} = \{m\in M\mid m\geq_{\mathscr L} L\}\,.
\]
Clearly
\[
	P^{\ast n}(M_{\geq_{\mathscr L}L}) = \sum_{L'\geq_{\mathscr L} L}P^{\ast n}(L')\,,
\]
and so by M\"obius inversion we have
\begin{equation}\label{mobiussum}
P^{\ast n}(L) = \sum_{L'\geq_{\mathscr L} L} P^{\ast n}(M_{\geq_{\mathscr L}L'})\cdot \mu(L,L')\,.
\end{equation}
where $\mu$ denotes the M\"obius function for the induced order on $M/{\mathscr L}$, then

Note that, if $L'\in \Lambda(M)$, then $P^{\ast n}(M_{\geq_{\mathscr L}L'})=\lambda_{L'}^n$. One then has the following
result in the left regular band case.

\begin{corollary}\label{lrbcaseofconvergencetime}
Suppose that $M$ is a left regular band and $P$ is an adapted probability.  Then,
\[
	P^{\ast n}(M \setminus \wh 0) = -\sum_{X>\wh 0}\lambda_X^n\cdot \mu(\wh 0,X).
\]
In particular, if $\nu$ is a probability on $\wh 0$, then
\[
	\|P^{\ast n}\nu-\pi\|_{TV}\leq -\sum_{X>\wh 0}\lambda_X^n\cdot \mu(\wh 0,X)
\]
where $\pi$ is the stationary distribution.
\end{corollary}
\begin{proof}
By \eqref{mobiussum} and using that $\lambda_{\wh 0}=1=\mu(\wh 0,\wh 0)$ and $P^{\ast n}(M_{\geq_{\mathscr L}L'})
=\lambda_{L'}^n$, we have that
\[
P(M\setminus \wh 0) = 1- P(\wh 0) = 1-\sum_{X\geq \wh 0} \lambda^n_X\cdot \mu(\wh 0,X)= -\sum_{X>\wh 0}\lambda_X^n\cdot \mu(\wh 0,X).
\]
\end{proof}

Note that this bound immediately implies that of Brown and Diaconis~\cite{BrownLRB,DiaconisBrown} for left regular band walks,
as well as the bound in~\cite{BHR} for hyperplane walks.

When $L$ does not consist of idempotents, computing $P^{\ast n}(M_{\geq_{\mathscr L}L})$ seems
to be challenging.

\subsection{Absorption times and mixing times}\label{subsection.absorb}

If $P$ is an adapted probability on an $\mathscr R$-trivial monoid $M$, then the right random walk on $M$ driven
by $P$ is absorbing with absorbing states the elements of the minimal ideal $\wh 0$.  Let $\tau$ be the random
variable which is the time that the random walk is absorbed into the minimal ideal.
Theorem~\ref{idempotentconvergence} essentially shows that $\tau$ is a \emph{strong stationary time}~\cite{Markovmixing}
for the random walk on $\wh 0$ driven by $P$ (or more generally, by Corollary~\ref{t:couplingfromthepast}, for any
ergodic random walk of $M$ on some set). More precisely, if $M$ acts transitively on $\Omega$, $P$ is an
adapted measure, $\nu$ is an initial probability on $\Omega$, and $\pi$ is the stationary distribution,
Corollary~\ref{t:couplingfromthepast} implies
\begin{equation}\label{stationary}
	\|P^{\ast n}\nu-\pi\|_{TV}=P^{\ast n}(M\setminus \wh 0)= \mathrm{Pr}\{\tau>n\}=\mathrm{Pr}\{\tau\geq n+1\}\,.
\end{equation}

As a consequence of our computations for left regular bands, we obtain the following.
\begin{theorem}
\label{expectedabsorption}
Let $M$ be a left regular band and $P$ an adapted probability on $M$.  Let $\tau$ be the absorption
time of the right random walk on $M$ driven by $P$, and let $\mu$ be the M\"obius function of
$\Lambda(M)$.  Then
\[
E[\tau] = -\sum_{X>\wh 0} \frac{1}{1-\lambda_X}\cdot \mu(\wh 0,X)\,,
\]
where $\lambda_X = \sum_{c(m)\geq X}P(m)$.
\end{theorem}
\begin{proof}
  Apply Corollary~\ref{lrbcaseofconvergencetime} using the standard
  fact about non-negative integer valued random variables
  (see~\cite{Markovmixing}) that the expected value of $\tau$ is given
  by
  \begin{equation}
    \label{expectation}
    E[\tau] = \sum_{n=0}^{\infty} \mathrm{Pr}\{\tau>n\} = \sum_{n=0}^{\infty}P^{\ast n}(M\setminus \wh 0)\,.\qedhere
  \end{equation}
\end{proof}

As an example, we obtain the usual formula for the expected waiting time for the coupon
collector problem, as well as the non-uniform version considered in~\cite{flajoletpaper}.

\begin{example}[Coupon collector]
Suppose we wish to collect $k$ different types of coupons.  With probability $p_i$ we
draw coupon $i$.  What is the expected number of draws to collect all $k$ coupons?
Let $\tau$ be the number of draws to collect all $k$ coupons.  Then $\tau$ is the absorption
time for the random walk on the join semilattice $P(\{1,\ldots,k\})$ driven by the adapted
probability $P(i)=p_i$.  For $I\subseteq \{1,\ldots, k\}$, let
\[
	\lambda_I = \sum_{i\in I} p_i.
\]
Then by Theorem~\ref{expectedabsorption} we retrieve the result of~\cite{flajoletpaper}:
\[
	E(\tau) = \sum_{I\subsetneq \{1,\ldots,k\}}(-1)^{k-|I|-1}\cdot \frac{1}{1-\lambda_I}\,.
\]

In particular, if $p_i=1/k$ for all $i$, this reduces to
\[
	E[\tau] = k\sum_{j=0}^{k-1}(-1)^{k-j-1}\binom{k}{j}\frac{1}{k-j}
	= k\sum_{q=1}^k(-1)^{q-1}\frac{1}{q}\binom{k}{q}
	=k\left[\sum_{i=1}^k\frac{1}{i}\right]\,,
\]
which is the standard computation for the coupon collector expectation.  One easily
obtains from this bound that
\[
  E[\tau]\leq k\log k+\gamma k+1/2+o(1)\,,
\]
where $\gamma$ is the Euler-Mascheroni constant.
\end{example}

As a consequence of Theorems~\ref{idempotentconvergence}
and~\ref{expectedabsorption}, we obtain the following bound on the
rate of convergence to stationarity for a random walk on an $\mathscr
R$-trivial monoid.
\begin{corollary}
\label{adaptedwalkmarkovbound}
Let $P$ be an adapted probability on an $\mathscr R$-trivial monoid $M$.  Let $\tau$ be
the absorption time of the right random walk on $M$, let $\nu$ be an initial distribution on $\wh 0$
and $\pi$ the stationary distribution. Then,
\[
	\|P^{\ast n}\nu -\pi\|_{TV}\leq \frac{1}{n+1}E[\tau].
\]
In particular, if $M$ is a left regular band, then
\[
	\|P^{\ast n}\nu -\pi\|_{TV}\leq -\frac{1}{n+1}\sum_{X>\wh 0}\frac{1}{1-\lambda_X}\cdot \mu(\wh 0, X).
\]
\end{corollary}

\begin{proof}
  Recall Markov's inequality~\cite{Markovmixing} for a non-negative
  discrete random variable $\tau$:
  \begin{equation*}
    \mathrm{Pr}\{\tau\geq a\}\leq \frac{1}{a}E[\tau]\,.
  \end{equation*}
  Using Theorem~\ref{idempotentconvergence} we then have
  \[
    \|P^{\ast n}\nu -\pi\|_{TV} = P^{\ast n}(M\setminus \wh 0) =
    \mathrm{Pr}\{\tau\geq n+1\} \leq \frac{1}{n+1}E[\tau]\,.
  \]
  Theorem~\ref{expectedabsorption} gives the second statement.
\end{proof}

\begin{example}[Tsetlin library]
Consider the Tsetlin library with $k$ books as a random walk on the free
left regular band on $\{1,\ldots,k\}$. We recall that the free left regular band on a set $A$ consists of all repetition-free words over $A$.
The product is concatenation followed by removing repetitions as you scan from left to right. A word belongs to the minimal ideal precisely
when it contains all letters.  Thus $\tau$ is the coupon collector random variable for $k$
coupons.  So if $p_i$ is the probability of selecting book $i$, then
\[
	\|P^{\ast n}\nu -\pi\|_{TV}\leq \frac{1}{n+1}\sum_{I\subsetneq \{1,\ldots,k\}}(-1)^{k-|I|-1}\cdot \frac{1}{1-\lambda_I}.
\]
In particular, if the weights are uniform, we recover the usual order $k\log k$ mixing time for the top-to-random shuffle.
\end{example}

\begin{example}[Promotion on a union of chains]
As our next example, let $j_1,\ldots, j_k\geq 1$ and let $M$ be the quotient of the free monoid on $x_1,\ldots, x_k$
by the relations which state that if $w$ is a word  with $j_i$ occurrences of $x_i$, then $wx_i=w$.  It is easy to
see that $M$ is a finite $\mathscr R$-trivial monoid.  The minimal ideal consists of those words with exactly
$j_i$ occurrences of $x_i$ for each $1\leq i\leq k$.

If we consider a probability $P$ supported on $x_1,\ldots, x_k$ with $P(x_i)=p_i$, then the random walk
on $\wh 0$ driven by $P$ admits the following description as a generalization of the Tsetlin library.  On a shelf one
has books $x_1,\ldots, x_k$ with $j_i$ copies of book $x_i$.  One chooses a book $x_i$ with probability $p_i$ and
moves the last copy of this book to the front.  This is a special case of the promotion random walk on a union of chains
considered in~\cite{ayyer_klee_schilling.2012}.

Note that the absorption time $\tau$ is the following well-studied variant of the coupon collecting problem,
see~\cite{coupongeneralized}.  As before, one has $k$ types of coupons
with different probabilities
of being chosen, but now one wants to collect $j_i$ copies of coupon $i$.  The expected value was
given in~\cite{coupongeneralized}. The result is
\begin{equation}
\label{gencoupeq}
	E[\tau]=\sum_{\emptyset\neq I\subset \{1,\ldots,k\}}(-1)^{|I|+1}\sum_{(r_i)\in
	\prod_{i\in I}\{0,1,\ldots,j_i-1\}}\frac{\sum_{i\in I}r_i\cdot \prod_{i\in I}p_i^{r_i}}{\left(\sum_{i\in I}p_i\right)^{1+\sum_{i\in I}r_i}}\,.
\end{equation}
It is not clear how useful this formula is for direct computation.  However, the case of uniform weights and an equal number
of copies of each book was studied earlier by Newmann and Shepp~\cite{NewmanShepp}.  A more precise
result was obtained by Erd\"os and R\'enyi~\cite{ErdosRenyi}.  If $j_1=\cdots=j_k=j$, then
\[
	E[\tau]=k\log k+(j-1)k\log\log k+k(\gamma-\log (j-1)!)+o(k)\,.
\]
Treating $j$ as a constant,  this gives a mixing time of order $k\log k+(j-1)k\log\log k$ for this generalized
Tsetlin library with equal multiplicities and uniform weights.
\end{example}

Our final result of the subsection gives a formula for the expected value of the absorption time for an arbitrary
$\mathscr R$-trivial monoid. However, this formula might be too cumbersome from a computational view point.

\begin{theorem}
Let $M$ be an $\mathscr R$-trivial monoid and $P$ an adapted probability on $M$.
Let $\tau$ be the absorption time of the right random walk on $M$ driven by $P$. Then,
retaining earlier notation,
\[
	E[\tau] = \sum_{\sigma\in \mathrm{St}(1)\cap \Delta(M\setminus \wh 0)} \frac{P(\sigma)}
	{\prod_{i=0}^{\dim \sigma}(1-\lambda_{d(\sigma_i)})}.
\]
\end{theorem}

\begin{proof}
This is immediate from Proposition~\ref{explicitcomputation} and \eqref{expectation}.
\end{proof}

As a consequence, we obtain the following bound on the mixing time for random walks on $\mathscr R$-trivial monoids.

\begin{corollary}
\label{corollary.mixing time}
Let $P$ be an adapted probability on an $\mathscr R$-trivial monoid $M$.  Let $\nu$ be a distribution
on the minimal ideal $\wh 0$ of $M$.  Let $\pi$ be the stationary distribution. Then
\[
	\|P^{\ast n}\nu -\pi\|_{TV}\leq \frac{1}{n+1}\sum_{\sigma\in \mathrm{St}(1)\cap
	\Delta(M\setminus \wh 0)} \frac{P(\sigma)}{\prod_{i=0}^{\dim \sigma}(1-\lambda_{d(\sigma_i)})}.
\]
\end{corollary}

\section{The \fhb}
\label{section.free tree monoid}

Let $X$ be a finite alphabet endowed with a total order $<$. The
\emph{\fhb} on $X$ is the monoid $\FHB(X)$ generated by $X$ subject to the
relations $x^2=x$ for $x\in X$, as well as $yxy=yx$ whenever $x<y\in
X$. We shall sometimes call quotients of $\FHB(X)$ (together with their distinguished ordered generating sets)
\emph{tree monoids} in this context.  Note that if $M$ is a tree monoid with respect to an ordered generating set
$X$ and $Y\subseteq X$ is considered with the induced order, then $\langle Y\rangle$ is a tree monoid with
respect to the generating set $Y$.

In this section we show that $\FHB(X)$ is $\RR$-trivial
(Corollary~\ref{corollary.fhb.r_trivial}), its combinatorics is governed by trees (Proposition~\ref{proposition.fhb.count}),
and that the lattice $\Lambda(\FHB(X))$ is the Boolean lattice~(Proposition~\ref{proposition.fhb.l_classes}).
In Section~\ref{subsection.fhb generalization}, we present a slight generalization, which does not
require the generators to be idempotent, but still yields an $\RR$-trivial monoid.

\subsection{Properties of the \fhb}

The defining relations of $\FHB(X)$ can be made into a length-reducing rewriting system in the
obvious way; this rewriting system is not necessarily confluent, meaning
that terms which can be rewritten in more than one way eventually yield the same result.
But it turns out that the Knuth--Bendix
completion terminates and the resulting system admits a nice
combinatorial description.

Formally speaking, a \emph{rewriting system} $R$ over an alphabet $X$ consists of a collection
of rules $\ell\to r$ with $\ell,r$ words over $X$. It is called \emph{length-reducing} if $|\ell|>|r|$ for
each rule $\ell\to r$.  If $u,v\in X^*$, then the one-step rewriting relation $u\Rightarrow_R v$ holds if there is a rule
$\ell\to r$ and a factorization $u=w\ell z$ with $v=wrz$.  One writes $\Rightarrow_R^*$ for the reflexive-transitive
closure of $\Rightarrow_R$.  The rewriting system $R$ is \emph{confluent} if $v\mathrel{\prescript{*}{R}\Leftarrow} u\Rightarrow_R^* w$
implies that there is a word $z$ such that $v\Rightarrow_R^* z\prescript{*}{R}\Leftarrow w$.  If the system is length-reducing,
it is enough to check that $v\mathrel{\prescript{}{R}\Leftarrow} u\Rightarrow_R w$ implies there is a word $z$ such that
$v\Rightarrow_R^* z\mathrel{\prescript{*}{R}\Leftarrow}w$. In fact, it is enough to check the case that the left hand sides  of the
two rules applied to obtain $v$ and $w$ from $u$ overlap.

A word $w$ is said to be \emph{reduced} with respect to $R$ (or \emph{irreducible}) if it contains no factor
which is the left hand side of a rule, i.e., it cannot be rewritten.  For a confluent, length-reducing rewriting
system, each word can be rewritten to a unique reduced word and each reduced word represents a distinct
element of the monoid with generating set $X$ and defining relations obtained by turning the rewriting
rules $\ell\to r$ into formal equalities $\ell =r$.  The Knuth--Bendix completion process is a way to take an
arbitrary rewriting system and complete it to a confluent one defining the same quotient monoid of the free
monoid $X^*$ (if the process terminates). See~\cite{ChurchRosser} for details.

The following proposition gives an inductive construction of the Knuth--Bendix completion of the rewriting
system defining $\FHB(X)$.

\begin{proposition}
  \label{proposition.fhb.kb}
  The Knuth--Bendix completion of the rewriting system $x^2\rightarrow x$
  and $yxy\rightarrow yx$ whenever $x<y$ over $X$ is given by the rewriting rules:
  \begin{align*}
    yuy\rightarrow yu & \quad
    \text{for $y\in X$ and $u$ a reduced word (possibly empty)}\\
    & \quad \text{in $\FHB(\{x\in X\suchthat x<y\})$.}
  \end{align*}
\end{proposition}
See Corollary~\ref{corollary.fhb.reduced_words} for an explicit
description of the reduced words and Remark~\ref{remark.fhb.kbcount}
for their number.
\begin{proof}
Let $R$ be the rewriting system consisting of the rules $x^2\to x$ and $yxy\to yx$ with $x<y$, for $x,y\in X$ and
let $R'$ be the rewriting system in the statement of the proposition.  Note that $R\subseteq R'$ because
the empty word and alphabet symbols are reduced with respect to $R'$.
We next show that the left and right hand sides of each rule of $R'$ are equal in the monoid
defined by the rewriting system $R$.  Indeed, if $u=u_1\cdots u_m$ is a  word with each $u_i<y$,
then a simple induction argument shows that $yu_1yu_2\cdots yu_m\Rightarrow_R^* yu_1\cdots u_m=yu$.
Thus $yuy\mathrel{\prescript{*}{R}\Leftarrow} yu_1yu_2\cdots yu_my\Rightarrow_R yu_1yu_2\cdots yu_m
\Rightarrow_R^* yu$
and so $yuy$ and $yu$ represent the same element of the monoid defined by $R$.

Let us take for $y$ the largest letter in $X$. Since the rewriting
rules in $R$ and $R'$ do not change the letters that appear in a word, we may assume
that the Knuth--Bendix completion for the alphabet $X\setminus\{y\}$ is as
given in the proposition, i.e., that $R'$ is confluent on $X\setminus \{y\}$. (Note that the base cases of $|X|\leq 1$ are trivial.)
We now apply a single step of the Knuth--Bendix completion after adding the relations involving $y$. The only left
hand sides that may overlap are of the form:
  \begin{itemize}
  \item $yuy$ with $yvy$ with $u$ and $v$ reduced words in
    $\FHB(X\setminus \{y\})$ (possibly empty). Suppose that $uv\Rightarrow^*_{R'} r$ with $r$ reduced over $X\setminus \{y\}$. Then we have
    \begin{displaymath}
      yry\mathrel{\prescript{*}{R'}\Leftarrow} yuvy \mathrel{\prescript{}{R'}\Leftarrow} yuyvy \Rightarrow_{R'} yuyv
      \Rightarrow_{R'} yuv\Rightarrow^*_{R'}yr\,.
    \end{displaymath}
    Since the rule $yry\rightarrow yr$ belongs to
    the rewriting system $R'$ over $X$, we have established confluence of $R'$.\qedhere
  \end{itemize}
\end{proof}

We remark that the empty word is reduced for any totally ordered alphabet and so, in particular,
$y^2\rightarrow y$ is a rewriting rule for any $y\in X$.

Note that, since the rewriting rules in the Knuth--Bendix completion
are strictly length-reducing, the two notions of a reduced word representing
an element $f$ are equivalent (i.e., words of minimal length representing $f$ are precisely those
that cannot be rewritten). In particular, each element $f\in \FHB(X)$ is represented by a
unique reduced word.

Given the form of the rewriting rules (all of the form $uy\to u$), we
obtain immediately the following description of the right Cayley graph.
\begin{corollary}
  \label{corollary.fhb.r_trivial}
  The right Cayley graph of $\FHB(X)$ is the prefix tree on the reduced
  words of its elements, with a loop $u\stackrel{i}{\rightarrow} u$
  whenever $ui$ is not a reduced word (see
  Figure~\ref{figure.fhb.r-graph}). In particular, $\FHB(X)$ is
  $\RR$-trivial and is Karnofsky--Rhodes with respect to $X$.
\end{corollary}
\begin{figure}[h]
  \centering
  \includegraphics{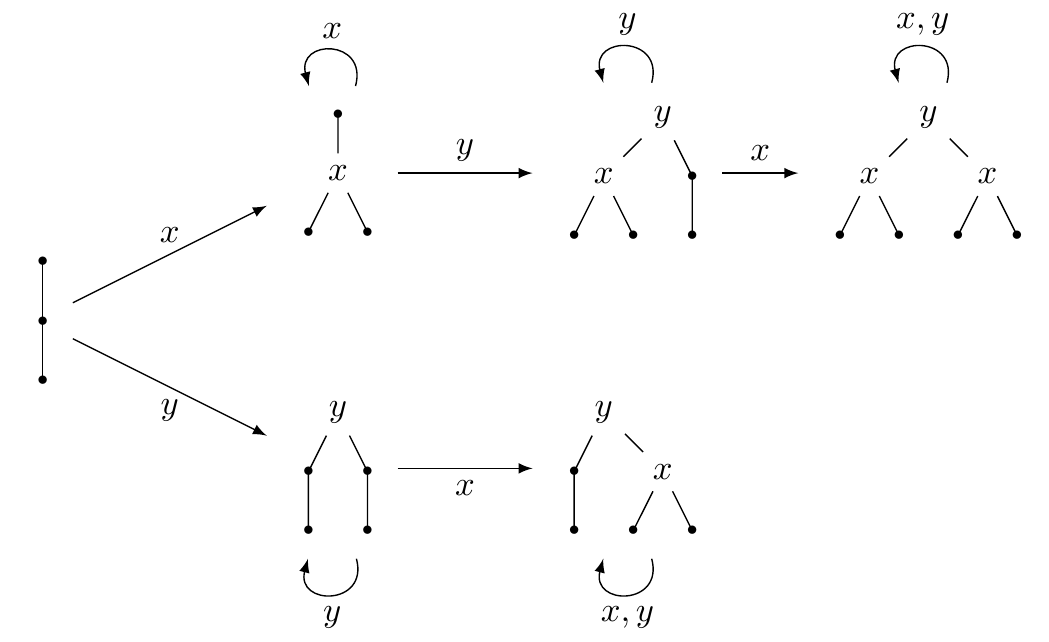}
  \caption{The right Cayley graph of the \fhb{} $\FHB(\{x<y\})$.}
  \label{figure.fhb.r-graph}
\end{figure}
\begin{figure}[h]
  \centering
  \includegraphics{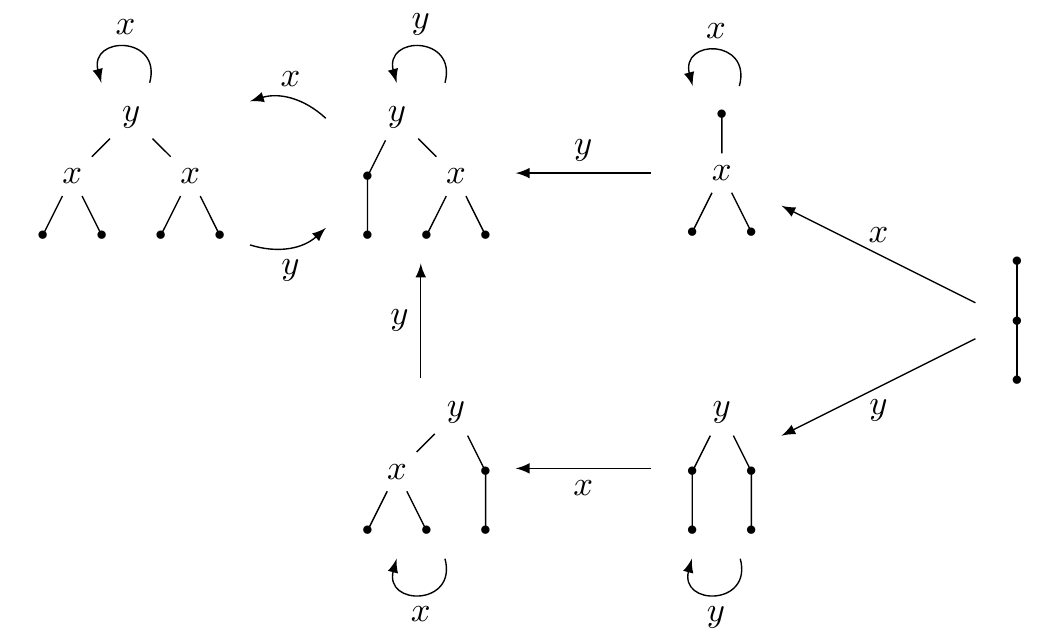}
  \caption{The left Cayley graph of the \fhb{} $\FHB(\{x<y\})$.}
  \label{figure.fhb.l-graph}
\end{figure}

Another immediate consequence is the following description of
the set of reduced words.
\begin{corollary}
  \label{corollary.fhb.reduced_words}
  A word $u$ is a reduced representative of an element of $\FHB(X)$ if and only
  if $u$ does not contain the largest letter $y$ of $X$ and is reduced in $\FHB(X\setminus \{y\})$, or $u$ has
  exactly one occurrence of $y$
  and the factorization $u=vyw$ according to $y$ gives recursively words $v,w$ that
  are reduced with respect to $\FHB(X\setminus \{y\})$.
\end{corollary}

Corollary~\ref{corollary.fhb.reduced_words} yields a recursive map
$\phi_{X}$ from reduced words of elements
of $\FHB(X)$ to trees. Namely, let $T_n$ be the set of ordered
unlabelled trees having nodes of out-degree 0,1,2 and such that all
leaves are at level $0$ while the root is at level $n$.
They are counted by the sequence $a(0)=1$ and
$a(n) = a(n-1)^2 + a(n-1)$ whose first terms are
\[
	1, 2, 6, 42, 1806,
	3263442, 10650056950806
\]
(see \#A007018 of~\cite{OEIS}).

Take now $u$ a reduced word. If $X$ (and therefore $u$) is empty, let
$\phi_X(u)$ be the tree in $T_0$ consisting of a single leaf. Otherwise, let
$x$ be the largest letter of $X$. If $x$ appears in $u$, write
$u=vxw$, where $v$ and $w$ belong to $X\setminus\{x\}$ and define
$\phi_X(u)$ as the tree, where the root has two subtrees
$\phi_{X\setminus \{x\}}(v)$ and $\phi_{X\setminus \{x\}}(w)$ in this
order. Otherwise, define $\phi_X(u)$ as the tree whose root has
$\phi_{X\setminus \{x\}}(u)$ as single subtree.

\begin{example}
  \label{example.fhb}
  Let $X=\{x_1,x_2,x_3,x_4\}$. Then,
  \begin{displaymath}
    \phi_X(x_3x_2x_4x_1x_2) =
    \raisebox{-7ex}{\includegraphics{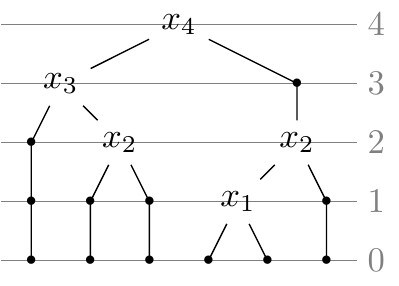}}\,,
  \end{displaymath}
  where, for ease of reading, we drew as additional information the
  generator corresponding to each inner node of out-degree $2$.

  Note that the number of leaves of the tree is given by the length of
  the word plus one.
\end{example}

\begin{proposition}
  \label{proposition.fhb.count}
  The map $\phi_X$ is a bijection between the elements of the \fhb{}
  $\FHB(X)$ and the trees in $T_{|X|}$.
\end{proposition}

\begin{remark}
  \label{remark.fhb.kbcount}
  The number of rules in the Knuth--Bendix completion for $\FHB(X)$ is
  given by $|X|+a(1)+\cdots+a(|X|-1)$.
\end{remark}

\begin{remark}
  \label{remark.leftfix}
  Let $e$ be an idempotent of $\FHB(X)$. Then $e$ fixes $u$ on the
  left, that is, $eu=u$, if and only if the reduced word of $e$ is a
  prefix of that of $u$ (this is an immediate consequence of the right
  Cayley graph being the prefix tree on reduced words, see Corollary~\ref{corollary.fhb.r_trivial}).
\end{remark}

\begin{remark}
  \label{remark.parabolicsubmonoid}
  If $Y\subseteq X$, then the submonoid of $\FHB(X)$ generated by $Y$ is clearly $\FHB(Y)$ with the
  induced ordering because
  the right hand side of each rule in Proposition~\ref{proposition.fhb.kb} has the same set of letters as the left hand side.
\end{remark}

\begin{proposition}\label{p:descentsfhb}
  Take $X=\{x_1<\dots<x_n\}$ and $t=\phi_X(u)$, where $u$ is an
  element of $\FHB(X)$. Then, $i$ is a right descent for $u$ (that is,
  $ux_i=u$) if and only if the $i^{th}$ inner node on the branch from the rightmost
  leaf to the root has out-degree $2$.

  Furthermore, $i$ is a left descent for $u$ if and only if the unique
  reduced word for $u$ starts with $i$ or, equivalently, the leftmost
  node of out-degree $2$ in $t$ is of height $i$.
\end{proposition}
\begin{proof}
  Looking at the completed rewriting system, we see that $ux_i=u$ if
  and only if $u$ admits a suffix of the form $x_iv$ with $v$ in
  $\FHB(\{x_1<\dots<x_{i-1}\})$. From the recursive definition of
  $\phi_X$, this is equivalent to the desired condition on $t$.

  For left descents this is an immediate consequence of
  Remark~\ref{remark.leftfix}.
\end{proof}

For example, $x_3x_2x_4x_1x_2$
  has $2$ and $4$ as right descents (see Example~\ref{example.fhb}).

For $u\in \FHB(\{x_1<\dots<x_n\})$, denote by $D_L(u)$
and $D_R(u)$ the set of left and right descents of $u$, respectively. For
example,
\[
	D_R(x_3x_2x_4x_1x_2)=\{2,4\} \quad \text{and} \quad D_L(x_3x_2x_4x_1x_2)=\{3\}.
\]

For $I\subseteq \{1,\dots,n\}$, define the \emph{right descent class}
indexed by $I$ as
\begin{displaymath}
  \FHB(X)^I=\{u\in \FHB(X) \suchthat D_R(u)=I\}\;.
\end{displaymath}

\begin{proposition}\label{proposition.countLclasssize}
  The size of the right descent class $\FHB(X)^I$ is given by
  $\prod_{i\in I} a(i-1)$. In particular, the minimal ideal of $\FHB(X)$
  is of cardinality $a(1)\cdots a(n-1)$.
\end{proposition}

\begin{proof}
  Any tree in $\FHB(X)^I$ can be constructed in a unique way by
  starting with a straight branch of length $n$ and, for each $i\in
  I$, grafting some subtree in $T_{i-1}$ on the left of the $i^{th}$
  inner node of the branch. The tree in Example~\ref{example.fhb} is
  obtained by grafting $\phi_{\{x_1\}}(x_1)\in T_1$ on the second
  inner node and $\phi_{\{x_1<x_2<x_3\}}(x_3x_2)\in T_3$ on the
  fourth.

  Formally, we prove this by induction on $|I|$.  If $I=\emptyset$, then $\FHB(X)^I$ consists
  of just the empty word.  Else, let $j\in I$ be maximal and let $I'=I\setminus \{j\}$.  From the
  proof of Proposition~\ref{p:descentsfhb}, we see that $D_R(w)=I$ if and only if the reduced
  form of $w$ is $ux_jv$ where $u,v$ are reduced words in the alphabet $\{x\mid x<x_j\}$ and
  $D_R(v)=I'$.  Thus there are $a(j-1)\cdot \prod_{i\in I'}a(i-1)$ elements in the descent class of $I$ by induction.

  The final statement, follows because the minimal ideal is the descent class $\FHB(X)^X$.
\end{proof}

\begin{proposition}
  \label{proposition.fhb.l_classes}
  The lattice $\Lambda(\FHB(X))$ is isomorphic to the power set $P(X)$ ordered by reverse
  inclusion (and so the monoid operation is union).  More precisely, the isomorphism sends the
  principal ideal $\FHB(X)e$ generated by an idempotent $e$ to the set of letters appearing in the
  reduced word representing $e$. Consequently, each subset $I=\{i_1<\cdots<i_\ell\}\subseteq\{1,\dots,n\}$ of $X$
  contributes one element to $\Lambda(\FHB(X))$, namely the principal ideal generated by the idempotent
  $e_I=x_{i_\ell}\cdots x_{i_1}$. This corresponding $\LL$-class is the minimal ideal of
  $\FHB(I)$ (viewed as a submonoid of $\FHB(X)$ via Remark~\ref{remark.parabolicsubmonoid}) and is of cardinality
  $a(1)\cdots a(|I|-1)$.
\end{proposition}

\begin{proof}
Since the singletons $\{x\}$ with $x\in X$ generate $P(X)$ and satisfy the relations of $\FHB(X)$, we have
a surjective homomorphism $f\colon \FHB(X)\to P(X)$ sending $x$ to $\{x\}$.  It is well known (cf.~\cite{MargolisSteinbergQuiver})
that any homomorphism from an $\mathscr R$-trivial monoid to a semilattice factors through $c$, so we
have that $f$ induces a surjective homomorphism $f'\colon \Lambda(\FHB(X))\to P(X)$.  Since $c(X)$
generates $\Lambda(\FHB(X))$  and $P(X)$ is a free semilattice with identity on $X$, we conclude that $f'$ is an
isomorphism.  The remaining statements follow easily. For example, Proposition~\ref{proposition.countLclasssize}
gives the cardinality of the $\mathscr L$-class associated to $I$.  Also $e_I$ is idempotent by a simple induction
argument of $|I|$ because if $I'=I\setminus \{i_\ell\}$, $e_I=x_\ell e_{I'}$ and hence
$e_Ie_I=x_\ell e_{I'}x_\ell e_{I'}= x_{\ell}e_{I'}e_{I'}=x_{\ell}e_{I'}$ where the penultimate equality uses that the
alphabet of $e_{I'}$ consists of symbols smaller than $x_\ell$ and the last equality uses induction.
\end{proof}

Note that under the isomorphism of $\Lambda(\FHB(X))$ and $P(X)$ we have that $d(u) = D_R(u)$ for $u\in \FHB(X)$.

Our next result shows that $\FHB(X)$ satisfies the conditions of Corollary~\ref{c:generic}.  Thus random walks of
$\FHB(X)$ on finite sets have diagonalizable transition matrices when driven by generic probabilities.  Several
such models will be considered in the subsequent sections.

\begin{proposition}\label{p:todiagonalizability}
Suppose that $u>_{\mathscr R} v$ in $\FHB(X)$.  Then $d(u)\neq d(v)$.  Consequently,  the transition matrix of
any random walk of $\FHB(X)$ on a finite set driven by a probability $P$ is diagonalizable as long as the partial sums
$\sum_{x\in I}P(x)$ are distinct for distinct subsets of $X$.
\end{proposition}

\begin{proof}
We prove the equivalent assertion that $D_R(u)\neq D_R(v)$.  Let $X=\{x_1,\ldots, x_n\}$ with $x_1<x_2<\cdots<x_n$.
We can identify subsets of $X$ with bit strings of length $n$ by setting, for $I\subseteq X$, $w_I=w_1\cdots w_n$
where $w_i=1$ if $i\in I$ and $w_i=0$, otherwise. We order bit strings by reverse lexicographical order (that is, by
least significant bit).  We claim that if $u>_{\mathscr R} v$, then $w_{D_R(u)}< w_{D_R(v)}$.  Since $\geq_{\mathscr R}$
is the prefix ordering, it suffices by induction to prove the assertion when $v=ux_i$ with $x_i\in X$.  The fact that
$u\neq v$ implies $x_i\notin D_R(u)$; on the other hand $x_i\in D_R(v)$.  We claim that if $j>i$, then $x_j\in D_R(u)$
if and only if  $x_j\in D_R(v)$.  It will then follow that $w_{D_R(u)}< w_{D_R(v)}$.

By the proof of Proposition~\ref{p:descentsfhb} we have that if $x_j\in D_R(v)$, then $v=ax_jb$ with $a,b$ reduced and
$b\in \FHB(\{x_1,\ldots, x_{j-1}\})$.  But then $b=b'x_i$ and $u=ax_jb'$ with $b'\in \FHB(\{x_1,\ldots, x_{j-1}\})$.  Thus
$x_j\in D_R(u)$.  Conversely, if $x_j\in D_R(u)$ then $u=ax_jb$ where $a,b$ are reduced and $b\in \FHB(\{x_1,\ldots, x_{j-1}\})$.
Then $v=ax_jbx_i$ and $bx_i\in \FHB(\{x_1,\ldots, x_{j-1}\})$ because $i<j$.  Thus $x_j\in D_R(v)$.  This completes the proof
of the first statement.  The second statement is immediate from Corollary~\ref{c:generic}.
\end{proof}

\subsection{Generalized tree monoids}
\label{subsection.fhb generalization}
Here we define a slight generalization of tree monoids by relaxing the idempotency condition
on the generators, which still admits an analogue of Corollary~\ref{corollary.fhb.r_trivial}.

\begin{definition}[Generalized tree monoid]
 Let $M$ be a monoid generated by elements in $X$ and let $<_X$ be a total order on $X$.
 Assume that for each generator $x\in X$, $x^{k+1}=x^k$ for some $k$.
 Furthermore, suppose that whenever $x<_Xy$ for $x,y\in X$, either $x$ and $y$ commute or
 $y$ is idempotent and $yxy=yx$. Then $M$ is a called a {\em generalized tree monoid}.
\end{definition}

The following proposition, establishing the $\RR$-triviality of generalized tree monoids, is proved via the same idea as
Proposition~\ref{p:todiagonalizability}.

\begin{proposition}
\label{proposition.fhb generalization}
Let $M$ be a generalized tree monoid. Then $M$ is $\RR$-trivial.
\end{proposition}

\begin{proof}
  The proof proceeds by defining a statistic $f(m)$ on monoid elements
  that increases strictly, for some appropriate order, along the non-trivial edges
  of the right Cayley graph, which implies $\RR$-triviality.

  Fix $x\in X$ and let $k$ be minimal such that $x^{k+1}=x^k$. For
  $m \in M$, define $f_x(m)$ as the largest integer
  $\leq k$ such that $m = m'x^{f_x(m)}$ for some $m'\in M$. Writing the
  elements of $X$ as $x_1>_X\dots>_X x_n$, associate to each element
  $m$ of the monoid the vector $f(m)=(f_{x_1}(m),\dots,f_{x_n}(m))$.
  When all the generators are idempotent, $f(m)$ is nothing but $\{x\in X \mid mx=m\}$, written as a binary
  vector. We use lexicographic order $<_\lex$ to compare vectors.

  Take $m\in M$ and $x\in X$ such that $mx\ne m$. We want to compare $f(m)$
  and $f(mx)$. Note that $f_x(m)<f_x(mx)$. Take $x<_X y$ in $X$. If
  $f_y(m)=0$, then trivially $f_y(mx)\ge f_y(m)$. Hence we may assume
  that $1\leq f_y(m)$. If $x$ and $y$ commute, then $f_y(mx)\geq
  f_y(m)$. Otherwise, $y$ is idempotent and $yxy = yx$. Since $y$ idempotent implies $my=m$, it follows
  that $mxy = m yxy =m yx = m x$ and thus $f_y(mx)=f_y(m)$ (which is
  1 since $y$ is idempotent).

  We conclude that $f(m)<_\lex f(mx)$, as desired.  It follows that the right Cayley digraph of $M$ is acyclic
  and hence $M$ is $\mathscr R$-trivial.
\end{proof}

\section{Toom-Tsetlin model}
\label{section.toom}

In statistical physics, the Ising model has been repeatedly studied
from several different points of view because of its inherent simplicity
and yet complex behavior. The two-dimensional Ising model is particularly
interesting because of its exact solution. The Toom model~\cite{toom1980}
is a dynamical variant of the two-dimensional
Ising model designed to study interface growth at low temperatures.

In the model, one considers Ising spins $\pm 1$ on a Cartesian lattice, which are simultaneously updated according
to the following rule: the spin at location $(i,j)$ gets updated to the majority of the spins at
$(i,j), (i,j+1)$ and $(i+1,j)$ with probability $1-p-q$, to $+1$ with probability $p$, and
to $-1$ with probability $q$. This model was considered~\cite{dlss1991a,dlss1991b} in the third
quadrant with the boundary condition that spins on the negative $x$-axis are $+1$ and
spins on the negative $y$-axis are $-1$. In the stationary state for small $p,q$, an interface is formed between 
the $+1$ and $-1$ spins which is a straight line starting at the origin at an angle 
depending on the ``noise'' parameters $p$ and $q$. On the
interface itself, there is a nonzero density of both spins, and the dynamics of the spins on
the interface is often also referred to as the Toom model.

A spin exchange model was proposed in~\cite{jnr1996} in order to understand
the border process of the Toom model. This model was defined
on the semi-infinite integer lattice whose finite analog we study here.

We generalize the model by considering both finite lattice sizes as well and arbitrary
particle numbers rather than just spins $\pm 1$. We find that this generalization has a remarkable connection to
another field of probability, namely the well-studied Tsetlin library~\cite{hendricks1,Fill1,Fill2,BHR}.
The Tsetlin library is a discrete-time Markov chain on permutations of books arranged in a line, where each book
$b_i$ is picked with probability $x_i$ and placed in the front of the line. The stationary
distribution of the Tsetlin library and the eigenvalues of the transition matrices are known
explicitly. There are also tight bounds on the mixing time of the Markov chain.

We consider two generalizations of the Tsetlin library involving multiple books. The first one (see
Section~\ref{subsection.toom multiperm}) with a fixed number of books of certain types,
is a Markov chain on words with fixed content. The second (see Section~\ref{subsection.toom interlibrary}) has a 
natural interpretation in terms of a library with ``interlibrary loan''.
This is a Markov chain on words of fixed length from a given alphabet but not of fixed content.

Let $\mathcal{B} = \{b_1,\dots,b_m\}$ be the alphabet, or equivalently the set of books
in the library. We consider words in $\mathcal{B}$ of length $L$.
Our probability parameters are $x_{b,k}$, for $b \in \B$ and $k \in \{1,\dots,L\}$.  As is usual in
the context of the Tsetlin library, states are indexed by words in the alphabet $\B$ of length $L$.
In both variants, we will see that all eigenvalues of the transition matrices are simple
linear expressions in the parameters~$x_{b,k}$.

\subsection{First variant: Tsetlin library with multiple copies of books}
\label{subsection.toom multiperm}

Here we consider the model where there is a fixed number $n_i$ of books $b_i$,
so that the total number of books is $\sum_{i=1}^m n_i = L$. The system is thus defined by a
vector $\vec n \in \N^m$. The configurations can be indexed by words (or multipermutations)
$\pi = (\pi_1, \dots, \pi_L)$ of prescribed content with letters in $\mathcal{B}$; that is, each $\pi_j = b_k$ for some
$1\le k\le m$ and $\sum_{j=1}^L 1_{\{\pi_j = b_k\}} = n_k$.
There are therefore $\binom{L}{n_1,\ldots,n_m}$ configurations.

The dynamics is as follows. Suppose the current state is $\pi$. At each discrete time step, we choose with probability
$x_{b,j}$ a book $b$ and an index $j$ (no greater than the number of copies of $b$) and we move the $j^{th}$ copy
of $b$ to the left, past all books not equal to $b$, until it is next to the $(j-1)^{st}$ copy of $b$.  If $j=1$, we interpret this as
moving $b$ to the front. Formally, if the $j^{th}$ copy of $b$ is in position $k$ of $\pi$, then
the new state becomes $\pi'$ as follows:
\be \label{bulkrules}
\begin{split}
\pi = &(\pi_1, \dots, \pi_{k-1},b,\pi_{k+1},\dots,\pi_L)  \mapsto  \\
\pi' = &\begin{cases}
(b,\pi_1, \dots, \pi_{k-1},\pi_{k+1},\dots,\pi_L), & \text{if $j=1$}, \\
(\pi_1,\dots,\pi_{i-1},b,b,\pi_{i+1},\dots, \pi_{k-1},\pi_{k+1},\dots,\pi_L), & \\
& \hspace{-5cm} \text{if $j>1$, $\pi_i = b$ and $b \notin \{\pi_{i+1},\dots,\pi_{k-1} \}$.}
\end{cases}
\end{split}
\ee
We denote this map by $\toomgens_{b,j}$, or more precisely $\pi'=\toomgens_{b,j}(\pi)$.

When there is exactly one copy of each book, then this Markov chain is the classical Tsetlin library chain.
When $m=2$, this version of the Tsetlin library reduces to a finite analog of the Toom model~\cite{jnr1996},
when all the probabilities are equal. The model consists of Ising spins $\pm 1$ on the integer lattice
$\Z$, where the leftmost spin in a block of spins of type $+1$ or $-1$ hops far enough to the left so that
it becomes the rightmost spin in the next block of spins to its left. Another difference is that the Toom model
is studied in continuous time.

\begin{proposition}
  \label{proposition.irred variant 1}
The Markov chain on words of length $L$ of content $\vec n$ in the alphabet $\B$ defined by
the operators $\{\toomgens_{b,j}\; |\; b \in \B, \, 1 \leq j \leq n_b \}$ is ergodic.
\end{proposition}

\begin{proof}
The graph associated to the Markov chain is primitive because of the presence of self-loops, such as the
operator $\toomgens_{\pi_1,1}$ acting on $\pi$.

To prove irreducibility, we show that we can get from any configuration to a specified
configuration. It will be convenient to express the target configuration $\gamma$ in block
form. We canonically represent $\gamma$ as $\gamma_1 \cdots \gamma_k$, where each
$\gamma_i$ is a sequence of the same $b_i\in \mathcal B$ and consecutive blocks do not consist
of the same symbol. So $\gamma=b_1^{n_1}\cdots b_k^{n_k}$ where $b_i\neq b_{i+1}$ for $i=1,\ldots, k-1$ and $\gamma_i =b_i^{n_i}$.

We construct $\gamma$ by building it one block at a time from the right. For each $i$
let $\ell(i)$ be the total number of occurrences of $b_i$ in the
prefix of $\gamma$ up to and including $\gamma_i$, i.e., $\ell(i) = \sum_{j\leq i, b_j=b_i}n_j$ in the above notation.

We define the operator
$\bar{\toomgens}_{i}$ to be the operator $\toomgens_{b_i,\ell(i)}
\circ \cdots \circ \toomgens_{b_i,1}$, where we remind the reader that we are acting on the left. We
then claim that the sequence of operators
\be \label{irred-opers}
	\bar{\toomgens}_1 \circ \cdots \circ \bar{\toomgens}_k
\ee
acting on any configuration $\pi$ returns $\gamma$.  Indeed, a straightforward induction shows that
\[\bar{\toomgens}_i\circ \cdots \circ \bar{\toomgens}_k(\pi)= \gamma_i'\cdots \gamma_k'\pi'\]
where, for $i\leq j\leq k$, 
\[
	\gamma_j' = \begin{cases}\gamma_j, & \text{if there exists $i\leq r<j$ with $b_r=b_j$,}\\
	b_j^{\ell(j)}, & \text{else,}\end{cases}
\] 
and $\pi'$ is word of the appropriate content.  Using that if $\gamma_i$ is the leftmost occurrence
of a block of the symbol $b_i$ in $\gamma$, then $b_i^{\ell(b_i)}=\gamma_i$, it follows that
$\bar{\toomgens}_1 \circ \cdots \circ \bar{\toomgens}_k(\pi)=\gamma$.
\end{proof}

The transition matrix will be denoted by $T_{\vec n}$. To describe our main result, we need
to extend the notion of derangement from permutations to words. A word $\pi$ of content $\vec n$
is called a {\bf derangement} if no letter in $\pi$ is in a position
occupied by the same letter in the sequence
\begin{equation}\label{equation.derangement.def}
(1,\dots,1,2,\dots,2,\dots,m,\dots,m).
\end{equation}
For example $(3,2,1,1)$ is a derangement, whereas $(2,1,1)$ is not since the first 1 sits in the
same slot as a 1 in $(1,1,2)$.

Let $d_{\vec n}$ denote the number of derangements of words of content $\vec n$.
Even and Gillis \cite{even-gillis-1976} first gave an explicit formula for derangements of words
(or multipermutations) in terms of Laguerre polynomials $L_n(x)$,
\be \label{defder}
	d_{\vec n} = (-1)^L \int_0^\infty e^{-x} \prod_{j=1}^m L_{n_j}(x) \mathrm{d}x,
\ee
and Carlitz \cite{carlitz1978} gave the first combinatorial proof of this result.
For $1 \leq j \leq m$ and $I_j \subseteq [n_j]=\{1,2,\ldots,n_j\}$,
let $x_{b_j,I_j} = \sum_{s \in I_j} x_{b_j,s}$.

\begin{theorem}
  \label{theorem.toom.interval.eigenvalues}
The characteristic polynomial of the transition matrix $T_{\vec n}$ is given by
\be
| \lambda \iden \; - \; T_{\vec n} | = \prod_{I_1 \subseteq [n_1], \dots,
    I_m \subseteq [n_m]} \left(\lambda - \sum_{j=1}^m x_{b_j,I_j}
    \right)^{\ds d_{(n_1-|I_1|,  \dots, n_m-|I_m|)}}.
\ee
When $x_{b,k}=x_b$ for $b \in \B$ and all $k$, this simplifies to
\be
| \lambda \iden  \; - \; T_{\vec n} | =
\prod_{(k_1,\dots,k_m) \leq (n_1,\dots,n_m)}
\left(\lambda - \sum_{i=1}^m k_i x_{b_i} \right)^
{\ts d_{(n_1-k_1,  \dots, n_m-k_m)} \prod_{i=1}^m \binom{n_i}{k_i}},
\ee
where $\le$ is component-wise comparison.
\end{theorem}

We postpone the proof of Theorem~\ref{theorem.toom.interval.eigenvalues} to Section~\ref{subsection.Toom proofs}.

\begin{example}
The transition matrix for
$n_1=n_2=2$ in the lexicographically ordered basis is given by
\[
\left(
\begin {array}{cccccc}
x_{1,1} + x_{1,2}+ x_{2,2} &x_{{1,2}}&x_{{1,2}}&0&0&0\\
\noalign{\medskip}0&x_{1,1}&0&x_{{1,1}}&0&0\\
\noalign{\medskip}0&x_{{2,2}}&x_{1,1}+x_{2,2}&0&x_{{1,1}}&x_{{1,1}}\\
\noalign{\medskip}x_{{2,1}}&x_{{2,1}}&0&x_{1,2}+x_{2,1}&x_{1,2}&0\\
\noalign{\medskip}0&0&x_{{2,1}}&0&x_{2,1}&0\\
\noalign{\medskip}0&0&0&x_{{2,2}}&x_{{2,2}}&x_{1,1}+x_{1,2}+x_{2,2}
\end {array}
\right),
\]
and its eigenvalues are
\[
1=x_{1,1}+x_{1,2}+x_{2,1}+x_{2,2},\; x_{1,1}+x_{2,1},\; x_{1,1}+x_{2,2},\; x_{1,2}+x_{2,1},\;    x_{1,2}+x_{2,2},\;   0.
\]
When we set $x_{1,1}=x_{1,2}=x_1$ and $x_{2,1}=x_{2,2}=x_2$, we get the eigenvalues $2x_1+2x_2$ with multiplicity 1 and
$x_1+x_2$ with multiplicity 4 as expected.
\end{example}

\begin{corollary}
For the Toom model (i.e. when $m=2$), Theorem~\ref{theorem.toom.interval.eigenvalues} simplifies to
\[
	| \lambda \iden \; - \; T_{(n_1,n_2)} | = \prod_{\substack{I_1 \subseteq [n_1], I_2 \subseteq [n_2]\\
	n_1-|I_1| = n_2-|I_2|}}
         \left(\lambda - x_{b_1,I_1} - x_{b_2,I_2} \right).
\]
\end{corollary}

\begin{proof}
For two letters the number of derangements is zero unless there are the same number of each letter,
in which case the number of derangements is one.
\end{proof}

The next theorem provides diagonalizability of the transition matrix for generic probabilities.
\begin{theorem}\label{t.diagonal}
The transition matrix $T_{\vec n}$ is diagonalizable as long as the partial sums of the $x_{b,j}$ over
distinct subsets of indices are distinct.
\end{theorem}
The proof of Theorem~\ref{t.diagonal} is postponed until Section~\ref{subsection.Toom R trivial}.

\subsection{Second variant: Tsetlin library with interlibrary loan}
\label{subsection.toom interlibrary}

We generalize the Tsetlin library with multiple copies of  Section~\ref{subsection.toom multiperm} to
include storage or interlibrary loan of books. One imagines that the library can hold $L$ books, and there is
the possibility of borrowing copies of books from an external source (such as storage or another library). We
remark that this model also makes sense from the point of view of the Toom model where this model
has the interpretation of looking at a window of $L$ sites in the one-dimensional lattice.

Our state space is now all possible words of size $L$ in the alphabet $\B$ of size $m$ and the
number of configurations is $m^L$.
We need to define the operators giving rise to a random mapping representation of this Markov chain. With a slight abuse of
terminology, we will again denote
the operators by $\toomgens_{b,j}$ for $b \in \B$, but this time, for all $j \in [L]$. As before, the operator
$\toomgens_{b,j}$ is chosen with probability $x_{b,j}$.
Let $n_b(\pi)$ be the number of occurrences of $b$ in the word $\pi$.

Given a word $\pi$, the operator $\toomgens_{b,j}$ acts as follows. If there are at least $j$ copies of the
book $b$ in $\pi$, then (as before) we move the $j^{th}$ copy of $b$ to the left until it is next to the $(j-1)^{st}$
copy (where if $j=1$, then $b$ is moved to the front). If there are $j-1$ copies of $b$ in $\pi$, then we insert a
new copy of $b$ (from storage or another library) immediately after the $(j-1)^{st}$ copy of $b$.  Finally, if there are
strictly fewer than $j-1$ copies of $b$ in $\pi$, we do nothing. Formally, the transitions are defined by
\be \label{rbrules}
\begin{split}
\pi = &(\pi_1, \dots,\pi_L)  \mapsto  \\
\pi' = &\begin{cases}
\eqref{bulkrules}, & \hspace{-3cm}\text{if $j \leq n_b(\pi)$}, \\
(b,\pi_1, \dots,\pi_{L-1}), & \\
& \hspace{-3cm}\text{if $n_b(\pi)=0$ and $j=1$}, \\
(\pi_1,\dots,\pi_{i-1},b,b,\pi_{i+1},\dots,\pi_{L-1}), & \\
& \hspace{-3cm} \text{if $n_b(\pi)>0, j=n_b(\pi)+1$}, \\
& \hspace{-3cm} \text{$\pi_i = b$ and $b \notin \{\pi_{i+1},\dots,\pi_{L} \}$}, \\
\pi, & \hspace{-3cm} \text{otherwise}.
\end{cases}
\end{split}
\ee

The loan operators in~\eqref{rbrules} are natural extensions of the operators in~\eqref{bulkrules} because
one imagines that a book from somewhere far to the right will jump far enough left so that it becomes the
rightmost book in the rightmost block of books of the same type.
Notice that $\pi$ is fixed by the operator $\toomgens_{\pi_L,n_{\pi_L}(\pi)}$.

We require $x_{b,j}$ to be positive for all $b \in \B$ and $1 \leq j \leq L$.

\begin{proposition}
  \label{proposition.irred variant 2}
The Markov chain on words of length $L$ in the alphabet $\B$ of $m$ letters defined by the
operators $\{\toomgens_{b,j}\; |\; b \in \B, \, 1 \leq j \leq L \}$ is ergodic.
\end{proposition}
\begin{proof}
Just as in the Markov chain of the Tsetlin library with multiple copies, the graph of the chain is primitive
because of the presence of self-loops. Since the operators in the former chain are a
subset of the operators here, all the self-loops there also occur here.

To show irreducibility, we again construct a series of operators that take any
configuration to a prescribed one, say $\gamma$. By the proof of
Proposition~\ref{proposition.irred variant 1}, it suffices to construct an operator that will
take any configuration to one with the same content as $\gamma$.

Suppose $\gamma$ has content $(n_1,\dots,n_m)$. Then the sequence of operators
\[
(\toomgens_{1,n_1} \circ \cdots \circ  \toomgens_{1,1}) \circ\cdots\circ
(\toomgens_{m,n_m} \circ \cdots \circ  \toomgens_{m,1})
\]
takes any configuration to $b_1^{n_1} \cdots b_m^{n_m}$, which has the same content as
$\gamma$. (Recall that the operators act on the left).
The operator \eqref{irred-opers} constructed in Proposition~\ref{proposition.irred variant 1} will then
take this configuration to $\gamma$.
\end{proof}

We denote the transition matrix for this model by $T_{m,L}$.

\begin{theorem}
\label{theorem.toom.window.eigenvalues}
  The characteristic polynomial of the transition matrix $T_{m,L}$ is given by
  \begin{equation*}
     | \lambda \iden \; - \; T_{m,L} | = \;\; \left(\lambda - \sum_{j=1}^m  x_{b_j,[L]} \right)
       \prod_{I_1,\ldots,I_m\subsetneq [L]}
      \left( \lambda - \sum_{j=1}^m x_{b_j,I_{j}} \right)^{m_{\vec{I}}}\;,
  \end{equation*}
where the multiplicity $m_{\vec{I}}$ for $\vec{I}=(I_1,\ldots,I_m)$ is given in~\eqref{equation.multiplicity interlibrary} below.
\end{theorem}

The proof of Theorem~\ref{theorem.toom.window.eigenvalues} is postponed to Section~\ref{subsection.Toom proofs}.
We conjecture that the multiplicities $m_{\vec{I}}$ are again given by
derangement numbers of words as in~\eqref{defder}.

\begin{conjecture}
For $\vec{I}$ with $I_i \subsetneq [L]$ for all $1\le i\le m$ we have
\[
	m_{\vec{I}} = \begin{cases}
	(m-1)\; d_{(|\bar{I}_1|-1, \dots, |\bar{I}_m|-1)}, & \text{if $\sum_i \max(\bar I_i) \leq L+m-1$,}\\
	0, & \text{otherwise,}
	\end{cases}
\]
where $\bar I = [L] \setminus I$ and $\max(I)$ is the maximal element of $I \subseteq [L]$.
\end{conjecture}

\begin{example}
The transition matrix for $L=2$ and $m=2$ in the lexicographically
ordered basis is given by
\[
\left(
\begin {array}{cccc} x_{1,1}+x_{1,2}+x_{2,2}& x_{1,2}&0&0\\
\noalign{\medskip}0&x_{1,1}+x_{2,2}&x_{1,1}&x_{1,1}\\
\noalign{\medskip}x_{2,1}&x_{2,1}&x_{1,2}+x_{2,1}&0\\
\noalign{\medskip}0&0&x_{2,2}&x_{1,2}+x_{2,1}+x_{2,2}
\end {array}
\right),
\]
and its eigenvalues are
\[
1 = x_{1,1}+x_{1,2}+x_{2,1}+x_{2,2}, \quad x_{1,1}+x_{2,2}, \quad x_{1,2}+x_{2,2},  \quad x_{1,2}+x_{2,1}\;,
\]
as expected by the statement of Theorem~\ref{theorem.toom.window.eigenvalues}.
\end{example}

Again we have diagonalizability of the transition matrix for generic probabilities.
\begin{theorem}\label{t.diagonalv2}
The transition matrix $T_{m,L}$ is diagonalizable as long as the partial sums of the $x_{b,j}$ over distinct
subsets of indices are distinct.
\end{theorem}
The proof of Theorem~\ref{t.diagonalv2} is postponed until Section~\ref{subsection.Toom R trivial}.

\subsection{$\RR$-triviality of the Toom--Tsetlin model}
\label{subsection.Toom R trivial}
Let $\vec n=(n_1,\ldots, n_m)$ be in $\mathbb N^m$ with $n_1+\cdots+n_m=L$ and put $\mathcal B=\{b_1,\ldots, b_m\}$.
Set $X_{\vec n}=\{\toomgens_{b_i,k} \mid b_i\in \mathcal B, 1\le k \le n_i\}$ where the $\toomgens_{b_i,k}$ are the mappings
associated to the Toom--Tsetlin model from Section~\ref{subsection.toom multiperm}.

\begin{lemma}
\label{lemma.toom relation}
Each $x\in X_{\vec n}$ is idempotent.  Moreover, we have
\[
	yxy =yx \quad \text{for all $x,y\in X_{\vec n}$}
\]
unless $x=\toomgens_{b,i+1}$ and $y=\toomgens_{b,i}$ for some $b=b_k\in \mathcal B$ and $1\le i < n_k$.
\end{lemma}

\begin{proof}
It is clear that each element of $X_{\vec n}$ is idempotent from the definition.  Let $x,y\in X_{\vec n}$.
Note that when $y=\toomgens_{b,i}$ and $x=\toomgens_{b,j}$ with
  $j>i+1$ or $j<i-1$, then $x$ and $y$ commute.  Indeed, $xy$ and $yx$ both have the effect of placing the
  $i^{th}$ and $j^{th}$ copies of $b$ immediately after the $(i-1)^{st}$ and $(j-1)^{st}$ copies of $b$, respectively
  (where this should be interpreted appropriately if $i$ or $j$ is $1$).
  Thus $yxy=yyx=yx$ since $y$ is idempotent.

  Suppose now that $j=i-1$ and let $w\in\mathcal B^L$ have content $\vec n$.   Assume first that $i>2$ and write $w=u_1\,b\,u_2\,b\,u_3\,b\,u_4$,
  where the leftmost $b$ is the $(i-2)^{nd}$ $b$ of $w$ and $b$ does not appear in $u_2,u_3$. Then
  \begin{multline*}
      yxy(u_1\,b\,u_2\,b\,u_3\,b\,u_4) = yx(u_1\,b\,u_2\,b\,b\,u_3\,u_4) = y(u_1\,b\,b\,u_2\,b\,u_3\,u_4)\\
      = u_1\,b^3\,u_2\,u_3\,u_4 = y(u_1\,b^2\,u_2\,u_3\,b\,u_4) = yx(u_1\,b\,u_2\,b\,u_3\,b\,u_4)\,.
  \end{multline*}
  If $i=2$ and $j=1$, and if $w=u_1\, b\, u_2\, b\, u_3$, where $u_1$ and $u_2$ do not contain any $b$s, then $yxy(w)=b^2\,u_1\,u_2\,u_3=yx(w)$.

  Next assume $x=\toomgens_{b,i}$ and $y=\toomgens_{b',j}$ with $b'\neq b$.  We claim that $xy=yx$
  (and hence $yxy=yx$) unless $i=1=j$.  For instance, if neither $i$ nor $j$ is $1$, then applying both operators
  in either order puts the $i^{th}$ copy of $b$ immediately after the $(i-1)^{st}$ copy and the $j^{th}$ copy of $b'$
  immediately after the  $(j-1)^{st}$ copy of $b'$ while preserving the relative order of all remaining books. The situation is similar when
  exactly one of $i,j$ is $1$:
  one book goes to the front and the other immediately after its predecessor of the same type. Trivially, if $i=1=j$ then
  $yxy$ and $yx$ both move the first copy of $b'$ to the front and the first copy of $b$ into the  second position.
\end{proof}

Note that Lemma~\ref{lemma.toom relation} implies that $X_{\vec n}$ generates a tree monoid.

\begin{theorem} \label{theorem.toom R trivial}
The monoid $\Monoid_{\vec n}$ generated by $X_{\vec n}$ is a tree monoid (with respect to an appropriate ordering
on $X_{\vec n}$) and hence $\RR$-trivial.
\end{theorem}

\begin{proof}
By Lemma~\ref{lemma.toom relation} we can view $\Monoid_{X_{\vec n}}$ as a tree monoid by choosing a topological
sorting of the partial order $<_{X_{\vec n}}$ on $X_{\vec n}$ defined by
$\toomgens_{a,i} <_{X_{\vec{n}}} \toomgens_{b,j}$ if $a=b$ and $i<j$.
Corollary~\ref{corollary.fhb.r_trivial} then provides the $\RR$-triviality of $\Monoid_{X_{\vec n}}$.
\end{proof}

Now let $X_I=\{\toomgens_{b,i}\mid b\in \mathcal B, 1\leq k\leq L\}$ where the $\toomgens_{b,k}$ are the mappings corresponding
to the Toom-Tsetlin library with interlibrary loan from Section~\ref{subsection.toom interlibrary}.  (Here the subscript $I$ in $X_I$ stands for interlibrary.)

\begin{corollary}\label{c:treemonoidforinterval}
The monoid $\Monoid_{X_I}$ generated by $X_I$ is a tree monoid (with respect to an appropriate ordering
on $X_I$) and hence $\RR$-trivial.
\end{corollary}

\begin{proof}
Let $\vec n=(L,L,\ldots, L)\in \mathbb N^m$.  Let $\Omega\subseteq \mathcal B^{mL}$ consist of the words of content $\vec n$ (i.e., those
words with exactly $L$ occurrences of each letter).  Define a surjective mapping $\pi_L\colon \Omega\to \mathcal B^L$ by putting
$\pi_L(w_1\cdots w_{mL})=w_1\cdots w_L$.  If $w\in \Omega$, $b\in \mathcal B$ and $1\leq k\leq L$, then it is immediate from the definitions that
\begin{equation}\label{e:asquotient}
	\pi_L(\toomgens_{b,k}(w))=\toomgens_{b,k}(\pi_L(w))\;,
\end{equation}
where $\toomgens_{b,k}$ on the left hand side of \eqref{e:asquotient} is seen as an element of $X_{\vec n}$ and $\toomgens_{b,k}$
on the right hand side is seen as an element of $X_I$.  Because $\pi_L$ is surjective, it follows that $\Monoid_{X_I}$ is a quotient of
$\Monoid_{\vec n}$ and hence is a tree monoid.
\end{proof}

A picture of the right Cayley graph for the Toom--Tsetlin model with interlibrary loan for $L=2$ is
shown in Figure~\ref{figure.toom cayley}.
\begin{figure}
\begin{center}
  \includegraphics[width=\textwidth]{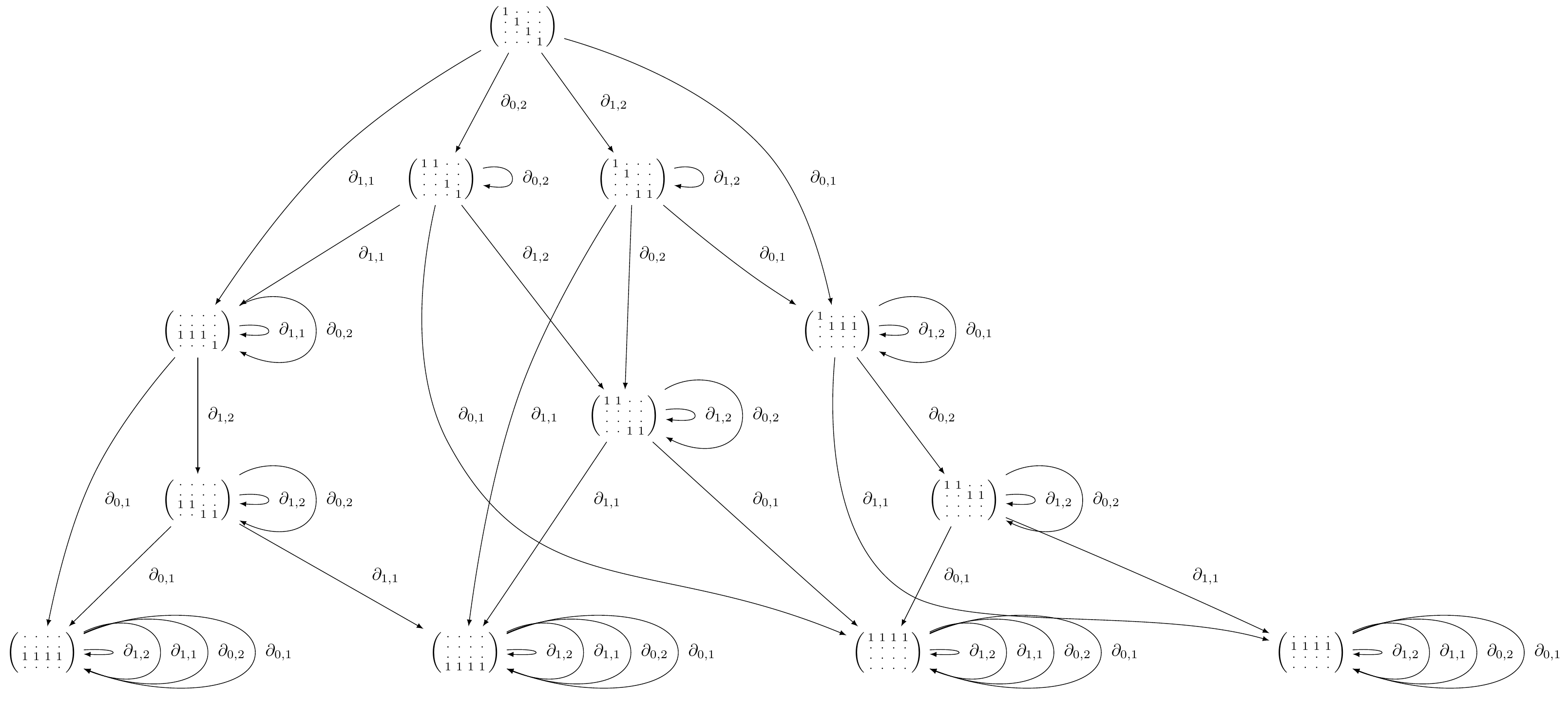}
\end{center}
\caption{The right Cayley graph for the Toom--Tsetlin model with
  interlibrary loan for $L=2$. Each element of the monoid is
  represented by the matrix of its action on the four bookshelves
  $(0, 0), (0, 1), (1, 0), (1, 1)$.
  \label{figure.toom cayley}}
\end{figure}

Theorems~\ref{t.diagonal} and~\ref{t.diagonalv2} are now immediate consequences of
Theorem~\ref{theorem.toom R trivial}, Corollary~\ref{c:treemonoidforinterval}, and Proposition~\ref{p:todiagonalizability}.

\subsection{Proof of Theorems~\ref{theorem.toom.interval.eigenvalues} and~\ref{theorem.toom.window.eigenvalues}}
\label{subsection.Toom proofs}

Finally we turn to the proof of Theorems~\ref{theorem.toom.interval.eigenvalues} and~\ref{theorem.toom.window.eigenvalues}.
We begin with a lemma generalizing a standard fact about usual derangements. For a vector $\vec{n}=(n_1,\ldots, n_m)$
with non-negative integer entries we denote by
\[
	\vec{n}! = \binom{\sum_{i=1}^m n_i}{n_1,\ldots,n_m} = \frac{\big(\sum_{i=1}^m n_i\big)!}{n_1!\cdots n_m!}
\]
the multinomial coefficient. When $\vec{n}$ contains negative entries, we set $\vec{n}!=0$.

\begin{lemma}\label{lemma.count.derangement}
Let $\vec{n}=(n_1,\ldots, n_m) \in \mathbb N^m$. We order $m$-tuples $\vec R=(R_1,\ldots, R_m)$ of subsets $R_i\subseteq [n_i]$
by the componentwise ordering, i.e., we write $(R_1,\ldots, R_m)\subseteq (S_1,\ldots, S_m)$ if $R_i\subseteq S_i$
for $1\leq i\leq m$. With this notation we have:
\begin{equation*}
    \vec{n}! =
    \sum_{\vec{S}\subseteq [n_1]\times\cdots \times [n_m]} d_{(n_1-|S_1|,  \ldots, n_m-|S_m|)}
\end{equation*}
or equivalently,
\begin{equation*}
	d_{(n_1,  \dots, n_m)} = \sum_{\vec{S}\subseteq [n_1]\times \cdots \times [n_m]}(-1)^{|S_1|+\cdots +|S_m|}
	(n_1-|S_1|,\ldots,n_m-|S_m|)!\;.
\end{equation*}
\end{lemma}

\begin{proof}
The first equation is a simple generalization to words of the corresponding statement for permutation derangements, namely that
the total number of permutations can be written as the number of permutations with a given fixed point
set (and the remainder of the permutation is a derangement).

More precisely, if $\vec S\subseteq [n_1]\times\cdots\times [n_m]$, then there are exactly  $d_{(n_1-|S_1|,  \ldots, n_m-|S_m|)}$ words $w$
of content $\vec n$ such that the $j^{th}$ copy of $i$ occurs in one of the positions occupied by $i$ in \eqref{equation.derangement.def}
if and only if $j\in S_i$.

The second equation follows from the first via M\"obius inversion using that the M\"obius function of a product is
the product of the M\"obius functions and that, for the Boolean lattice, $\mu(A,B)=(-1)^{|B|-|A|}$.
\end{proof}

\begin{proof}[Proof of Theorem~\ref{theorem.toom.interval.eigenvalues}]
  In Lemma~\ref{lemma.toom relation}, we showed that the generators
  of the Toom--Tsetlin model satisfy the relations of the \fhb. Since the \fhb{}
  is $\RR$-trivial by Corollary~\ref{corollary.fhb.r_trivial}, we can apply
  the $\RR$-trivial monoid technology to
  recover eigenvalues. The advantage of doing this is that by
  Proposition~\ref{proposition.fhb.l_classes} we already know that the lattice of idempotent-generated
  left ideals is the full Boolean lattice (so the M\"obius inversion is
  easy), and we have a natural choice of idempotent representatives
  (decreasing products of generators).

  The strategy of the proof is to show that both the multiplicities of the irreducible characters and the derangement
  numbers are obtained by inclusion-exclusion from the same statistic (multinomial numbers), so that
  they coincide.

  We first compute the character (i.e., number of fixed points) of the
  idempotent representatives acting on the state set of the Toom--Tsetlin model from
  Section~\ref{subsection.toom multiperm}.

  Consider a subset $R$ of the generators and, for $1\le i\le m$, set $R_i=\{j\in [n_i]\mid \partial_{b_i,j}\in R\}$.
  Set $r_i=|R_i|$, $\vec{R}=(R_i)_{1\le i\le m}$. Note that $R$ and $\vec{R}$ completely determine each other
  and that $R\subseteq S$ if and only if $\vec{R}\subseteq \vec{S}$, where we write $\vec{R}\subseteq \vec{S}$
  if and only if $R_i\subseteq S_i$ for $i=1,\ldots, m$. Hence we can identify $\Lambda(\FHB(X_{\vec n}))$ with the set
  of such $\vec{R}$ with the dual to this ordering.

  As the idempotent associated to $R$ (or equivalently, $\vec{R}$), we take
  \begin{equation}
  \label{equation.idempotent.toom}
    e_{\vec{R}}= \prod_{i=1}^m\prod_{j\in R_i} \toomgens_{b_i,j}\,,
  \end{equation}
  where the inside products are taken decreasingly along $R_i$ and the outer product is taken increasingly
  along $i=1,\ldots, m$ (reading products from left to right). For example, if $m=2$, $R_1=\{1,3\}$, and
  $R_2=\{2,3,5\}$, we obtain the idempotent
  \begin{equation}
  \label{equation.example.idempotent}
  e_{\vec R}=\toomgens_{1,3}\toomgens_{1,1}\toomgens_{2,5}\toomgens_{2,3}\toomgens_{2,2}.
  \end{equation}

  \noindent
  \textbf{Claim}: The number of fixed points of $e_{\vec{R}}$ is given by the multinomial coefficient
  \begin{equation}
  \label{equation.multinomial.fixedpoint}
      |e_{\vec R}\Omega|=
      (n_1-|R_1|,\ldots,n_m-|R_m|)!\;.
  \end{equation}

  \begin{proof}[Proof of Claim]
  First we sketch the idea of the proof. For a product of
  generators in this order, after some operator $\toomgens_{b,j}$ moves
  the $j^{th}$ $b$ right after the $(j-1)^{st}$ $b$, the succeeding
  generators will never separate them. Hence, if $\{j,\dots,j+k\}\subseteq
  R_i$, then in the result the $(j-1)^{th}$ to $(j+k)^{th}$ $b_i$s are
  consecutive and, if $j,\dots,j+k$ is of maximal length, we say that
  those $b_i$s form a block. Note that there may be two consecutive
  blocks of $b$s. One also has to be a bit careful when
  $j=1$.  For the intuition assume that there
  is a fake $0^{th}$ $b_1$ at the beginning of the word, and a fake $0^{th}$
  $b_i$ just after the first block of $b_{i-1}$s. After the application of the
  full idempotent, there are, besides the first $m$ starting blocks,
  $n_i-r_i$ blocks of $b_i$s for each $1\le i\le m$. Thus,
  producing all the elements in the image set of $e_{\vec{R}}$ amounts to
  choosing among all possible ways to intertwine those blocks of $b$s; there are
  $(n_1-r_1,\ldots,n_m-r_i)!$ such choices.

  Let us now formalize this argument by simultaneous induction on $|R|=r_1+\cdots +r_m$ over all
  possible contents $(n_1,\ldots, n_m)$. By a slight abuse
  we use the same notation for the operators even if we change the content. If $R=\emptyset$, then
  $e_{\vec R}$ is the identity and so the fixed point set is $\Omega$,
  whose cardinality is $\vec{n}!$ as desired.

  Take now $R$ with  $|R|\geq 1$, and assume that the claim holds for all subsets of cardinality
  strictly less than $|R|$. Take $k$ minimal such that $R_k\neq \emptyset$ and let $j$ be the largest element of $R_k$.
  Define $R'$ such that $R_i=R'_i$ whenever $i\neq k$ and $R'_k=R_k\setminus \{j\}$.  Then
  $e_{\vec R}=\toomgens_{b_k,j}e_{\vec {R'}}$.
  Let $\Omega'$ be the set of all words over $\mathcal B$ with content $(n_1,n_2,\ldots, n_{k-1},n_k-1,n_{k+1},\ldots, n_m)$.
  Let $\pi\colon \Omega\twoheadrightarrow \Omega'$
  denote the mapping which erases the $j^{th}$ copy of $b_k$ from a word.  We claim that $\pi$ restricts to a bijection
  $\pi\colon e_{\vec R}\Omega\to e_{\vec R'}\Omega'$. This will complete the proof by applying induction to $e_{\vec R'}$
  because $|R'_k|=|R_k|-1$ and hence $(n_k-1)-|R'_k|=n_k-|R_k|$.

  First observe that if $b\neq b_k$, then
  $\pi\toomgens_{b,i}=\toomgens_{b,i}\pi$ because copies of $b$ can always move past copies of $b_k$. Also, if $i<j$,
  then $\pi\toomgens_{b_k,i}=\toomgens_{b_k,i}\pi$ because $\toomgens_{b_k,i}$ only changes the prefix of a word
  preceding the $j^{th}$ copy of $b_k$. Finally, $\pi\toomgens_{b_k,j}=\pi$.  Thus we have
  $\pi e_{\vec{R}} = \pi e_{\vec{R'}} = e_{\vec{R'}}\pi$
  and hence $\pi (e_{\vec{R}}(w)) = e_{\vec{R'}}\pi(w)$ for all
  $w\in \Omega$.  Therefore, $\pi(e_{\vec{R}}\Omega)\subseteq e_{\vec{R'}}\Omega'$. To complete the proof it is
  convenient to note that $e_{\vec R}\Omega\subseteq \toomgens_{b_k,j}\Omega$.

  There are two cases.  Suppose first that $j>1$. Then the fixed points of $\toomgens_{b_k,j}$ are those words
  where the $j^{th}$ copy of $b_k$ is immediately after the $(j-1)^{st}$ copy of $b_k$. So define
  $\rho\colon \Omega'\to \Omega$ to be the map inserting a $b_k$ immediately after the $(j-1)^{st}$ copy
  of $b_k$.  Trivially, $\pi\rho=1_{\Omega'}$ and if $w\in \toomgens_{b_k,j}\Omega$, then $\rho\pi(w)=w$.
  Thus to show that $\pi\colon e_{\vec{R}}\Omega\to e_{\vec{R'}}\Omega'$ is a bijection, it remains to show
  that $\rho(e_{\vec {R'}}\Omega')\subseteq e_{\vec{R}}\Omega$. Recalling that
  $\pi e_{\vec{R}}\rho  = e_{\vec{R'}}\pi\rho =e_{\vec{R'}}$, it follows that if $u\in e_{\vec{R'}}\Omega'$,
  then $\pi (e_{\vec{R}}(\rho(u)))=u=\pi(\rho(u))$. Thus $e_{\vec{R}}(\rho(u)))$ can differ from $\rho(u)$ only in
  the position of the $j^{th}$ copy of $b_k$.  But in both of these words the $j^{th}$ copy of $b_k$ is immediately
  after the $(j-1)^{st}$ copy of $b_k$.  Thus $\rho(u)=e_{\vec{R}}(\rho(u))$ and so
  $\rho\colon e_{\vec{R'}}\Omega'\to e_{\vec{R}}\Omega$ is inverse to  $\pi\colon e_{\vec R}\Omega\to e_{\vec R'}\Omega$.

  For the case $j=1$, observe that the fixed point set of $\toomgens_{b_k,1}$ consists of those words beginning
  with $b_k$.  So this time, let $\rho\colon \Omega'\to \Omega$ be the mapping inserting $b_k$ at the beginning of
  a word. Then again $\pi\rho=1_{\Omega'}$ and if $w\in \toomgens_{b_k,1}\Omega$, then $\rho\pi(w)=w$. As before,
  it just remains to show that $\rho(e_{\vec {R'}}\Omega')\subseteq e_{\vec{R}}\Omega$. The same argument as the
  previous case shows that if $u\in e_{\vec{R'}}\Omega'$, then $\pi (e_{\vec{R}}(\rho(u)))=u=\pi(\rho(u))$. Thus
  $e_{\vec{R}}(\rho(u)))$ can differ from $\rho(u)$ only in the position of the $1^{st}$ copy of $b_k$.  But  both
  of these words have the $1^{st}$ copy of $b_k$ as their first symbol.  Thus $\rho(u)=e_{\vec{R}}(\rho(u))$,
  completing the proof.
  \end{proof}

Applying Theorem~\ref{theorem.eigenvalues} and recalling the isomorphism between $P(X_{\vec n})$ and
$\Lambda(\FHB(X_{\vec n}))$ ordered by reverse inclusion, there is an eigenvalue $\lambda_R$ corresponding
to each subset $R\subseteq X_{\vec n}$ given by $\lambda_R=\sum_{\toomgens_{b_i,j}\in R} x_{b_i,j}=
\sum_{i=1}^m x_{b_i,R_i}$.  Let us continue to put $r_i=|R_i|$. The multiplicity of this eigenvalue according to
Theorem~\ref{theorem.eigenvalues} is
\begin{align*}
	m_{R}&= \sum_{R\subseteq U} (-1)^{|U|-|R|}|e_{\vec U}\Omega|\\
	&= \sum_{\vec{R}\subseteq \vec{U}}(-1)^{\sum_{i=1}^m|U_i|-r_i} (n_1-|U_1|,\ldots,n_m-|U_m|)!\\
	&=\sum_{\vec{S}\subseteq ([n_1]\setminus R_1,\ldots, [n_m]\setminus R_m)}(-1)^{\sum_{i=1}^m |S_i|}
	(n_1-r_1-|S_1|,\ldots,n_m-r_m-|S_m|)!\\
	&=d_{(n_1-|R_1|,\ldots, n_m-|R_m|)}\;,
\end{align*}
where the penultimate equality reindexes the sum by setting $S_i= U_i\setminus R_i$ and the final equality
is from Lemma~\ref{lemma.count.derangement}.
\end{proof}

\begin{proof}[Proof of Theorem~\ref{theorem.toom.window.eigenvalues}]
  By Corollary~\ref{c:treemonoidforinterval} we know that the monoid $M_{X_I}$ for the interlibrary loan Toom model
  is a tree monoid and $\RR$-trivial. Hence, as before, the lattice of idempotent-generated left ideals is the
  full Boolean lattice $P(X_I)$ by Proposition~\ref{proposition.fhb.l_classes} and we can apply  Theorem~\ref{theorem.eigenvalues}.
  We compute the number of fixed points of the idempotents.  We retain the notation from the proof of Theorem~\ref{theorem.toom.interval.eigenvalues}.


  \begin{lemma}
  \label{lemma.fixed points interlibrary}
  If $\vec{R}\subsetneq [L]^m$ and $\sum_{i=1}^m \min(\overline{R}_i)<L+m$, then the number of fixed points of
  $e_{\vec{R}}$ is
  \begin{equation}
  \label{equation.fixed points interlibrary}
  |e_{\vec R}\Omega|= \sum_{\vec n\in I(\vec R)}
   	(\vec{n}-\vec{f}(\vec{R},\vec{n}))!\;,
  \end{equation}
  where $I(\vec R)$ consists of those $\vec n\in \mathbb N^m$ such that $\|\vec n\|_1 =L$,
  $n_i \neq 0$ if $1\in R_i$, and there is at most one $i\in \{1,\ldots, m\}$ with $n_i+1\in R_i$.
  Furthermore, $\vec{f}(\vec{R},\vec{n})$ is the $m$-dimensional vector with
  \[
  f_i(\vec{R},\vec{n}) = |\{r\in R_i \mid n_i\ge r-1\}|.
  \]
  Otherwise, $|e_{\vec R}\Omega|= 1$.
  \end{lemma}

\begin{proof}
If $R_i=[L]$, then there is a unique fixed point of $e_{\vec R}$. Note that for a word to be a fixed point of $e_{\vec R}$, the letter
$b_1$ needs to be in positions $1$ up to $\min(\overline{R}_1)-1$, the letter $b_2$ in positions $\min(\overline{R}_1)$ to $\min(\overline{R}_1)
+\min(\overline{R}_2)-2$, etc. Hence if $\sum_{i=1}^m (\min(\overline{R}_i)-1)\ge L$, there is certainly only one fixed point.

So from now on we assume $R_i\subsetneq [L]$ for all $1\leq i\leq m$ and $\sum_{i=1}^m \min(\overline{R}_i)<L+m$.

Let $\Omega=\mathcal B^L$ and partition $\Omega$ by content:
\[
  	\Omega = \bigcup_{\{\vec{n}\in \mathbb N^m\mid  \|\vec{n}\|_1=L\}} \Omega_{\vec{n}}\;,
\]
where $\Omega_{\vec{n}} \subseteq \Omega$ is the subset of words of content $\vec{n}$.

\noindent
\textbf{Claim}: Let $\vec n=(n_1,\ldots, n_m)$ be in $\mathbb N^m$ with $\|\vec n\|_1=L$. Then
\begin{equation}\label{e:fixedpointsize}
|e_{\vec R}\Omega\cap \Omega_{\vec n}|=\begin{cases}
(\vec n-\vec f(\vec R,\vec n))!, & \text{if}\ \vec n\in I(\vec R),\\
0, & \text{else.}
\end{cases}
\end{equation}
We proceed by induction on $L$ where the case $L=0$ is trivial. Assume $|L|\geq 1$.
We consider two cases.  Suppose first that $n_i+1\notin R_i$ for  $1\leq i\leq m$.  Then $\Omega_{\vec n}$ is invariant under
$\toomgens_{b_i,r}$, for all $1\leq i\leq m$ and $r\in R_i$, and thus under  $e_{\vec R}$.  Also if $r>n_i+1$, then $\toomgens_{b_i,r}$
fixes $\Omega_{\vec n}$.   Therefore, if we define $\vec{Q}$ by $Q_i=\{r\in R_i\mid n_i\geq r\}$, then the action of $e_{\vec R}$ on
$\Omega_{\vec n}$ agrees with that of $e_{\vec Q}$ on $\Omega_{\vec n}$.   But the latter  is exactly the same as the action of
$e_{\vec Q}$ on $\Omega_{\vec n}$ in the monoid $\Monoid_{\vec n}$ for the Toom-Tsetlin model from Section~\ref{subsection.toom multiperm}.
Thus \[|e_{\vec R}\Omega\cap \Omega_{\vec n}|=|e_{\vec Q}\Omega_{\vec n}| = (\vec n-\vec f(\vec R,\vec n))!\] by \eqref{equation.multinomial.fixedpoint}
since $n_i-|Q_i| = f_i(\vec R,\vec n)$ because $n_i\neq r-1$ for $r\in R_i$.

Next suppose that $n_i+1\in R_i$ for some $1\leq i\leq m$.  Choose $i$ maximal with this property.  Let $S=R\setminus \{\partial_{b_i,n_i+1}\}$,
viewed as operators on $\mathcal B^{L-1}$ for the model with one fewer book on the shelf.  We claim that if $e_{\vec R}\Omega\cap \Omega_{\vec n}$
is non-empty, then $i$ is the only index $k$ with $n_k+1\in R_k$ and that $w\in \Omega_{\vec n}$ is fixed by $e_{\vec R}$ if and only if $w=ub_i$
with $e_{\vec S}(u)=u$ and $u\in \Omega_{\vec n-\vec e_i}$ where $\vec e_i$ is the $i^{th}$-standard unit vector and we are working in the model
with $L-1$ books on the shelf. The claim will then follow from the inductive hypothesis because $\vec f(\vec R,\vec n)=f(\vec S,\vec n-\vec e_i)$.

Indeed, suppose that $e_{\vec R}(w)=w$ with $w\in \Omega_{\vec n}$ and factor the expression \eqref{equation.idempotent.toom} as
$e_{\vec R}=\alpha\toomgens_{b_i,n_i+1}\beta$. Assume that $\beta(w)=vb_j$.  Note that $\beta(w)\in \Omega_{\vec n}$ by maximality of $i$.
Suppose first that $j>i$.  Then $\toomgens_{b_i,n_i+1}\beta(w)$ will have $n_j-1$ occurrences of $b_j$.  Since $\alpha$ contains no
operator $\toomgens_{b_j,t}$, applying $\alpha$ to $\toomgens_{b_i,n_i+1}\beta(w)$ cannot create a new $b_j$, and so $w$ cannot
be fixed by $e_{\vec R}$.

Next suppose that $j<i$. Let us first show that $w=ub_i$. Indeed, $\toomgens_{b_i,n_i+1}(vb_j)$ will have $n_i+1$ occurrences of
$b_i$, with the last two consecutive.  Since $w$ is fixed by $e_{\vec R}$, we must be able to factor $\alpha=\alpha'\toomgens_{b,t}\alpha''$
where $\alpha''(\toomgens_{b_i,n_i+1}(vb_j))=zb_ib_i$ and $\toomgens_{b,t}(zb_ib_i)= z'b_i$ has $n_i$ copies of $b_i$.  Since no operator
in $\alpha'$ can insert or move a $b_i$, it follows that $w=\alpha(vb_j)$ must end in $b_i$.  Thus $w=xb_jyb_i$ where $y$ has no $b_j$.
There are two cases.

Suppose first that $x$ contains a $b_i$. Write $w=x'b_ix''b_jy'b_iz$ where $x''y'$ contains no $b_i$. Say these two copies of $b_i$ are the
$p^{th}$ and $(p+1)^{st}$ copy.  Then since $\beta(w)=vb_j$, it follows that $\beta$ contains the operator $\toomgens_{b_i,p+1}$ and
so in $\beta(w)$ the $p^{th}$ and $(p+1)^{st}$ copies of $b_i$ are consecutive.  Since $\alpha\toomgens_{b_i,n_i+1}$ does not
contain $\toomgens_{b_i,p}$, it follows that they remain consecutive in $w=e_{\vec R}(w)=\alpha\toomgens_{b_i,n_i+1}\beta(w)$, a contradiction.

Next suppose that $x$ contains no $b_i$, that is, the last $b_j$ is to the left of all the $b_i$s.  We shall contradict the assumption that $\sum_{k=1}^m\min(\overline{R_k})<L+m$.
From $\beta(w)=vb_j$, we conclude that all copies of $b_i$ are moved passed the last $b_j$ by $\beta$ and so
we have that $\{1,2,\ldots, n_i+1\} \subseteq R_i$.  
Let $n=\min(\overline{R}_i)$. Then we can factor the expression \eqref{equation.idempotent.toom}
into $e_{\vec R}=\gamma\toomgens_{b_i,n-1}\cdots \toomgens_{b_i,1}\gamma'$ and $\toomgens_{b_i,n-1}\cdots \toomgens_{b_i,1}\gamma'(w)
=b_i^{n-1}z$ where $z$ has no $b_i$ and each letter occurs in $z$ no more than it occurs in $w$.  Notice that $w=\gamma(b_i^{n-1}z)$ will have all its $b_i$s consecutive and so, in fact, $w=u'b_i^{n_i}$
(recall that we already showed that $w$ ends in $b_i$).  Since $\gamma$ contains no operator $\toomgens_{b_k,r}$ with $k>i$, we can neither
move, nor reinsert any letter $b_k$ of $z$ with $k>i$.  Thus we conclude that all $b_k$ in $z$ satisfy $k<i$.  In order for $\gamma$ to take
$b_i^{n-1}z$ to the word $u'b_i^{n_i}$, we must have that, for each letter $b_k$ in $z$, the expression for $\gamma$ has a factor
$\toomgens_{b_k,r}\toomgens_{b_k,r-1}\cdots\toomgens_{b_k,1}$ and the total number of operators coming from such factors must be
at least $|z|=L-n+1$. Let us lower bound $\sum_{k=1}^m\min(\overline{R_k})$ by $m$ (because each $1\leq k\leq m$ contributes
at least $1$) plus an additional $n-1$ for $k=i$ plus an additional 
$L-n+1$ coming from those $k$ with $b_k$ appearing in $z$, and hence
yielding factors of $\gamma$ of the above form.  This implies
\[
	\sum_{k=1}^m\min(\overline{R_k})\geq m+n-1+L-n+1 = L+m\;.
\]
This is a contradiction.

We are left now with the case $j=i$, and consequently $n_i\neq 0$.  We now claim that, for all factorizations $e_{\vec R}=\rho\sigma$
of the expression \eqref{equation.idempotent.toom}, we have $\sigma(w)$ ends in $b_i$ and has content $\vec n$.  In other words, we
claim that when computing $e_{\vec R}(w)$ the content never changes and the last $b_i$ never moves.  This, in particular, will imply
that there is no other $t$ with $n_t+1\in R_t$ (else the content would change at some point).  From $\beta(w)=vb_i$ with $vb_i$ having
content $\vec n$, if $\sigma$ is a suffix of $\beta$ the statement is clear. Also for $\sigma=\toomgens_{b_i,n_i+1}\beta$, we have
$\sigma(w)=\toomgens_{b_i,n_i+1}(vb_i)=vb_i$, as desired.  No operator in $\alpha$ can move the $n_i^{th}$ copy of $b_i$.
Hence when computing $\alpha(vb_i)$, if the content is ever changed then the last $b_i$ will be removed and cannot be reinserted.
But then $w$ cannot be fixed by $e_{\vec R}$. Thus the claim is also true when $\sigma$ contains $\toomgens_{b_i,n_i+1}\beta$ as a suffix.

It remains to show that $w=ub_i$ with content $\vec n$ is fixed by $e_{\vec R}$ if and only if $e_{\vec S}(u)=u$.  Assume first that
$e_{\vec R}(w)=w$ and write $\beta(w)=vb_i$ as above. Since $vb_i$ has content $\vec n$, we have
$ub_i=w=\alpha\toomgens_{b_i,n_i+1}(vb_i)=\alpha(vb_i)=\alpha\beta(w)=\alpha\beta(ub_i)$.  As the last $b_i$ is never moved and the content is
never changed when computing $\alpha\beta(ub_i)$, we deduce that $u=\alpha\beta(u)=e_{\vec S}(u)$.

Conversely, assume that $e_{\vec S}(u)=u$.  If $n_i\notin R_i$ (and hence $n_i\notin S_i$), then $\Omega_{\vec n-\vec e_i}$ is invariant under each of the operators $\toomgens_{b_k,r}$ appearing in $e_{\vec S}$ (i.e., with $r\in S_k$).  Thus $\beta(ub_i) = \beta(u)b_i$ and
$e_{\vec R}(w) =\alpha\toomgens_{b_i,n_i+1}(\beta(u)b_i)=\alpha(\beta(u)b_i)=\alpha\beta(u)b_i=e_{\vec S}(u)b_i=ub_i=w$.  If
$n_i\in R_i$ (and hence $n_i\in S_i$), then from $e_{\vec S}(u)=u$ and $(n_i-1)+1=n_i\in S_i$, we must have by the above that
$u$ ends in $b_i$, this $b_i$ never moves when computing $e_{\vec S}(u)$ and the content never changes during the computation.  Therefore, writing
$\beta=\toomgens_{b_i,n_i}\beta'$, we then have $\beta'(u) = v'b_i$ where $v'b_i$ has content $\vec n-\vec e_i$ and
$\beta(u)=\toomgens_{b_i,n_i}(v'b_i)=v'b_i$.  But then $\beta(w) = \toomgens_{b_i,n_i}\beta'(ub_i)=\toomgens_{b_i,n_i}(\beta'(u)b_i)
=\toomgens_{b_i,n_i}(v'b_ib_i)=v'b_ib_i=\beta(u)b_i$ and so $e_{\vec R}(w)=\alpha\toomgens_{b_i,n_i+1}\beta(w)=\alpha\toomgens_{b_i,n_i+1}(\beta(u)b_i)
= \alpha(\beta(u)b_i)=\alpha\beta(u)b_i=e_{\vec S}(u)b_i=ub_i=w$.  This completes the proof of \eqref{e:fixedpointsize}

The lemma is now immediate from \eqref{e:fixedpointsize}.
\end{proof}


  As in the proof of Theorem~\ref{theorem.toom.interval.eigenvalues}, we are going to apply Theorem~\ref{theorem.eigenvalues}
  to find the multiplicity $m_{\vec{R}}$ for each eigenvalue $\lambda_{\vec{R}}$. First suppose that there exists
  at least one index $1\le i\le m$ such that $R_i=[L]$. In this case
  \[
  m_{\vec{R}} = \sum_{\vec{R} \subseteq \vec{U}} (-1)^{\|\vec{U}\|_1-\|\vec{R}\|_1} |e_{\vec{U}} \Omega|
  = \sum_{\vec{R} \subseteq \vec{U}} (-1)^{\|\vec{U}\|_1-\|\vec{R}\|_1} 1
  = \begin{cases} 1, & \text{if $\vec{R}=[L]^m$,}\\ 0, & \text{otherwise,} \end{cases}
  \]
  as desired.

  Now let $\vec{R}$ be such that $R_i \subsetneq [L]$ for all $1\le i\le m$. Then
  \begin{equation}
  \label{equation.multiplicity interlibrary}
	m_{\vec{R}} = \sum_{\vec{R}\subseteq \vec{U}} (-1)^{\|\vec{U}\|_1-\|\vec{R}\|_1}|e_{\vec U}\Omega|
  \end{equation}
  with $|e_{\vec U}\Omega|$ as in Lemma~\ref{lemma.fixed points interlibrary}.
\end{proof}

\section{Nonabelian directed sandpile model}
\label{sandpile}

In this section we briefly show that the monoid associated to the landslide
nonabelian sandpile model introduced in~\cite{ayyer_schilling_steinberg_thiery.sandpile.2013}
can be shown to be $\RR$-trivial using the \fhb{} technique of Section~\ref{section.free tree monoid}.
In~\cite{ayyer_schilling_steinberg_thiery.sandpile.2013} this property was proved using the
wreath product.

\subsection{The landslide nonabelian directed sandpile model}

The landslide nonabelian directed sandpile model is defined on an arborescence.
An \emph{arborescence} is a directed graph with a special vertex being the root such that there is
exactly one directed path from any vertex to the root. Vertices without an incoming edge
are called the leaves. Let $V$ be the set of all vertices of the arborescence. We associate
to each vertex $v\in V$ a threshold $T_v$. Then the state space of the Markov chain
is defined to be
\begin{equation*}
	\Omega = \{ (t_v)_{v\in V} \mid 0 \le t_v \le T_v\}\;.
\end{equation*}

We consider two types of operators on the state space (which are the generators of the underlying monoid),
the source and topple operators. There is a source operator $\source_v$, for each $v\in V$, which informally works as follows:
a grain enters at $v$ and stays at the first vertex below its threshold on the unique path from $v$ to the root; if no such vertex exists,
then the grain leaves the system. The topple operator $\tau_v$, for $v\in V$, takes all grains at vertex $v$ and topples
them to the first available slots along the unique path from $v$ to the root (again grains which cannot find available slots leave the system).

Letting $\su(v)$ be the unique successor of the vertex $v$ in the arborescence, we can formally define these
operators recursively by picking one fixed leaf $\ell$ and writing any configuration as $(t_\ell,t)$, where $t_\ell$
is the number of grains at $\ell$ and $t$ is the state on the remaining vertices. Then we have
\begin{align}
\label{equation.sandpile recursion}
	\source_\ell(t_\ell,t) &= \begin{cases}(t_\ell+1,t), & \text{if}\ t_\ell<T_\ell\\ 
	(T_\ell,\source_{\su(\ell)}t), & \text{if}\ t_\ell=T_\ell\end{cases}
\nonumber \\
\source_v(t_\ell,t)&=(t_\ell,\source_{v}t) & (v\neq \ell)\\
\tau_\ell(t_\ell,t)&= (0,\source_{\su(\ell)}^{t_\ell}t) \nonumber \\
\tau_v(t_\ell,t)&=(t_\ell,\tau_{v}t) & (v\neq \ell). \nonumber
\end{align}

For more details, see~\cite{ayyer_schilling_steinberg_thiery.sandpile.2013}.

\subsection{$\RR$-triviality of the landslide directed sandpile model}
We begin with a lemma which enables us to use the generalization of the \fhb{} of
Section~\ref{subsection.fhb generalization} to prove $\RR$-triviality.
Let $X_\tau$ be the set of generators of the landslide nonabelian directed sandpile model.

\let\del=\partial
\begin{lemma}
  \label{lemma.landslide.relations}
  We claim that any two operators $x$ and $y$ in $X_\tau$ commute, except
  when $y=\tau_u$ and $x=\tau_v$ or $x=\source_v$ for two nodes $u$ and
  $v$ with $u$ on the path from $v$ to the root. When $x$ and $y$ do not commute,
  $y$ is an idempotent, and $yxy = yx$.
\end{lemma}

\begin{proof}
 We first check that any two operators $\source_u$ and $\source_v$
 commute. This is obvious if $u=v$.
 If neither $u$ nor $v$ is the fixed leaf $\ell$, then the result is clear from \eqref{equation.sandpile recursion} and induction.
 So without loss of generality, assume that $v=\ell$.
 Then, applying induction
 and the recursion formula~\eqref{equation.sandpile recursion} (see
 also~\cite[Table 1]{ayyer_schilling_steinberg_thiery.sandpile.2013})
 we obtain that if $t_v<T_v$, then
  \begin{displaymath}
     \source_u\circ\source_v(t_v, t) = \source_u(t_v+1, t) = (t_v+1, \source_u t) =
    \source_v(t_v,\source_u t) = \source_v\circ\source_u (t_v,t);
  \end{displaymath}
  and if $t_v=T_v$, then
  \begin{multline*}
    \source_u\circ \source_v(T_v, t) = \source_u(T_v, \source_{s(v)}(t))
    = (T_v, \source_u \circ \source_{s(v)}(t)) \\
    = (T_v, \source_{s(v)}\circ \source_u(t))
    = \source_v (T_v,\source_u(t)) = \source_v \circ \source_u (T_v,t).
  \end{multline*}
  The other commutation relations are treated similarly.

  In the remaining case, $y=\tau_u$ is idempotent as desired, and the
  relation is checked similarly.
\end{proof}

\begin{theorem}
  The monoid $M_\tau = \langle \source_v, \tau_v \mid v\in V \rangle$ of the landslide directed sandpile model is $\RR$-trivial.
\end{theorem}

\begin{proof}
  We choose the following total order on the elements of the generators in $X_\tau$
  such that for the nodes $u,v$ of the tree, $\tau_v<_{X_\tau} \tau_u$ and $\source_v <_{X_\tau} \tau_u$ whenever $u$ is
  in the path from the root to $v$ (where $v=u$ is allowed in the second case). Then Lemma~\ref{lemma.landslide.relations}
  and the easily checked fact (using induction and \eqref{equation.sandpile recursion}) that $\source_v^m=\source_v^{m+1}$ for $m$ large enough imply
  that the hypotheses of Proposition~\ref{proposition.fhb generalization} are satisfied. Therefore,
  the monoid $M_\tau$ is $\RR$-trivial.
\end{proof}

\section{The exchange walk on a Coxeter group}
\label{section.coxeter}
This section requires the reader to be familiar with basic notions from the theory of finite Coxeter groups.  Standard
references include~\cite{Brown:book2,bjornerbrenti}.

Let $(W,S)$ be a finite Coxeter system.  Let $R(w)$ denote the set of reduced expressions of an element $w$ of $W$.
If $w_0$ is the longest element of $W$, then $R(w_0)$ can be viewed as the set of maximal chains in the weak order
on $W$.  Let us denote words over $S$ by Greek letters in what follows and write $[\alpha]_M$ for the image of
$\alpha\in S^*$ in an $S$-generated monoid $M$.  Let $s\in S$ and let $\alpha=s_1\cdots s_m$ be a reduced
decomposition of $w_0$. Then, by the Exchange Condition for Coxeter groups, there is a unique index $i$ such that
$e_s(\alpha)=ss_1\cdots \wh{s_i}\cdots s_m$ is a reduced decomposition of $w_0$ where $\wh{s_i}$ means omit
$s_i$.  For example, if $W=(\mathbb Z/2\mathbb Z)^n$ with $S$ the standard unit vectors,
then $w_0$ is the all-ones vector and the reduced decompositions of $w_0$ are all linear orderings of $S$
(written as words).  Then $e_s$ moves $s$ to the front of a linear
ordering of $S$, as in the Tsetlin library. Another example is presented in
Figure~\ref{figure.exchange_walk}.

\begin{figure}[h]
  \centering
  \includegraphics[width=\textwidth]{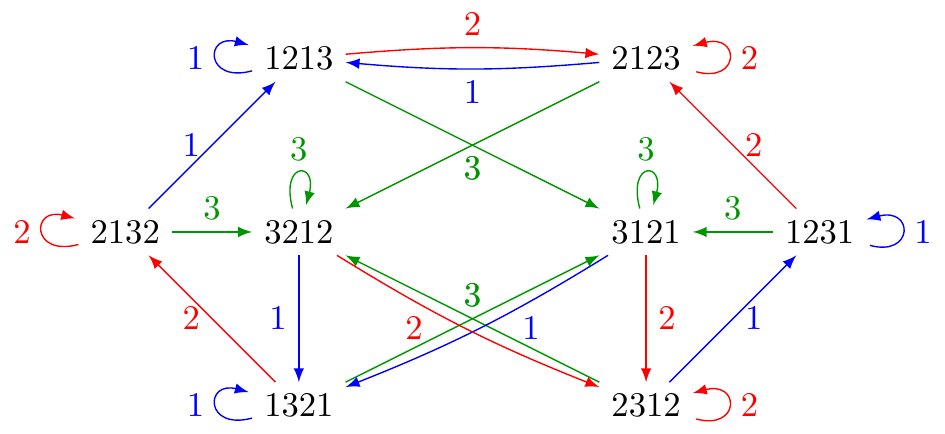}
  \caption{The exchange walk on $(W,S)=(S_3\times S_2$,
    $\{s_1,s_2,s_3\})$, where $s_1$ and $s_2$ satisfy the braid
    relation, and $s_3$ commutes with $s_1$ and $s_2$. For short, the
    reduced word $s_1s_2s_1s_3$ is denoted by $1213$.}
  \label{figure.exchange_walk}
\end{figure}

Consider a probability $P$ on $S$ and consider the following Markov chain, which we call the \emph{exchange walk on $(W,S)$}.
The state set is $R(w_0)$.  Transitions are given by changing from state $\alpha$ to state $e_s(\alpha)$ with probability
$P(s)$.  For the example above of $(\mathbb Z/2\mathbb Z)^n$ we recover the Tsetlin library. The main goal of this section
is to use the theory of $\RR$-trivial monoids and Markov chains developed in this paper to prove properties of the exchange walk on
$(W,S)$.

To state our main result of this section, we need some notation.  Let $W_J=\langle J\rangle$ be the standard parabolic
subgroup associated to $J\subseteq S$. Let
\[
	D_R(w) = \{s\in S\mid \ell(ws)<\ell(w)\}
\]
be the set of right descents of $w\in W$. Let $w_J$ denote the longest element of $W_J$; note that $w_J$
is an involution and $w_S=w_0$. Our result is the following.

\begin{theorem}\label{exchangethm}
Let $(W,S)$ be a finite Coxeter system and let $P$ be a probability on $S$ with support $S$.  Let $T$
be the transition matrix of the exchange walk on $(W,S)$. Then the exchange walk is ergodic and the following hold.
\begin{enumerate}
\item The eigenvalues of $T$ are
\[
	\lambda_J = \sum_{s\in J}P(s),
\]
where $J\subseteq S$.
\item The multiplicity of $\lambda_J$ as an eigenvalue is given by
\[
	\sum_{K\supseteq J}(-1)^{|K|-|J|}\cdot |R(w_Kw_0)|.
\]
\item The stationary distribution $\pi$ is given as follows: if $\alpha =s_1\cdots s_m$ is a reduced decomposition
of $w_0$, then
\[
	\pi(\alpha) = \prod_{i=1}^m\frac{P(s_i)}{1-\lambda_{D_R([s_1\cdots s_{i-1}]_W)}}.
\]
\item Let $m$ be the length of the longest element $w_0\in W$, $n=|S|$ be the number of generators (usually called the
\emph{rank} of $W$), and $p = \min\{P(s) \mid s\in S\}$. Then, the mixing time for the exchange walk is $O(m/p)$ in general
and $O(mn)$ when $P$ is the uniform distribution on $S$.
\end{enumerate}
\end{theorem}

To prove Theorem~\ref{exchangethm} we introduce an $\mathscr R$-trivial monoid $R(W,S)$, which
is the Karnofsky--Rhodes expansion of the $0$-Hecke monoid $H(W,S)$. First we recall that notion of the $0$-Hecke
monoid. Details can be found in~\cite{Carter0Hecke,Norton,denton_hivert_schilling_thiery.2010,Denton,Fayers,DoubleCatalan}.

The \emph{$0$-Hecke monoid} $H(W,S)$ is the monoid generated by $S$ and whose defining relations are the same
commutation and braid relations as those of $W$, but the quadratic relation $s^2=1$ is replaced by $s^2=s$ for $s\in S$.
It follows from Tits' solution to the word problem for Coxeter groups that the reduced expressions for $H(W,S)$ and $W$ are
the same and that two reduced expressions are equivalent in $W$ if and only if they are equivalent in $H(W,S)$.  Thus
the elements of $W$ are in bijection with the elements of $H(W,S)$ via the map that sends $w\in W$ to the unique
element $\pi_w$ of $H(W,S)$ that has the same reduced decompositions as $w$.  Moreover, the
idempotents of $H(W,S)$ are the elements $e_J=\pi_{w_J}$ with $J\subseteq S$.  The monoid $H(W,S)$ is both
$\mathscr R$- and $\mathscr L$-trivial (and hence $\mathscr J$-trivial). Also $\Lambda(H(W,S))$ is isomorphic to
$P(S)$ (ordered by reverse inclusion) via $H(W,S)e_J\mapsto J$.  For $w\in W$ and $s\in S$, one has
$s\in D_R(w)$ if and only if $\pi_ws=\pi_w$.

We define $R(W,S)$ here directly (the reader can refer to~\cite{Elston} for the Karnosfky--Rhodes expansion and its
properties in general). Let
\[
	R=\bigcup_{w\in W}R(w)\,.
\]
Define $R(W,S)$ to be the monoid with generators $S$ and relations $\alpha s=\alpha$ whenever
$\alpha\in R$ and $s\in D_R([\alpha]_W)$ (or equivalently, whenever $[\alpha s]_{H(W,S)}=[\alpha]_{H(W,S)}$).
Notice that we have a natural surjective homomorphism $\psi\colon R(W,S)\to H(W,S)$ because $H(W,S)$ satisfies
these relations.  Consider the rewriting system $\mathcal R$ with rules $\alpha s\to \alpha$ whenever $\alpha$ is a
reduced expression for $W$ and $s\in S$ with $[\alpha s]_{H(W,S)}=[\alpha]_{H(W,S)}$.  This rewriting system is
length-reducing and defines $R(W,S)$.  We claim that it is complete.

First note that any word $\alpha\in S^*$ can be rewritten using $\mathcal R$ to a reduced expression for $W$ by
scanning from left to right and erasing right descents as they occur (this uses that in a Coxeter group $s\notin D_R(w)$
implies that $\ell(ws)=\ell(w)+1$). Also note that reduced words (in the Coxeter sense) cannot be rewritten since the left
hand side of each rule of $\mathcal R$ is not reduced for $W$ and factors of reduced words are reduced. Next, observe
that any overlap of two rules is of the form $\alpha\beta s\to \alpha\beta$ and $\beta s\gamma s'\to \beta s\gamma$ where
$s,s'\in S$, $\alpha\beta$ and
$\beta s\gamma$ in $S^*$ are reduced for $W$ and $[\alpha\beta s]_{H(W,S)}=[\alpha\beta]_{H(W,S)}$ and
$[\beta s\gamma s']_{H(W,S)}=[\beta s\gamma]_{H(W,S)}$. As observed at the beginning of this paragraph, there is a word
$\rho\in S^*$ reduced for $W$ such that $\alpha\beta\gamma\Rightarrow_{\mathcal R}^*\rho$. Also, we have
\begin{multline*}
	[\rho s']_{H(W,S)}=[\alpha\beta\gamma s']_{H(W,S)}=[\alpha\beta s\gamma s']_{H(W,S)}
	= [\alpha\beta s\gamma]_{H(W,S)}\\
	 = [\alpha\beta\gamma]_{H(W,S)}=[\rho]_{H(W,S)}
\end{multline*}
and so $\rho s'\to \rho$ belongs to $\mathcal R$.  Therefore,
\[
	\rho\mathrel{\prescript{}{\mathcal R}\Leftarrow}\rho s'\mathrel{\prescript{*}{\mathcal R}\Leftarrow}\alpha\beta\gamma
	s'\mathrel{\prescript{}{\mathcal R}\Leftarrow}
	\alpha\beta s\gamma s'\Rightarrow_{\mathcal R} \alpha\beta s\gamma\Rightarrow_{\mathcal R} \alpha\beta\gamma
	\Rightarrow_{\mathcal R}^* \rho.
\]
It follows that $\mathcal R$ is confluent.  We conclude that $R$ is the set of reduced words for $\mathcal R$ and
so we can identify $R(W,S)$ with $R$ where the product is given by concatenation followed by scanning from left
to right, removing descents.  Moreover, since all the rewriting rules of $\mathcal R$ are of the form $\alpha s=\alpha$
with $\alpha$ reduced, it follows that the right Cayley digraph of $R(W,S)$ with respect to $S$ is the prefix tree of $R$ and so $R(W,S)$ is
$\mathscr R$-trivial and Karnofsky-Rhodes. Also $\alpha s=\alpha$ if and only if $s\in D_R([\alpha]_W)$ for $\alpha\in R$.

If $s\in D_R([\alpha]_W)$, then $s$ appears in $\alpha$ by standard Coxeter theory and so both sides of each rule $\alpha s\to \alpha$
of $\mathcal A$ have the same letters.  Thus the projection $S^*\to P(S)$
(where the latter is made a monoid with union) given by $s\mapsto \{s\}$ factors through $R(W,S)$.  The same argument
as in the proof of Proposition~\ref{proposition.fhb.l_classes}  shows that $\Lambda(R(W,S))$ is isomorphic to $P(S)$ ordered
by reverse inclusion and, moreover, that $c(\alpha)$ is the set of letters in $\alpha$ and $d(\alpha) = D_R([\alpha]_W)$ under the
identification of $\Lambda(R(W,S))$ with $P(S)$. Here $c$ and $d$ are the content and descent maps from
Section~\ref{subsection.spectrum}. The minimal ideal of $R(W,S)$ is $R(w_0)$ and the action of $S$ on the left of it is
via the operators $e_s$ described above.

\begin{proof}[Proof of Theorem~\ref{exchangethm}]
With the above arguments, the proof of most of Theorem~\ref{exchangethm} is straightforward from
Corollary~\ref{productformulacor}.  The multiplicities follow by observing that if we fix
$\alpha_K\in R(w_K)$ for each $K\subseteq S$, then $\alpha_K\cdot R(w_0)$ (in $R(W,S)$) consists of
all reduced expressions of $w_0$ the form $\alpha_K\beta$ where $\beta$ is a reduced decomposition
of the shortest element of the right coset $W_Kw_0$, which is precisely $w_K^{-1}w_0=w_Kw_0$. This proves
points (1)-(3).

To prove point (4), let $\ell(\alpha)$ denote, as usual, the length of a reduced word $\alpha\in R$.  Then $\alpha$
belongs to the minimal ideal of $R(W,S)$ if and only if $\ell(\alpha)=m$ ($
=\ell(w_0)$). If $\ell(\alpha)<m$, then there is at least
one element $s\in S$ with $\alpha s$ reduced (and hence $\ell(\alpha s)=\ell(\alpha)+1)$ because $w_0$ is the
unique element $w$ of $W$ with $D_R(w)=S$.  Thus if we run the right random walk on $R(W,S)$ driven by $P$
with initial state the empty word, then the statistic $\ell$ on $R(W,S)$ starts at $0$ and increases with probability
at least $p=\min_{s\in S} P(s)$ until it reaches the value $m$, when a constant map is obtained.

Applying Lemma~\ref{l:statisticbound} with $f(\alpha)=m-\ell(\alpha)$ as the
statistic, yields a bound on the mixing time
of $\frac{2(m+c-1)}{p}$, where we require that after $k$ steps $\|P^{*k} - \pi\|_{TV}\le e^{-c}$ with $\pi$ the
stationary distribution. When all
generators in $S$ appear with uniform probability $p=1/n$, the
mixing time is $O(mn)$.
\end{proof}

Notice that the canonical projection $\psi\colon R(W,S)\to H(W,S)$ has the property that
$\alpha\in R(W,S)$ belongs to the minimal ideal if and only if $\psi(\alpha)=\pi_{w_0}$.
As $\pi_{w_0}$ is the zero element of $H(W,S)$, it follows that the probability of obtaining
a constant map for the right random walk on $R(W,S)$ driven by $P$ is the probability
of absorption into $\pi_{w_0}$ for the right random walk on the $0$-Hecke monoid $H(W,S)$ driven by $P$.
Hence the mixing time of the exchange walk for $(W,S)$ is bounded above by the absorption
time into $\pi_{w_0}$ for the right random walk on $H(W,S)$ driven by $P$ by \eqref{stationary}.

\begin{example}[Tsetlin library]
  Consider the Tsetlin library, realized as exchange walk for the
  Coxeter system $W=(\mathbb Z/2\mathbb Z)^n$ with the standard basis
  $S$. In this case $R(W,S)$ is the free left regular band on $S$ and $H(W,S)$
   is the power set of $S$ under union. The
  random walk on $H(W,S)$ driven by $P$ is exactly the coupon
  collector chain; therefore its mixing time is $O(n\log n)$. This
  shows that the upper bound (here $O(n^2)$) given by
  Theorem~\ref{exchangethm} is not always tight. This is because the
  argument does not take advantage of the fact that, at the beginning
  of the chain, the probability of collecting a good coupon is closer
  to $1$ than to $1/n$.
\end{example}

\begin{example}[Exchange walk for the symmetric group]
Note that, when $W=\mathcal S_n$ is the symmetric group, then the elements of $H_n=H(\mathcal S_n,S)$
can identified with permutations.  The action on the right of a permutation $\sigma$ of the
generator $s_i\in S$ corresponding to the transposition $(i\ i+1)$ is to fix $\sigma$ if $\sigma(i)>\sigma(i+1)$
and otherwise to send $\sigma$ to $\sigma\circ (i\ i+1)$.  Thus the right random walk on $H_n$ driven
by the uniform distribution on $S$ is the Markov chain that has initial state the identity permutation and
at each step of the chain picks uniformly randomly a position $1\leq i\leq n-1$ of the permutation and
swaps positions $i,i+1$ if they are in order, and otherwise does nothing. This Markov chain absorbs
into the permutation in which all pairs of positions are inverted. The absorption time for this discrete time analogue of the oriented
swap process studied in~\cite{AHR09} was given in~\cite[Theorem~1.4]{BBHM06} to be $O(n^2)$ (where $p=1$
in the setting of~\cite{BBHM06}).  This then translates to an $O(n^2)$ bound on the mixing time for the exchange walk on
the symmetric group $\mathcal S_n$, which is better than $O(n^3)$ provided by our Theorem~\ref{exchangethm}.
\footnote{We thank Zachary Hamaker for pointing out the relation of our chain to \cite{AHR09,BBHM06}.}
\end{example}

\begin{theorem}
The mixing time for the exchange walk on $\mathcal S_n$ is $O(n^2)$.
\end{theorem}

\bibliographystyle{alpha}
\bibliography{domwals}
\end{document}